\documentclass[a4paper,english,11pt]{article}
\usepackage[margin=24mm]{geometry}

\usepackage{inputenc}
\usepackage{amsmath,amssymb,amsthm}
\usepackage{xcolor}
\usepackage{color}
\usepackage{babel}
\usepackage{graphicx} 
\usepackage{tikz}
\usepackage{stmaryrd}

\usepackage{mathtools}

\usepackage{physics} 
\usepackage{enumerate}
\usepackage{enumitem}
\usepackage{mathrsfs,cancel}
\usepackage[font=small]{caption}
\usepackage[font=small]{subcaption}
\allowdisplaybreaks

\newcommand{\weakstar}{\overset{*}{\rightharpoonup}}

\newtheorem{theorem}{Theorem}[section]

\newtheorem{proposition}[theorem]{Proposition}
\newtheorem{lemma}[theorem]{Lemma}

\newtheorem{definition}[theorem]{Definition}
\newtheorem{remark}[theorem]{Remark}

\newtheorem*{theorem*}{Theorem}

\allowdisplaybreaks

\numberwithin{equation}{section}

\newcommand{\supp}{\operatorname{supp}} 
\newcommand{\divv}{\operatorname{div}}

\newcommand{\rel}{\mathrm{rel}}
\newcommand{\loc}{\mathrm{loc}}
\newcommand{\Om}{\Omega}
\newcommand{\ve}{\varepsilon}
\newcommand{\sm}{\sigma}

\newcommand{\vp}{\varphi}
\newcommand{\ol}{\overline}
\newcommand{\ul}{\underline}
\newcommand{\wt}{\widetilde}
\newcommand{\wh}{\widehat}
\newcommand{\uref}{u^{\textrm{ref}}}

\newcommand{\la}{\langle}
\newcommand{\ra}{\rangle}
\newcommand{\pa}{\partial}

\newcommand{\DC}{A}
\newcommand{\RFL}{\operatorname{\mathfrak{R}}}
\newcommand{\BIL}{\vartheta}
\newcommand{\PP}{P}

\newcommand{\ee}{\mathrm{e}}

\newcommand{\intOsm}{\int_{\Omega_{\sm}}}

\newcommand{\Osm}{\Omega_{\sigma}}
\newcommand{\Osmm}{\Omega_{-\sigma}}
\newcommand{\dxdt}{\,\dd x\dd t}
\newcommand{\dxds}{\,\dd x\dd s}
\newcommand{\dSdt}{\,\dd \mathcal{H}^{d-1}\dd t}
\newcommand{\dSds}{\,\dd \mathcal{H}^{d-1}\dd s}

\newcommand{\frakP}{\mathfrak{P}}
\newcommand{\bfs}{{\bf s}}
\newcommand{\XIH}{\chi}

\makeatletter
\@namedef{subjclassname@2020}{\textup{2020} Mathematics Subject Classification}
\makeatother

\usepackage{fancyhdr}
\usepackage{pdfpages}
\usepackage{datetime2}
\newcommand{\bfu}{\boldsymbol{u}}
\newcommand{\bfv}{\boldsymbol{v}}
\newcommand{\bfU}{\boldsymbol{U}}
\renewcommand{\bra}[1]{\left(#1\right)}

\newcommand{\eps}{\varepsilon}
\newcommand{\bff}{\boldsymbol}

\title{Renormalised solutions to reaction--diffusion systems with interface conditions: 
Global existence and weak--strong uniqueness} 
\author{Katharina Hopf\thanks{Weierstrass Institute for Applied Analysis and Stochastics, Berlin, Germany} \text{ }and Bao Quoc Tang\thanks{Universität Graz, Institut für Mathematik und Wissenschaftliches Rechnen, Graz, Austria}}
    
    \usepackage[backend=biber, style=alphabetic, giveninits, maxnames=6, maxalphanames=4,url=false,isbn=false]{biblatex}
    \usepackage{csquotes}
    
    \usepackage[colorlinks, linkcolor=blue,citecolor=blue!50!cyan,urlcolor=teal]{hyperref}

\newcommand{\revcite}[1]{\citeauthor{#1}~\cite{#1}}

\providecommand{\keywords}[1]
{\small\text{\textit{Keywords and phrases: }} #1
}

\begin{document}
	
	\maketitle
	
	\begin{abstract}
	We introduce an extension of the concept of renormalised solutions for entropy-dissipating reaction--diffusion systems due to J.\ Fischer {\itshape (Arch.\;Ration.\;Mech.\;Anal.\;218, 2015)} to systems coupled by nonlinear interface conditions.
	For this notion of solution, we establish global existence as well as a weak--strong stability estimate. Our framework allows to handle entropy-dissipating interfacial transmission rates without growth restrictions, including power-law nonlinearities as arising in the thermodynamic modelling of dissipative bulk--interface systems via generalised gradient structures.
	Our analysis relies on suitable extensions of the species' densities across the interface as well as on a non-local truncated variant of the relative entropy.
	\end{abstract}
	
	\noindent\keywords{dissipative bulk--interface reaction--diffusion systems, renormalised solutions, relative entropy method, reflection technique, generalised gradient structure.}

    {\small
	\tableofcontents
}

\section{Introduction}

The purpose of this article is to develop a global-in-time solution theory for 
entropy-dissipating reaction--diffusion systems on disjoint bounded Lipschitz domains $\Om_\pm\subset\mathbb{R}^d, d\ge1,$ coupled by conservative interfacial flux conditions on a $C^1$-interface $\Gamma\subseteq\pa\Om_+\cap\pa\Om_-$ taking the form
\begin{subequations}\label{eq:sys.example}
	\begin{alignat}{3}\label{eq:bulk.example}
				\partial_t u_i^\pm - \divv(A_i^\pm\nabla  u_i^\pm)&= f_i^\pm(u^\pm),
		\hspace{3.5cm}&&	\text{ in } (0,\infty)\times\Omega_\pm,\\
			-A_i^\pm\nabla u_i^\pm \cdot \nu^\pm &=r_i^\pm(u^+,u^-):=\pm r_i(u^+,u^-)
			&&	\text{ on } (0,\infty)\times\Gamma,\\
			-A_i^\pm\nabla u_i^\pm \cdot \nu^\pm &= 0 &&\text{ on } (0,\infty)\times(\partial\Omega_{\pm}\setminus\Gamma),
			\label{eq:transmission.example}
	\end{alignat}
\end{subequations}
where $u^\pm:=(u_1^\pm,\dots,u_n^\pm)$ and $\nu^\pm$ denotes the outward unit normal to $\Om_\pm$.

\begin{subequations}\label{eq:example.nonlin}
Prototypical examples for admissible coupling functions $f_i^\pm,r_i$ are power-law nonlinearities of the form
\begin{align}\label{eq:example.f}
	f_i^\pm(u^\pm)&=	
	-k_\pm(\alpha_i{-}\beta_i) \:\bigg( \prod_{j=1}^n
	\big(\ul u_j^\pm
	\big)^{\alpha_j}  
	- \prod_{k=1}^n \big(\ul u_k^\pm
	\big)^{\beta_k}  \bigg),
	\\\label{eq:example.r}
	r_i(u^+,u^-)&=k_\Gamma(\gamma_i-\delta_i)
	\bigg( \prod_{j=1}^n
	\big(\ul u_j^+
	\big)^{\gamma_j}	\big(\ul u_j^-
	\big)^{\delta_j}    
	- \prod_{k=1}^n \big(
	\ul u_k^+
	\big)^{\delta_k} \big(
	\ul u_k^-
	\big)^{\gamma_k}  \bigg),
\end{align}
\end{subequations}
with
 $\ul u_i^\pm:={u_i^\pm}/{u_i^{\text{ref},\pm}}$ where $u_i^{\text{ref},\pm}\in(0,\infty),i=1,\dots,n$,
denote positive reference densities and $k_\pm,k_\Gamma\ge0$ are non-negative reaction resp.\ transmission coefficients.
Furthermore, $\alpha,\beta,\gamma,\delta\in\mathbb{N}_0^n$ and $\alpha=(\alpha_1,\dots,\alpha_n)$ etc. Nonlinearities such as~\eqref{eq:example.nonlin} arise in the thermodynamic modelling of reaction--diffusion processes via generalised gradient flows (cf.\ Section~\ref{ssec:GS}). The bulk reaction rates~\eqref{eq:example.f} are well known to model reversible chemical reactions according to the mass action law (see e.g.~\cite{Feinberg_2019} and the references therein).
Our analysis also covers simple linear transmission conditions
\begin{align}\label{eq:transmission.CP}\tag{\text{\ref*{eq:example.r}'}}
r_i(u^+,u^-)= k_i\big(\ul u_i^+-\ul u_i^-\big)
\end{align}
 with general species-dependent transmission coefficients $k_i>0$ as well as their nonlinear generalisations
 \begin{align}\label{eq:transmission.nonlin.ki}\tag{\text{\ref*{eq:example.r}''}}
 	r_i(u^+,u^-)= k_i(u^+,u^-)\big((\ul u_i^+)^{\gamma_i}-(\ul u_i^-)^{\gamma_i}\big),
 \end{align}
 provided that $k_i\in C^{0,1}_\loc([0,\infty)^{2n}), k_i\ge0$.
Due to the absence of a maximum principle for system~\eqref{eq:sys.example},~\eqref{eq:example.nonlin}, there are no  $L^\infty$ a priori bounds on the densities, and entropic estimates are insufficient to guarantee integrability of the interface transmission and reaction terms, which would be needed for a distributional interpretation. 
In fact, to the best of the authors' knowledge no time-global framework exists for nonlinearities such as~\eqref{eq:example.r},~\eqref{eq:transmission.nonlin.ki} without smallness restrictions on the data. Even for condition~\eqref{eq:transmission.CP} available global existence results appear to be restricted to the case that $k_i\equiv k$ are independent of $i$ (cf.~\cite{CP_2021_existence_membrane}).

\paragraph{Renormalised solutions}
In order to handle interface conditions of the form~\eqref{eq:example.r} or~\eqref{eq:transmission.nonlin.ki}, our strategy is to adapt the concept of renormalised solutions introduced by  \revcite{Fischer_2015} for bulk reaction--diffusion systems,
which is based on truncation functions acting simultaneously on all densities. Renormalised solutions were first introduced by DiPerna \& Lions for the transport and Boltzmann equation~\cite{DiPL_1989_transport,DiPL_1989_Boltzmann}. Since then, this concept has found various applications in PDEs, including elliptic and parabolic equations.
The key point in Fischer's extension is to consider truncation functions $\xi=\xi(u)$ that act on the vector of densities $u=(u_1,\dots,u_n)$.
To explain the idea in the bulk case let us suppose that $u:=u^+$, where $u^+_i\ge0$ for all $i=1,\dots,n$, is a sufficiently regular solution to~\eqref{eq:sys.example} in $(0,\infty)\times\ol\Om_+$ with $A_i:=A_i^+$,
$f_i:=f_{i}^+$, and
 $r_i\equiv0$. Take $\xi\in C^\infty([0,\infty)^n)$ with compactly supported derivative $D\xi\in C^\infty_c([0,\infty)^n)^n$. Then, by the chain rule, the composition $\xi\circ u$ satisfies
 \begin{align}\label{eq:renorm.symb}
 \frac{\dd}{\dd t}\xi(u)=\divv\left(\sum_{i=1}^n D_i\xi(u)A_i\nabla u_i\right)-\sum_{i,j=1}^n D_{ij}\xi(u)\nabla u_j\cdot A_i\nabla u_i+\sum_{i=1}^nD_i\xi(u)f_i(u),
 \end{align}
 where the second term on the right-hand side arises from a commutator.
 The control that follows from entropy dissipation ensures that all terms on the right-hand side are well-defined integrable functions. Requiring identity~\eqref{eq:renorm.symb} to hold  in the distributional sense (in duality with $C^\infty_c([0,\infty)\times\ol\Om_+)$) for all such $\xi$
is sufficient to ensure 
equivalence to the weak formulation in case of integrable reaction rates (cf.~\cite[Theorem~2 \& Lemma 4]{Fischer_2015}). Fischer's renormalised solution concept further enjoys 
a weak--strong uniqueness principle, see~\cite{Fischer_2017}.
Since Fischer's seminal work, the idea of renormalisation in multi-component models of reaction--diffusion type was extended to Shigesawa--Kawasaki--Teramoto
cross-diffusion--reaction systems and energy--reaction--diffusion systems, see~\cite{CJ_existence_2019,CJ_wkstr-unique_2019,DT_2019} and~\cite{FHKM_2022,Hopf_2022}
for global existence, weak--strong uniqueness and exponential equilibration. For existence and long-time behaviour in the case of porous-medium diffusion, see~\cite{FFKT_2023}. 
All these models are pure bulk equations without boundary or interfacial coupling.

\paragraph{Objectives and strategy}
 In the present work, we aim to perform the first step towards an extension of renormalised solutions to reaction--diffusion systems interacting with lower-dimensional spatial structures.
To get a first indication of the new difficulties arising in renormalised formulations for bulk--interface reaction--diffusion systems, pick $\sigma\in\{\pm\}$
and consider the equation for $u^\sigma$ in~\eqref{eq:sys.example}.
A naive direct application of the renormalisation concept in~\cite{Fischer_2015}
would amount to the use of truncation functions $\xi=\xi(u^\sigma)$ for the equation in $\Om_\sigma$.
However, following~\eqref{eq:renorm.symb}, this would lead to interface terms of the form $D_i\xi(u^\sigma)r_i(u^+,u^-)$, whose integrability is not guaranteed by the entropy dissipation inequality for nonlinearities $r_i$ such as~\eqref{eq:example.r},~\eqref{eq:transmission.nonlin.ki}. 
The problem is that $D_i\xi(u^\sigma)$ does not involve simultaneous truncation in $u^+$ and $u^-$.
We therefore need to adjust the renormalisation concept.
Our strategy is to introduce suitable local continuations $\tilde u^{\sm}$ extending the densities $u^{-\sm}$ defined on the neighbouring bulk compartment $\Om_{-\sm}$ across $\Gamma$ to the domain of interest $\Om_\sm$ and
formulate the notion in terms of truncation functions $\xi=\xi(u^\sm,\tilde u^{\sm})$. At interior points of the interface $\Gamma$, which is assumed to be a compact $(d-1)$-dimensional $C^1$ submanifold of $\mathbb{R}^d$ with or without boundary, this can be achieved by the classical reflection method for Lipschitz boundaries, whereas at boundary points $p\in\pa\Gamma$ involving triple junctions (cf.\ Figure~\eqref{fig:junction}) we use a two-step procedure that requires a somewhat more stringent assumption on the geometry, see hypothesis~\ref{hp:C1ext}. In the special case of a flat interface $\Gamma$ with $\Om:=\Om_+\cup\Om_-$ symmetric about a hyperplane $\wt\Gamma$ containing $\Gamma$, the extension to the neighbouring compartment can be defined globally and is simply given by reflection at $\wt\Gamma$.

\begin{figure*}[h!]
	\centering
	\begin{subfigure}[t]{0.5\textwidth}
		\centering
		\includegraphics[scale=.7]{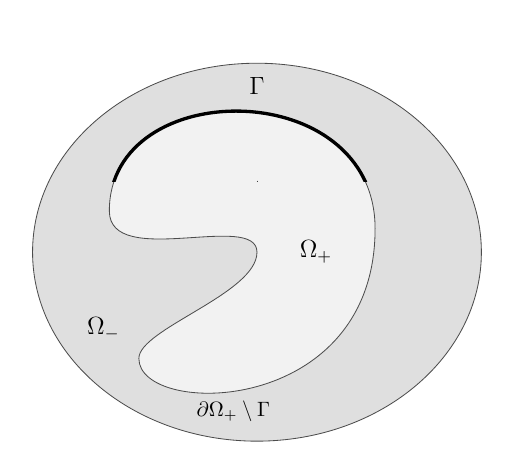}
	\caption{Smooth geometry.}
	\label{fig:annular}
\end{subfigure}%
~ 
\begin{subfigure}[t]{0.5\textwidth}
	\centering
	\includegraphics[scale=.7]{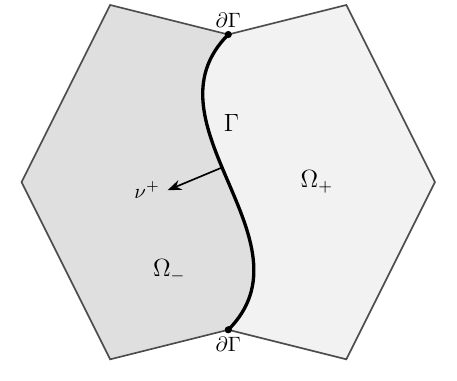}
	\caption{Triple junction structure at $\pa\Gamma$.}
	\label{fig:junction}
\end{subfigure}
\caption{Examples of admissible bulk-interface geometries.}
\end{figure*}

\label{p:strategy.0} To motivate the properties that we want our extension $\tilde u$ to satisfy, let us formally derive the equations appearing in the renormalised formulation (cf.~\eqref{eq:def.renorm.c.gen} for the precise equations) based on truncation functions
$\xi=\xi(s,\tilde s)$ for the compartment $\Om_\sigma\cup\Gamma$, where $\sigma\in\{\pm\}$ is fixed.  
To this end, we consider, for simplicity, a reaction--diffusion equation for a scalar quantity $u:(0,\infty)\times(\Om_+\dot\cup\Om_-)\to[0,\infty)$, $u(t,x)=u^\pm(t,x)$ for $x\in\Om_\pm$, with  flux ${\sf J}^\pm=-a^\pm\nabla u^\pm$, $a^\pm>0$, reaction rate $f\in C([0,\infty))$, and interface transmission rates $r^\pm\in C([0,\infty)^2)$,
\begin{alignat}{3}
	&\pa_tu+\divv({\sf J}^\pm)=f(u)\qquad&&\text{for }t>0,x\in\Om_+\dot\cup\Om_-,
\label{eq:conteq.strong}
	\\&u_{|t=0}=u_0\quad&&\text{for }x\in\Om_+\dot\cup\Om_-\nonumber
\end{alignat}
subject to the flux conditions 
\begin{alignat*}{3}
{\sf{J}}^\pm\cdot\nu^\pm&=r^\pm(\tr_\Gamma u^+,\tr_\Gamma u^-)\qquad&&\text{for }t>0,z\in\Gamma,
\\{\sf{J}}^\pm\cdot\nu^\pm&=0\qquad&&\text{for }t>0,z\in\pa\Om_\pm\setminus \Gamma.
\end{alignat*}
Given a point $\beta\in \Gamma$ and a small open neighbourhood $V=V_\beta$ of $\beta$ that admits a measure-preserving bi-Lipschitz homeomorphism $\Phi:\Om_\sm\cap V\to\Om_{-\sm}\cap V$ (cf.\ Section~\ref{ssec:ext} for details), which thus uniquely extends to $\ol\Om_\sm\cap V$, with the crucial property of keeping $\Gamma\cap V$ fixed in the pointwise sense, we abbreviate $\tilde u:=\tilde u^\beta:=u\circ\Phi$ (where $\circ$ is with respect to space variable), 
and compute
\begin{align*}
	\frac{\dd}{\dd t}\xi(u,\tilde u)=D_1\xi(u,\tilde u)\pa_tu+D_2\xi(u,\tilde u)\pa_t\tilde u \qquad\text{in }\mathcal{D}'([0,\infty)\times(\ol\Om_\sm\cap V)).
\end{align*}
For the first term on the right-hand side, the renormalised formulation is mostly classical~\cite{Fischer_2015}, except for the dependence of $D_1\xi$ on $\tilde u$ required for the interface condition.
To handle the second term on the right-hand side, it is convenient to change variables $y=\Phi^{-1}(x)$ and then use the equation for $u=\tilde u\circ\Phi^{-1}$ in the opposite domain $\Om_{-\sm}$.\footnote{Alternatively, one could have used the fact that $\tilde u$ satisfies, for $x\in \Om_\sm\cap V$, the transformed equation $\pa_t\tilde u+\divv^\Phi(\wt{\sf J})=f(\tilde u)$, where  $\divv^\Phi H:=\divv(D\Phi^{-1}_{|\Phi(\cdot)}H)$ and $\wt{\sf J}={\sf{J}}\circ\Phi$.}
We thus arrive at the following 
integral formulation of~\eqref{eq:conteq.strong} for any $\psi\in C^\infty_c([0,\infty)\times V)$ and (almost) all $T\in(0,\infty)$
\begin{align*}
&\int_{\Om_\sigma\cap V}\psi(T,\cdot)\,\xi(u(T,\cdot),\tilde u(T,\cdot))\,\dd x-\int_{\Om_\sigma\cap V}\psi(0,\cdot)\,\xi(u_0,\tilde u_0)\,\dd x -\int_0^T\!\!\int_{\Om_\sigma\cap V}\pa_t\psi\, \xi(u,\tilde u)\,\dd x\dd t
\\&=\int_0^T\!\!\int_{\Om_\sigma\cap V}\nabla(\psi D_1\xi(u,\tilde u))\cdot{\sf J}^\sm	\,\dd x\dd t
-\int_0^T\!\!\int_{\Gamma_\sigma\cap V}\psi\,D_1\xi(u,\tilde u)r^\sigma(u^+,u^-)\,\dd\mathcal{H}^{d-1}\dd t
\\&
\qquad+\int_0^T\!\!\int_{\Om_\sigma\cap V} \psi D_1\xi(u,\tilde u)f(u) \dxdt 
+\int_0^T\!\!\int_{\Phi(\Om_\sm\cap V)}(\psi D_2\xi(u,\tilde u))\circ\Phi^{-1}\pa_tu\,|\det D\Phi^{-1}|\dd x\dd t.
\end{align*}
Integrals analogous to the first and third terms on the right-hand side also appear in renormalised formulations for pure bulk problems.
Concerning the second term on the right-hand side involving the surface integral, we observe that it is bounded because of the compact support of $D_1\xi$ and since 
\begin{align}\label{eq:trace.tilde}\begin{aligned}
	&\tr_{\Gamma_\sm\cap V}u=\tr_{\Gamma\cap V} u^\sm,
	\\&\tr_{\Gamma_\sm\cap V}\tilde u=\tr_{\Gamma \cap V}u^{-\sm}.
	\end{aligned}
\end{align}
For this last property, it is indispensable that $\Phi$ fixes $\Gamma$ pointwise.
It remains to control the last term on the right-hand side. Since renormalised solutions do not possess a distributional time derivative in the classical sense, our strategy is to rewrite this term using a renormalised version of the equation~\eqref{eq:conteq.strong} for $u^{-\sm}$. To arrive at such an appropriate expression, we insert $\pa_tu^{-\sm}=-\divv{\sf J^{-\sm}}$ in a first step.
Next, since the flux ${\sf J^\pm}$ is not guaranteed to have Sobolev regularity for renormalised solutions, we rely on integration by parts in the divergence term, which requires 
$|\det D\Phi^{-1}|$ to be at least Lipschitz continuous. Adapting ideas of the classical reflection technique at Lipschitz boundaries, we will be able to ensure that $\Phi$ is even measure preserving, so that $|\det D\Phi^{-1}|\equiv 1$. Another point that is essential when using integration by parts is to avoid uncontrolled boundary traces, which can be ensured provided that
$\Phi(\pa\Om_\sm\cap V)=\pa\Om_{-\sm}\cap V$, $\Phi(\Om_\sm\cap \pa V)=\Om_{-\sm}\cap \pa V$.
Assuming that $\Phi$ satisfies all these properties, we can rewrite the last term on the right-hand side of the above equation as 
\begin{multline*}
\int_0^T\!\!\int_{\Om_{-\sm}\cap V}\nabla(\psi\circ\Phi^{-1} D_2\xi(u,\tilde u)\circ\Phi^{-1})\cdot{\sf J}^{-\sm}\,\dd x\dd t
-\int_0^T\!\!\int_{\pa\Om_{-\sm}\cap V}\psi\circ\Phi^{-1} D_2\xi(u,\tilde u)\circ\Phi^{-1}{\sf J}^{-\sm}\cdot \nu^{-\sm}\,\dd \mathcal{H}^{d-1}\dd t
\\+\int_0^T\!\!\int_{\Om_{-\sm}\cap V}\psi\circ\Phi^{-1} D_2\xi(u,\tilde u)\circ\Phi^{-1}f(u)\,\dd x\dd t,
\end{multline*}
where for the boundary integral we used the fact that $\psi(t,\Phi^{-1}(\cdot))$ vanishes on $\pa V\cap\Om_{-\sm}$ for all $t$ as a consequence of $\supp\psi(t,\cdot)\Subset V$ and $\Phi^{-1}(\Om_{-\sm}\cap \pa V)=\Om_\sm\cap \pa V$. Note also that the integrand involving the reaction term in the last line is bounded due to $D_2\xi(u,\tilde u)\circ\Phi^{-1}=D_2\xi(u \circ\Phi^{-1},u)$ and $\supp D_2\xi\Subset[0,\infty)^{2}$.
Finally, to deal with the second integral, we split it into two parts where spatial integration is restricted to $\Gamma\cap V$ and to $(\pa\Om_{-\sm}\setminus\Gamma)\cap V$, respectively. 
For the latter, we can use the no-flux boundary conditions on $\pa\Om_{-\sm}\setminus\Gamma$. For the former, we use the interface flux conditions and~\eqref{eq:trace.tilde}, so that the entire term reduces to
\begin{align*}
-\int_0^T\!\!\int_{\Gamma_{-\sm}\cap V}D_2\xi(\tr u^\sm,\tr u^{-\sm})r^{-\sm}(\tr u^+,\tr u^-)\,\dd \mathcal{H}^{d-1}\dd t, \qquad \text{(with $\tr:=\tr_\Gamma$)},
\end{align*}
whose integrand is bounded, since $D_2\xi$ has compact support and $|(\tr u^\sm,\tr u^{-\sm})|_1=|(\tr u^+,\tr u^-)|_1$  on $\Gamma$.
\label{p:strategy.1}

For the construction of a generalised extension map $\Phi$ with the properties above, see Section~\ref{ssec:ext}. 

\paragraph{State-of-the art}

Most of the literature on global existence results for bulk--interface or bulk--surface reaction--diffusion systems is either based on uniform a priori bounds for the densities that rely on specific assumptions on the nonlinearities or diffusion coefficients or on growth restrictions on the nonlinearities involved. In fact, for bulk reaction--diffusion systems with superquadratic reactions rates, it is known that the property of mass dissipation or mass control does not prevent blow-up of the densities in $L^\infty$~\cite{PS_2023_blow-up,PS_2000}. It remains to be seen whether examples of blow-up solutions can also be constructed under the somewhat stronger structural hypothesis of entropy dissipation. 
Concerning the bulk--interface systems proposed by~\revcite{Mielke_2013_thermomechanical}, Disser~\cite{Disser_2015,Disser_2020_bulk-interface}
is one of the first to provide some rigorous underpinning. 
She focusses on bulk--interface interactions in \textit{scalar} reaction--diffusion equations and uses maximal $L^p$ regularity theory to establish a well-posedness theory assuming only Lipschitz regularity on the domains and $C^1$ regularity on the interface. Her global well-posedness result in~\cite{Disser_2020_bulk-interface} relies on uniform $L^\infty$ bounds that follow by comparison principle arguments, which crucially exploit the fact that the reaction--diffusion equations are scalar.
For a certain class of bulk--surface reaction--diffusion systems in a smooth geometry, \revcite{AB_2024} obtain
global well-posedness via maximal $L^p$ regularity and uniform a priori bounds on the densities. The latter rely on a  triangular structure of the reaction network and additional growth assumptions in the nonlinearities.
Notably, the treatment of the nonlinear normal flux conditions in this work is facilitated by the linear growth bounds on the ad-/desorption rates to and from the surface. 
\revcite{MT_2023} establish the global well-posedness of classical solutions for regular initial data for mass-controlled bulk--surface reaction--diffusion systems assuming that the reaction networks have a triangular structure, which is the key ingredient to deduce uniform spatial bounds. The global well-posedness result applies to a rather general class of nonlinearities with polynomial growth and is obtained from a priori bounds for a family of auxiliary energy type functionals in combination with duality methods.
\revcite{CP_2021_existence_membrane} consider a class of reaction--diffusion systems of the form~\eqref{eq:sys.example} 
with the linear interface conditions~\eqref{eq:transmission.CP} and mass-controlling reactions $f_i^\pm$ of at most quadratic growth. Under the assumption $k_1=\dots=k_n$, the authors establish the global existence of weak solutions by adapting the $L^1$-theory for bulk reaction--diffusion systems (cf.~\cite{Pierre_2010_survey} and references therein). The restriction $k_1=\dots=k_n$ appearing in the main results of~\cite{CP_2021_existence_membrane} is needed to ensure applicability of a duality method. The latter provides
$L^2_{t,x}$-integrability of all densities, which implies integrability of the (subquadratic) reaction rates. An additional $W^{1,1}$-regularity property on each bulk compartment ensures well-defined traces at the interface.
There are various further works that establish well-posedness for bulk--interface reaction--diffusion systems arising in specific applications such as~\cite{HR_2018,BR_2022}, which rely on the explicit form of the nonlinearities. For further details and a discussion of several applications, we refer to the survey in~\cite[Sections~1 \& 2]{AB_2024}.

\medskip
In this paper, by developing the reflection technique outlined above, we establish a first framework of global existence for reaction--diffusion systems with bulk--interface interactions without any restriction on the growth rates of the bulk reactions $f_i^{\pm}$ and interfacial transmission functions $r_i^{\pm}$. Crucially,  within in this framework, we are able to establish a weak--strong uniqueness principle, which implies that any renormalised solution must coincide with the strong solution whenever the latter exist.

\paragraph{Outline of the paper}
In Section~\ref{ssec:GS}, we show how reaction--diffusion systems with interface conditions such as~\eqref{eq:sys.example},~\eqref{eq:example.f} with~\eqref{eq:example.r},~\eqref{eq:transmission.CP}, or~\eqref{eq:transmission.nonlin.ki} arise in the thermodynamic modelling via generalised gradient flows.
In Section~\ref{sec:main.results}, we state the precise model hypotheses, introduce the local extension technique across the interface, specify the concept of renormalised solutions taking into account the interface condition, and formulate our main results.
The subsequent part, Sections~\ref{sec:weakstruniq}--\ref{sec:existence}, is devoted to the proofs of the main results.

\subsection{A gradient structure for reaction--diffusion systems with interface conditions}\label{ssec:GS}
The modelling of entropy-dissipating bulk--interface reaction--diffusion systems via gradient structures goes back to the works~\cite{Mielke_2011,Mielke_2013_thermomechanical,GM_2013}.
The main ingredients in this formalism are a driving functional $\mathcal{H}=\mathcal{H}(\bfu)$, hereafter loosely corresponding to the negative entropy of the physical system, and a dual dissipation potential $\mathcal{R}^*=\mathcal{R}^*(\bfu;\boldsymbol{\Xi})$ leading to 
the abstract gradient-flow equation 
\begin{align}\label{eq:gfe}
	\dot\bfu\in \pa_{\boldsymbol{\Xi}}\mathcal{R}^*(\bfu;-D\mathcal{H}(\bfu)),
\end{align}
where $D\mathcal{H}$ denotes the formal differential or variational derivative of $\mathcal{H}$.
At this stage, the structures $\mathcal{H},\mathcal{R}^*$ and the inclusion~\eqref{eq:gfe} are purely formal in the sense that no function spaces have been specified.
While the works mentioned above are based on Onsager's formalism involving quadratic dual dissipation potentials $\boldsymbol{\Xi}\mapsto\mathcal{R}^*(\bfu;\boldsymbol{\Xi})$, in the modelling of chemical reaction processes non-quadratic potentials $\mathcal{R}^*(\bfu;\Xi)$ were suggested due to their link to microscopic descriptions via large deviation theory~\cite{ADPZ_2011,ADPZ_2013,MPR_2014,LMPR_2017,MPPR_2017,PS_2023}.
Let us also mention the recent work~\cite{CCCE_2024_moving-interface} where the use of {\em generalised} gradient flows based on cosh structures was discussed to model linear interface conditions of Butler--Volmer type driven by a  Boltzmann entropic functional.

Throughout this manuscript, we assume that $\mathcal{H}$ is a relative Boltzmann entropy
\begin{subequations}
\begin{align}\label{eq:entr.int}
	\mathcal{H}(\bfu)=\sum_{\sigma\in\{\pm\}}\int_{\Omega_\sigma}h^\sigma(u^\sigma)\,\dd x,
\end{align}
where 
\begin{align}\label{eq:entr.fn}
	h^\sigma(u):=\sum_{i=1}^nu_i^{{\rm ref},\sigma}
	\mathfrak{B}(u_i/u^{{\rm ref},\sigma}_{i}),\qquad \mathfrak{B}(r):=r\log r-r,\quad  u^{{\rm ref},\sigma}_{i}>0.
\end{align}
\end{subequations}
To capture~\eqref{eq:sys.example}, we  adapt~\cite[Section~4]{Mielke_2013_thermomechanical} and enhance classical gradient structures for bulk reaction--diffusion systems by an additional dissipation mechanism inducing non-trivial interfacial transmission. To this end, we consider dual dissipation potentials of the form
\begin{equation}\begin{split}
	\mathcal{R}^*(\bfu;\boldsymbol{\Xi})&:=\sum_{\sigma\in\{\pm\}}\mathcal{R}^*_{\sigma}(\bfu;\boldsymbol{\Xi})+\mathcal{R}^*_\Gamma(\bfu;\boldsymbol{\Xi})
	\\[1mm]&:=\sum_{\sigma\in\{\pm\}}\big(\mathcal{R}^*_{\sigma,\textrm{diff}}(u^\sigma;\Xi^\sigma)+\mathcal{R}^*_{\sigma,\textrm{react}}(u^\sigma;\Xi^\sigma)\big)+\mathcal{R}^*_\Gamma(\bfu;\boldsymbol{\Xi}),
\end{split}\end{equation}
where $\bfu=(u^+,u^-)$, $\boldsymbol{\Xi}=(\Xi^+,\Xi^-)$
and 
\begin{subequations}\label{eq:dissp.ex.bulk}
	\begin{align}\label{eq:dissp.ex.bulk.diff}
		\mathcal{R}^*_{\sigma,\textrm{diff}}(u;\Xi)=\frac12\int_{\Om_\sigma}\sum_{i=1}^n u_iA_i^\sigma\nabla\Xi_i:\nabla\Xi_i\,\dd x.
	\end{align}
As pointed out above, 
to model entropy-dissipating reactions, cosh structures are often favoured over quadratic dissipation potential due to their link to fluctuation theory. The mass--action reactions~\eqref{eq:example.f} are captured in this approach by
\begin{align}\label{eq:dissp.ex.bulk.react}
	\mathcal{R}^*_{\sigma,\textrm{react}}(u;\Xi)
	= \int_{\Om_\sigma} \kappa_\pm(u)
	\mathfrak{C}((\alpha{-}\beta)\cdot\Xi)
	\,\dd x,
\end{align}\end{subequations}
where $\kappa_\pm(u)=k_\pm\prod_{i=1}^n\big(\frac{u_i}{\uref_i}\big)^{\alpha_i/2}\big(\frac{u_i}{\uref_i}\big)^{\beta_i/2}$
and
\begin{align*}
\mathfrak{C}(r):=4(\cosh(r/2)-1),\;\text{ so that }\mathfrak{C}'(r)=2\sinh(r/2),\;\text{  }\;\mathfrak{C}'(\log r)=\sqrt{r}-1/\sqrt{r}.
\end{align*}

\paragraph{Example for interface conditions}
Linear and nonlinear interface conditions dissipating the entropy functional~\eqref{eq:entr.int}
 naturally arise in the PDE modelling via generalised gradient flows:
\begin{enumerate}
	\item \textsf{Species-dependent linear interface conditions.}\\
	The linear transmission condition~\eqref{eq:transmission.CP} can be modelled by
	\begin{align*}
		\mathcal{R}^*_\Gamma(\bfu;\boldsymbol{\Xi})=\int_\Gamma\sum_{i=1}^nk_i\sqrt{\tfrac{u_i^+u_i^-}{u_i^{\textrm{ref},+}u_i^{\textrm{ref},-}}}
		\,\mathfrak{C}(\llbracket \boldsymbol{\Xi}_i\rrbracket)\,\dd\mathcal{H}^{d-1},
	\end{align*}
	where $\llbracket \boldsymbol{\Xi}\rrbracket:=\Xi^+-\Xi^-$.
	\item\label{it:example.polynomial} \textsf{Polynomial interface condition.}\\
	A generalised gradient system for  the bulk--interface reaction--diffusion equations~\eqref{eq:sys.example},~\eqref{eq:example.nonlin} driven by  the Boltzmann entropy $\mathcal{H}$ 
	 is obtained by taking $\mathcal{R}^*_{\pm}(\bfu;\boldsymbol{\Xi})$ of the form~\eqref{eq:dissp.ex.bulk} and choosing $\mathcal{R}^*_\Gamma$ as
	\begin{align*}
		\mathcal{R}^*_\Gamma(\bfu;\boldsymbol{\Xi})=\int_\Gamma
		\kappa_\Gamma(\bfu)\,\mathfrak{C}((\gamma{-}\delta){\,\cdot\,}\llbracket \boldsymbol{\Xi}\rrbracket)
		\,\dd\mathcal{H}^{d-1},
	\end{align*}
	where $\kappa_\Gamma(\bfu)=
	k_\Gamma\prod_{i=1}^n\Big(\frac{u_i^+u_i^-}{u^{\textrm{ref},+}_i u^{\textrm{ref},-}_i}\Big)^{(\gamma_i+\delta_i)/2}$, $k_\Gamma>0$. 
\item \textsf{Generalisation.}\\ 
The two examples above are special cases of
\begin{align*}
	\mathcal{R}^*_\Gamma(\bfu;\boldsymbol{\Xi})=\int_\Gamma \sum_{\ell=1}^N\kappa^{(\ell)}(\bfu)
	 \mathfrak{C}(\lambda^{(\ell)}{\cdot}\llbracket\boldsymbol{\Xi}\rrbracket)
	\,\dd\mathcal{H}^{d-1},
\end{align*}
where $N\in \mathbb{N}$, $\lambda^{(\ell)}\in \mathbb{N}_0^n$ and $\kappa^{(\ell)}(\bfu)=k_\Gamma^{(\ell)}(\bfu)\prod_{i=1}^n\Big(\frac{u_i^+u_i^-}{u^{\textrm{ref},+}_i u^{\textrm{ref},-}_i}\Big)^{\lambda^{(\ell)}_i/2}$ for $\ell=1,\dots,N$, where $k_\Gamma^{(\ell)} \in C([0,\infty)^{2n}),$ with $k_\Gamma^{(\ell)}\ge0$. (To recover Example~\ref{it:example.polynomial}, choose $N=1$, $\lambda^{(1)}_i=|\gamma_i-\delta_i|$, and $k_\Gamma^{(1)}(\bfu)=k_\Gamma\prod_{j=1}^n(u_j^+u_j^-)^{\min\{\gamma_j,\delta_j\}}$.)
\end{enumerate}
A natural question in this context, which is beyond the scope of this manuscript, is that of a physical foundation of the interface conditions above.
Let us, however, emphasise that, while a significant class of examples is generated by generalised gradient structures, our analysis applies more generally to entropy-dissipating systems, see Section~\ref{sec:main.results} for the precise assumptions.

\section{Main results}\label{sec:main.results}

\subsection{Hypotheses and general model equations}

\paragraph{Assumptions on the geometry}
\begin{enumerate}[label=(G\arabic*)]
	\item\label{hp:geo.basic} $\Om_\pm\subset\mathbb{R}^d$ are disjoint bounded Lipschitz domains (in the sense of~\cite[Definition 1.2.1.1]{Grisvard_1985}) and $\Gamma\subseteq\pa\Om_+\cap\pa\Om_-$.
	\item\label{hp:C1ext}  $\Gamma$ is a compact $(d-1)$-dimensional embedded $C^1$-submanifold of $\mathbb{R}^d$ with or without boundary,
	and there exists a $(d-1)$-dimensional embedded $C^1$-submanifold of $\mathbb{R}^d$ without boundary,  denoted $\wt\Gamma$, which contains $\Gamma$,
	partitions $\mathbb{R}^d$ into two disjoint connected components $G_-\dot\cup G_+=\mathbb{R}^d\setminus\wt\Gamma$, and separates $\Om_\pm$ in the sense that $\Om_\pm\subseteq G_\pm$.
\end{enumerate}
As we will detail below, hypothesis~\ref{hp:C1ext} allows to define a local extension of the densities on $\Om_\pm$ across the interface that has a well-defined trace on $\Gamma$ and preserves the outer no-flux boundary condition. With such an extension at hand, we can introduce a sufficiently stringent generalised solution concept that not only guarantees global existence but also enjoys a weak--strong stability estimate. 

\paragraph{Assumptions on coefficients functions and entropic structure}
For simplicity, we consider the same logarithmic entropy function in each compartment with reference densities $u_i^{\sm,\rm{ref}}\equiv1$ for all $i=1,\dots,n,\sm\in\{\pm\}$, i.e.\ we let $h^\pm(u):=h(u)$ where
\begin{align}\label{eq:entropy.function}
	h(u):=\sum_{i=1}^n\mathfrak{B}(u_i),\quad\qquad \mathfrak{B}(r):=r\log r-r.
	\end{align}
In line with this, we suppose that the reaction rates do not depend on $\pm$, so that $f_i^\pm=f_i$. This is merely to simplify notation, and our analysis can easily be extended to the case that the constant values $u_i^{\sm,\textrm{ref}}>0$ depend on $i,\sm$.
\begin{enumerate}[label=(C\arabic*)]
	\item\label{it:HP.diff} {\sf Diffusion coefficients:} $A_i^\pm\in L^\infty(\Om_\pm,\mathbb{R}_{\textrm{sym}}^{d\times d})$ and there exists $\mathfrak{a}>0$ such that 
	\begin{align*}
	\zeta\cdot A_i^\pm(x)\zeta\ge\mathfrak{a}|\zeta|^2\qquad \text{for all }\zeta\in \mathbb{R}^d\text{ and $\mathcal{L}^d$-almost all }x\in\Om_\pm
	\end{align*}
	for every $i=1,\dots,n$.
	\item \label{it:HP.react} {\sf Reaction rates:} 
$f_i\in C^{0,1}_\loc([0,\infty)^n)$ 
	 are \textit{quasi-positive}, i.e.\ for all $u\in[0,\infty)^n$
	\begin{align}
	u_i=0\implies  f_i(u)\ge0,
	\end{align}
	and \textit{entropy dissipating}, i.e.\ for all $u\in[0,\infty)^n$
	\begin{align}\label{eq:f.diss}
	\sum_{i=1}^n	f_i(u)D_ih(u)\le0,
	\end{align}
	where $h$ is as in~\eqref{eq:entropy.function}.
	\item\label{it:HP.int} {\sf Interface condition:}
	$r_i^\sigma\in C^{0,1}_\loc([0,\infty)^{2n})$ 
	are \textit{quasi-positive}, i.e.\ 
	for all $\bfu=(u^+,u^-)\in [0,\infty)^{2n}$
	\begin{align}
	u_i^\sigma=0\implies r_i^\sigma(\bfu)\le0,\qquad
	\end{align}
	 \textit{entropy dissipating}, i.e.\ for all $\bfu=(u^+,u^-)\in [0,\infty)^{2n}$
	\begin{align}
	\label{eq:R.diss}
	\sum_{\sigma\in\{\pm\}}\sum_{i=1}^n r_i^\sigma(\bfu)D_ih(u^\sigma)\ge0,
	\end{align}
	and \textit{mass preserving}, i.e.\ for all $\bfu=(u^+,u^-)\in [0,\infty)^{2n}$
	\begin{align}\label{eq:mass.pres}
		\sum_{\sigma\in\{\pm\}}r^\sigma_i(\bfu)=0.
	\end{align}
\end{enumerate} 
\begin{remark}
	\begin{enumerate}[label=(\roman*)]
		\item 	The reaction rates~\eqref{eq:example.f} fulfil~\ref{it:HP.react}, and each of 
		the interface conditions~\eqref{eq:example.r},~\eqref{eq:transmission.CP}, and~\eqref{eq:transmission.nonlin.ki} satisfies~\ref{it:HP.int}.
		\item The mass conservation property of the interface conditions in~\eqref{eq:mass.pres} is not essential for the analysis in this manuscript. 
		However, for pure transmission of species through a passive interface, it appears to be a natural property from the modelling viewpoint, cf.~\cite{CP_2021_existence_membrane}.
        	A more detailed study of the qualitative behaviour including  conservation laws and long-time asymptotics for specific classes of bulk-interface systems is deferred to future work.
	\item In the global existence result below (cf.\ Theorem~\ref{thm:ex}), we expect that the local Lipschitz regularity of $f_i^\pm$ and $r_i^\pm$ can be relaxed to continuity by using entropic variables, see e.g.~\cite{CJ_existence_2019}.
	\item We expect that our theory can be extended to a broader class of entropies obtained from a power-law ansatz for the entropy function generalising the Boltzmann logarithmic  choice~\eqref{eq:entropy.function}  (cf.~\cite[Section 2.3.2]{Hopf_2022}).
	\end{enumerate}
\end{remark}

\paragraph{Notations and model equations}
Below, it will be convenient to abbreviate $\Om:=\Om_-\dot\cup\Om_+$. We caution that, with this convention, $\Gamma$ is not contained in $\Om$ and that 
the symbol $\Om$ merely serves notational convenience.

We define
\[	A_i(x)=A_{i}^\pm(x)\quad \text{for } \mathcal{L}^d\text{-a.e.\ }x\in\Om_{\pm}.\]
Similarly, given $u^\pm:(0,T)\times\Om_\pm\to[0,\infty)^n$, we abbreviate for $i=1,\dots,n$
\begin{align}\label{eq:def.ui}
u_i(t,x)=\begin{cases}
	u_i^+(t,x),& x\in\Om_+\\
	u_i^-(t,x),& x\in\Om_-.
\end{cases}
\end{align}
We usually do not explicitly indicate the dependence of $A_i$ on the spatial variable, 
 and simply write $A_i:=A_i(x)$. Finally, throughout this manuscript, unless specified otherwise, we abbreviate $\bfu=(u^+,u^-)$ with $u^\pm=(u_1^\pm,\dots,u_n^\pm)$.

The strong formulation of the reaction--diffusion systems with interface conditions considered in this manuscript then reads as
\begin{equation}\label{eq:sys.sigma}
	\begin{cases}
		\partial_t u_i - \divv(A_i\nabla  u_i) = f_i(u) &\text{ in } (0,\infty)\times\Omega,\\
		-A_i^\pm\nabla u_i^\pm \cdot \nu^\pm =  r_i^\pm(\bfu) &\text{ on } (0,\infty)\times\Gamma,\\
		-A_i^\pm\nabla u_i^\pm \cdot \nu^\pm = 0 &\text{ on } (0,\infty)\times\partial\Omega_\pm\setminus\Gamma,\\
	\end{cases}
\end{equation}
assuming sufficient regularity on the data (including $A_i^\pm$) and the unknown $u$.

Throughout this manuscript, the term $r^\sm_i(\bfu)$ for $\bfu=(u^+,u^-)$ is to be understood in the sense $r^\sm_i(\tr_\Gamma u^+,\tr_\Gamma u^-)$. Observe that here $u^+$ lives on $\Om_+$, while $u^-$ lives on $\Om_-$, so that the traces need to be takes from opposite sides of $\Gamma$.
Furthermore, the terms ``$\Gamma_{\sm}$-trace''  or ``$\Gamma$-trace from $\Om_\sm$" will be used as a short-hand notation for the $\Gamma$-trace of the restriction to  $\Om_{\sm}$ of a given function on $\Om$ (used for $\sm\in\{\pm\}$). Integrals $\int_{\Gamma_\sm}\dots\dd\mathcal{H}^{d-1}$ are to be understood in a similar sense.

Finally, for sufficiently smooth functions $g=g(u_1,\dots,u_m)$, $m\in\mathbb{N}$, we let $D_ig=\frac{\pa g}{\pa u_i}$ and $D_{ij}g=D_iD_jg=\frac{\pa^2 g}{\pa u_j\pa u_i}$, $i,j=1,\dots,m$.

\subsection{Geometric preliminaries}\label{ssec:ext}

The purpose of this section is to devise appropriate local extensions of the species' densities $u_i^\sm$ across the interface to the compartment $\Om_{-\sm}$, which is an indispensable ingredient in our renormalised solution concept coping with higher-order nonlinear interface conditions.
To this end, it is instructive to recall the standard extension technique across Lipschitz boundaries summarised in Appendix~\ref{ssec:reflection}.
In particular, for interior points $\beta\in\Gamma\setminus\pa\Gamma$ of the $C^1$ hypersurface $\Gamma$ with (or without) boundary, it suffices to consider the classical reflection map $\RFL_\beta:V_\beta\to V_\beta$ at $\Gamma$ on a 
sufficiently small neighbourhood $V_\beta\subset\mathbb{R}^d$ of $\beta$ (see~Appendix~\ref{ssec:reflection} for the concrete construction). The map $\RFL_\beta$ is a bi-Lipschitz homeomorphism (and even a $C^1$-diffeomorphism since $\Gamma$ is $C^1$) and measure-preserving, i.e.\ $|\det D\RFL_\beta|=1$ a.e.\ in $V_\beta$. 
\smallskip

\noindent If $\beta\in \pa\Gamma$, a mere extension by reflection at $\wt\Gamma$ will not be sufficient. Loosely speaking, the reason is that we need to preserve the no-flux boundary condition at $\pa\Om_\pm\setminus\Gamma$ for the extended species in order to avoid uncontrolled  traces. Hence, we need a refinement, which bijectively maps the region $V\cap\Om_\sm$ to $V\cap\Om_{-\sm}$ for some neighbourhood $V$ of $\beta$ and respects the relevant boundary structures.

\paragraph{Extension across flat interface}
In this paragraph, we outline the main idea of the extension procedure in the case where the given point $\beta\in\pa\Gamma$ coincides with the origin, i.e.\ $\beta=0_{\mathbb{R}^d}$, $\pa\Gamma$ is flat near $\beta$, and $\wt\Gamma$ is a hyperplane such that 
\begin{align*}
\Gamma\cap B_1^d(\beta)&=\big(\mathbb{R}^{d-2}\times\mathbb{R}_{\le0}\times\{0_{\mathbb{R}}\}\big)\cap B_1^d(\beta),
\\\wt\Gamma\cap B_1^d(\beta)&=B_1^{d-1}(\beta)\times\{0_{\mathbb{R}}\}=\{x_d=0\}\cap B_1^d(\beta).
\end{align*} 
Let $\mathfrak R$ denote the standard reflection at $\{x_d=0\}$ given by $\mathfrak R(x',x_d)=(x',-x_d)$ for all $x=(x',x_d)\in\mathbb{R}^d$. Then, in particular, $\mathfrak R(\beta')=\beta'$ for all $\beta'\in \Gamma\cap B_1^d(\beta)$.
We now assert that, for a suitably chosen small open neighbourhood $V\subseteq B_1^d(\beta)$ of the origin $\beta$, we can define a measure-preserving bi-Lipschitz homeomorphism 
$\Phi_{-}:\Om_-\cap V\to \Om_{+}\cap V$ fixing $\Gamma\cap V$ pointwise in the trace sense
by setting $\Phi_{-}:={\wh\Phi}_{|\Om_-\cap V}$, where
\begin{align}\label{eq:reflection.general}
	\wh\Phi:=O\circ \vartheta_2\circ\vartheta_1^{-1}\circ O^{-1}\circ \mathfrak{R}.
\end{align}
Here, $O\in SO(d)$ is a suitable rotation,
 $\vartheta_i:B^{d-1}_r(0)\times(-r,r)\to W_i\subset\mathbb{R}^d, i=1,2,$ are bi-Lipschitz homeomorphisms defined via $\vartheta_i(y,s)=(y,\eta_i(y)+s)$, where the Lipschitz maps $\eta_1,\eta_2$ are  height functions that locally determine the Lipschitz domains $\mathfrak{R}(\Om_-)$ resp.\ $\Om_+$
 near $\beta$ as hypographs over $\mathbb{R}^{d-1}$ in such a way that 
 \begin{align}\label{eq:common.frame}
 \text{$\eta_1(y)=\eta_2(y)$ for all $y\in B_r^{d-1}(0)$ with $O(\vartheta_2(y,0))\in\Gamma$. }
 \end{align}
 Note that the Lipschitz regularity of $\mathfrak{R}(\Om_-)$ follows from the fact that $\mathfrak R$ is a $C^1$ diffeomorphism.
To show that property~\eqref{eq:common.frame} is achievable under hypothesis~\ref{hp:C1ext} we argue as follows:
     	 since the domains $\Om_+$ and $\mathfrak{R}(\Om_-)$ are Lipschitz regular, lie on one side of the hyperplane $\wt\Gamma$ (let's say the side containing $\{x_d<0\}\cap B_1^d(\beta)$), and since their boundaries both contain the flat portion $\ol{B_{1/2}^d}(\beta)\cap\Gamma\ni\beta$,
     	  it follows that they admit a common interior cone segment with vertex $\beta=0_{\mathbb{R}^d}$, axis of orientation $\vec{n}=(-e_{d-1}-\delta e_{d})/\sqrt{1+\delta^2}$ for $0<\delta\ll1$ small enough, and sufficiently small height $\ell>0$ and aperture $\theta>0$.
     Let $H_{\vec n}$ denote the hyperplane passing through the origin $\beta$ whose normal is $\vec n$. Then, choosing the neighbourhood $V$ of $\beta$ suitably small,  the sets $\Om_+\cap V$ and $\mathfrak{R}(\Om_-\cap V)$ can be written as hypographs over $H_{\vec n}\cap B^{d-1}_{\tilde\ve}(\beta)$ for some small $\tilde\ve>0$. 
     
     We refer to Figure~\ref{fig:ext} for an illustration of this procedure.
     The crucial point is the existence of a \textit{common} direction $\vec n$ (and fixed $O\in SO(d)$) that does not depend on $i$, which guarantees that the map $\wh\Phi$ keeps $\Gamma$ fixed pointwise, as will be required in the renormalisation.  
Let us summarise the properties of $\Phi_{-}$ which we rely on in our analysis and which will equally be obtained for the transformation in the non-flat case: 
\begin{enumerate}[label={\normalfont($\Phi$\arabic*)}]
	\item\label{it:Phi.biLipschitz.homeo} The restriction $\Phi_{-}:=\wh\Phi_{|\Om_-\cap V}$ is a bi-Lipschitz homeomorphism from $\Om_-\cap V$ to $\Om_+\cap V$.
			\item\label{it:Phi.Gamma.fix} Every $\beta'\in\Gamma\cap V$ is a fixed point of $\wh\Phi_{|\ol\Om_-\cap V}$
			and $\wh\Phi((\pa\Om_-\setminus \Gamma)\cap V)=(\pa\Om_+\setminus \Gamma)\cap V$.
	\item\label{it:Phi.measure-pres} $|\det D\Phi_{-}|\equiv1$, $\mathcal{L}^d$-a.e.\ in $\Om_-\cap V$.
	\\	This last property follows from the fact that $\det D\vartheta_i=1$ for $i=1,2$ by the specific choice of $\vartheta_i$ (see also Appendix~\ref{ssec:reflection}) in combination with the fact that $O,\mathfrak R\in O(d)$.
\end{enumerate}
	We set $\Phi_{+}:=\Phi_{-}^{-1}$. The  fact that the construction above is asymmetric w.r.t.\ $\pm$ has no significance. The choice of $\Phi_{+}$ 
	was made to simplify the set-up, ensuring in particular that $\Phi_{+}$ and $\Phi_{-}$ are mutually inverse.

 \begin{figure}[t!]
	\centering
	\begin{subfigure}[t]{0.49\textwidth}
		\centering
		\includegraphics[scale=.7]{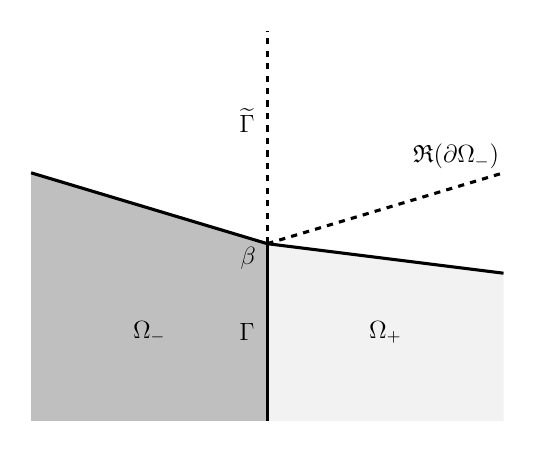}
		\caption{Local reflection of $\pa\Om_-$ near point $\beta\in\pa\Gamma$ at  flat hypersurface $\widetilde\Gamma\supset\Gamma$.}
		\label{fig:step1.bdry-reflect}
	\end{subfigure}%
	~
	\begin{subfigure}[t]{0.49\textwidth}
		\centering
		\includegraphics[scale=1]{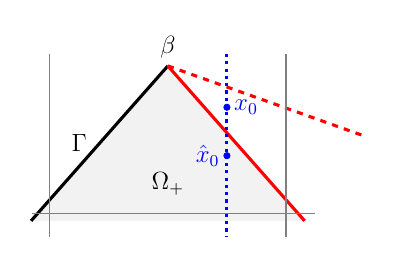}
		\caption{Measure-preserving change using Lipschitzian hypograph representation. The portion of hypograph below $\Gamma$ is kept fixed, while $\hat x_0=\vartheta_2\circ\vartheta_1^{-1}(x_0)$.}
		\label{fig:step2.bdry-reflect}
	\end{subfigure}
	\caption{Two-step extension procedure for $\beta\in\pa\Gamma$ in the case of a flat interface.}\label{fig:ext}
\end{figure}

\paragraph{General extension near interface points}

If $\wt\Gamma$ and $\pa\Gamma$ are not both flat but still $C^{1,1}$ smooth, then given $\beta\in\pa\Gamma$, one option is to simply choose a $C^{1,1}$ diffeomorphism $\Theta_\beta$ that provides a change of coordinates to the flat setting and then set $\wh\Phi_{\beta}:=\Theta_\beta\circ\wh \Phi\circ\Theta^{-1}_\beta$. (While in this case the transformation is no longer measure preserving, the uniform positivity and Lipschitz regularity of $|\det D\wh\Phi_\beta|$ would be sufficient for our theory of renormalised solutions).
However, in order to admit general $C^1$-regular hypersurfaces, we proceed directly as in~\eqref{eq:reflection.general} working with the curved $C^1$ surface. The main point that remains to be verified is the existence of a common interior cone to ensure the analogue of property~\eqref{eq:common.frame}. To this end, it suffices to notice that 
 the image of an open cone segment with vertex $0$ under a $C^1$ diffeomorphism $\Psi$ contains an open cone segment with vertex $\Psi(0)$. Concretely, after applying the standard reflection $\mathfrak R_\beta$ at the curved $C^1$ hypersurface $\wt\Gamma$ in a small open neighbourhood $V$ of the point $\beta$ (cf.\ Appendix~\ref{ssec:reflection}), the common interior cone segment for $\mathfrak R_\beta(\Om_-\cap V)$ and $\Om_+$ with vertex $\beta$ can be obtained by constructing the cone $\mathcal C$ with vertex $0$ in the image under $\wh\Phi^{-1}$, in which $\wt\Gamma$ and $\pa\Gamma$ are flat (see previous paragraph). Then $\Psi(\mathcal C)\subset\mathfrak R_\beta(\Om_-\cap V)\cap\Om_+$ contains a cone segment with vertex $\beta$, whose axis provides us with the normal direction of a hyperplane that can be used for the common hypograph representation of $ \mathfrak R_\beta(\Om_-\cap V),$ $\Om_+\cap V$, possibly after passing to a smaller neighbourhood $V$.

Finally, note that the standard reflection for interior points $\beta\in\Gamma\setminus\pa\Gamma$ can be regarded as a special case of the construction above with $\eta_1=\eta_2$. Thus, below we will use a uniform notation for the associated change of variables $\Phi_{\beta,\sm},\sm\in\{\pm\}$, for any $\beta\in\Gamma$. Importantly, all maps $\Phi_{\beta,\sm}$ satisfy properties~\ref{it:Phi.biLipschitz.homeo}--\ref{it:Phi.measure-pres}.

Given $\beta\in\Gamma$ and $w:\Om\to\mathbb{R}$, we define $\tilde w^\beta:V_\beta\cap\Om\to\mathbb{R}$ via
\begin{align}\label{eq:def.ext.c}
	\tilde w^\beta=\begin{cases}
		w^{-}\circ\Phi_{\beta,+} & \text{ in }\Om_+ \cap V_\beta,
	\\	w^{+}\circ\Phi_{\beta,-} & \text{ in } \Om_{-}\cap V_\beta.
	\end{cases}
\end{align}
Occasionally, it will be convenient to use the notation $\Phi_\beta$ for the measure-preserving bi-Lipschitz homeomorphism $\Phi_\beta:\Om\cap V_\beta\to\Om\cap V_\beta$ determined by 
\begin{align}\label{eq:def.Phi.beta}
\Phi_\beta(x)=\begin{cases}
	\Phi_{\beta,+}(x), \quad & x\in\Om_+ \cap V_\beta,
	\\\Phi_{\beta,-}(x), & x\in\Om_{-}\cap V_\beta.
\end{cases}
\end{align}
Notice that, just like standard reflection mappings, $\Phi_\beta$ is an involution, i.e.\ $\Phi_\beta^{-1}=\Phi_\beta$.

For simplicity, we use the same notation for time-dependent functions $w^\pm$, where the composition~\eqref{eq:def.ext.c} is to be understood with respect to the spatial variable, pointwise in time.
The renormalised formulation will make use of the local (generalised) reflection $\tilde u^\beta_i$ of the functions $u_i$ (cf.~\eqref{eq:def.ui}) near $\beta\in\Gamma$ described above. 

\subsection{Concept of renormalised solutions and main results}

\begin{definition}[Dissipative renormalised solutions]\label{def:renormalised.gen} Suppose~\ref{hp:geo.basic},~\ref{hp:C1ext} and~\ref{it:HP.diff}--\ref{it:HP.int} hold true.
 Let $\bfu_0=(u_{0,1}^+,\dots,u_{0,n}^+,u_{0,1}^-,\dots,u_{0,n}^-)$ with $u_{0,i}^\sm\in L^1(\Om_\sm), u_i^\sm\ge0,$ and $h(u^\sm)=\sum_{i=1}^n\mathfrak{B}(u_{0,i}^\sm)\in L^1(\Om_\sm)$ for $\sm=\{\pm\}$. 
	We call	$\bfu =(u^+,u^-)$ a renormalised solution to \eqref{eq:sys.sigma} with initial value $\bfu_0$ if, for all $\sigma\in\{\pm\}$, the following items 1.\ and 2. hold true:
	\begin{enumerate}
		\item Regularity: For $i=1,\dots,n$ it holds that 
		\begin{equation}\label{eq:reg.g}
			u_i^{\sigma} \in L^{\infty}([0,\infty); (L\log L)(\Omega_{\sigma})), \quad \nabla \sqrt{u_i^{\sigma}}\in L^2([0,\infty);L^2(\Omega_{\sigma})).
		\end{equation}
		\item Evolution equations
		\begin{enumerate}
			\item Bulk and outer boundary:
			for all $\zeta\in C^\infty([0,\infty)^n)$ with $D\zeta\in C^\infty_c([0,\infty)^n)^n$,
			 all $\psi\in C^\infty_c([0,\infty)\times(\ol\Om_\sigma\setminus\Gamma))$, and almost all $T\in(0,\infty)$,
			\begin{align}\label{eq:def.renorm.out.gen}
				\begin{aligned}
				\int_{\Om_\sigma}\psi(T,\cdot)&\,\zeta(u(T,\cdot))\,\dd x	-\int_{\Om_\sigma}\psi(0,\cdot)\,\zeta(u_0)\,\dd x-\int_0^T\!\!\int_{\Om_\sigma}\pa_t\psi\, \zeta(u)\,\dd x\dd t=
					\\&-\sum_{i=1}^n\int_0^T\!\!\int_{\Om_\sigma}\nabla(\psi  D_i\zeta(u))\cdot A_i\nabla u_i
					\,\dd x\dd t
					+\sum_{i=1}^n\int_0^T\!\!\int_{\Om_\sigma}\psi\, D_i\zeta(u)f_i(u)\,\dd x\dd t,
				\end{aligned}
			\end{align}
			where $u=(u_1,\dots,u_n)$ and $u_0$ are defined via~\eqref{eq:def.ui} from $\bfu=(u^+,u^-)$ resp.\ $\bfu_0=(u^+_0,u_0^-)$. 
			\item\label{it:renorm.int} Interface:
				for all $\xi\in C^\infty([0,\infty)^{2n})$ with $D\xi\in C^\infty_c([0,\infty)^{2n})^{2n}$, all $\beta\in\Gamma$
			with associated (generalised) local reflection map $\Phi_{\beta,\sigma}:\Om_\sigma\cap V_\beta\to \Om_{-\sigma}\cap V_\beta$ (cf.\ Section~\ref{ssec:ext}),
			 all $\psi\in C^\infty_c([0,\infty)\times V_\beta)$, and almost all $T\in(0,\infty)$,
			\begin{align}
			\int_{\Om_\sigma\cap V_\beta}&\psi(T,\cdot)\,\xi(\wt\bfu^\beta(T,\cdot))\,\dd x	-\int_{\Om_\sigma\cap V_\beta}\psi(0,\cdot)\,\xi(\wt\bfu^\beta_0)\,\dd x-\int_0^T\!\!\int_{\Om_\sigma\cap V_\beta}\pa_t\psi\, \xi(\wt\bfu^\beta)\,\dd x\dd t=
					\nonumber
					\\&-\sum_{i=1}^n\int_0^T\!\!\int_{\Om_\sigma\cap V_\beta}\nabla(\psi D_i\xi(\wt\bfu^\beta))\cdot A_i\nabla u_i	\,\dd x\dd t\label{eq:def.renorm.c.gen}
					\\&-\sum_{i=1}^n\int_0^T\!\!\int_{\Gamma_\sigma\cap V_\beta}\psi\,D_{i}\xi(\wt\bfu^\beta)r_i^\sigma(\bfu)\,\dd\mathcal{H}^{d-1}\dd t\nonumber
					\\& +\sum_{i=1}^n\int_0^T\!\!\int_{\Om_\sigma\cap V_\beta}\psi\, D_i\xi(\wt\bfu^\beta)f_i(u)\,\dd x\dd t\nonumber
					\\&-\sum_{i=1}^n\int_0^T\!\!\int_{\Om_{-\sigma}\cap V_\beta}
					\nabla \big((\psi D_{i+n}\xi(\wt\bfu^\beta)\circ\Phi_{\beta,-\sigma})\big)\cdot A_i\nabla u_i\, \dd x\dd t
					\nonumber
					\\&-\sum_{i=1}^n\int_0^T\!\!\int_{\Gamma_{-\sigma}\cap V_\beta}(\psi D_{i+n}\xi(\wt\bfu^\beta))\circ\Phi_{\beta,-\sigma}
					 r_i^{-\sigma}(\bfu)\,\dd\mathcal{H}^{d-1}\dd t\nonumber
					\\&	+\sum_{i=1}^n\int_0^T\!\!\int_{\Om_{-\sigma}\cap V_\beta} (\psi D_{i+n}\xi(\wt\bfu^\beta)\big)\circ\Phi_{\beta,-\sigma} f_i(u)\,\dd x\dd t\nonumber
			\end{align}				
			with $\wt\bfu^\beta:=(u,\tilde u^\beta)$, where $\tilde u^\beta$ is defined according to~\eqref{eq:def.ext.c} and $u=(u_1,\dots,u_n)$ via~\eqref{eq:def.ui}. 
		\end{enumerate} 
	\end{enumerate}
	If, in addition, item 3.\ below is satisfied, we call $\bfu$ a {\em dissipative renormalised solution} to~\eqref{eq:sys.sigma}:
	\begin{enumerate}[resume]
		\item Entropy dissipation inequality: for $\mathcal{L}^1$-almost all $t> 0$
		\begin{align}\label{eq:edi}
			H(u(t))+\int_0^t\Big(\sum_{\sigma\in\{\pm\}}\mathcal{D}_{\sigma}(u)+\mathcal{D}_{\textrm{\normalfont int}}(\bfu)\Big)\,\dd s\le H(u_0),
		\end{align}
		where $H(u(t))=\int_{\Om}h(u(t))\,\dd x=\sum_{\sigma\in\{\pm\}}\int_{\Om_\sm}h(u^\sm(t))\,\dd x$ and
		\begin{align*}
			\mathcal{D}_{\sigma}(u)&=\int_{\Om_\sigma}\sum_{i=1}^n4|\sqrt{A_i}\nabla\sqrt{u_i}|^2\,\dd x+\int_{\Om_\sigma}\sum_{i=1}^n\Big(-f_i(u)D_ih(u)\Big)\,\dd x,
			\\\mathcal{D}_{\textrm{\normalfont int}}(\bfu)&=\int_{\Gamma}\sum_{\sigma\in\{\pm\}}\sum_{i=1}^n r_i^\sigma(\bfu)D_{i}h(u^\sigma)\,\dd\mathcal{H}^{d-1}.
		\end{align*}
	\end{enumerate}
\end{definition}
\begin{remark}
	\begin{enumerate}[label=(\roman{*})]
		\item 	The notion of renormalised solutions introduced above is well defined, since $\wt\bfu^\beta\circ\Phi_{\beta,-\sigma}=(u,\tilde u^\beta)\circ\Phi_{\beta,-\sigma}=(\tilde u^\beta,u)$ in $\Om_{-\sigma}\cap V_\beta$
		for all $\beta \in \Gamma$.
		\item The concept of {\em dissipative} renormalised solutions to reaction--diffusion systems was first introduced in~\cite{Hopf_2022}. The point is that while along the construction of solutions the entropy-dissipation inequality (cf.~\eqref{eq:edi}) can typically be derived, it may not always be possible a posteriori. For the present model, 
		the interface conditions appear to prevent such a derivation.
		\item Let us provide a basic consistency check (assuming smooth data): any positive classical solution $\bfu=(u^+,u^-)$ to~\eqref{eq:sys.sigma} is a dissipative renormalised solution in the sense of Definition~\ref{def:renormalised.gen} with initial value $\bfu_{|t=0}$. This can be seen by a direct calculation as sketched in the introduction (cf.\ pages~\pageref{p:strategy.0}--\pageref{p:strategy.1}) using the properties of the generalised reflection mappings $\Phi_{\beta,\sm}$, cf.~\ref{it:Phi.biLipschitz.homeo}--\ref{it:Phi.measure-pres}. On the other hand, if $\bfu$ is a renormalised solution in the sense of Definition~\ref{def:renormalised.gen} that is additionally bounded, say $u^\sm_i\le E$ for all $i,\sm$, then upon choosing $\xi\in C^2_c([0,\infty)^{2n})$, $\zeta\in C^2_c([0,\infty)^{n})$ with $\xi(s,\tilde s)=s_k$, $\zeta(s)=s_k$ for $s,\tilde s\in [0,E]^{n}$,  $k=1,\dots,n$,
		in equations~\eqref{eq:def.renorm.c.gen},~\eqref{eq:def.renorm.out.gen}, glued together with a partition of unity (cf.~\eqref{eq:partition1}),  one easily arrives at a standard weak formulation of~\eqref{eq:sys.sigma}. \\We expect that the tools that we will develop below for the weak--strong uniqueness principle (including the "partition-of-unity" device) allow to conclude further, refined consistency and qualitative properties for the above solution concept. 
		For the sake of brevity, such side questions will, however, not be explicitly discussed in the present manuscript.
	\end{enumerate}
\end{remark}

Global existence based on this solution concept is fundamental to our theory.

\begin{theorem}[Existence of dissipative renormalised solutions]\label{thm:ex}
	Assume hypotheses~\ref{hp:geo.basic},~\ref{hp:C1ext} and~\ref{it:HP.diff}--\ref{it:HP.int}. Then for any initial data $\bfu_0$ that is componentwise non-negative with $H(u_0)=\sum_{\sm\in\{\pm\}}\int_{\Om_\sm}h(u^\sm_0)\,\dd x<\infty$,
	there exists a global dissipative renormalised solution $\bfu$ in the sense of Definition~\ref{def:renormalised.gen}. 
\end{theorem}
The proof of Theorem~\ref{thm:ex}, which adapts~\cite{Fischer_2015}, is postponed to Section~\ref{sec:existence}. We remark that, along the construction, it is possible to derive a stronger version of the entropy-dissipation inequality, that is additionally valid for a.e.\ $0<s<t$, which is important for obtaining convergence rates to equilibrium (see e.g.~\cite[Propositions 2.1, 2.2]{Hopf_2022}).

\begin{remark}[Regularity hypotheses on the geometry]$ $
	The extra assumption of $C^1$-regularity of the interface was made for simplicity and can be relaxed.  In principle, it is only needed in a neighbourhood of points $p\in\pa\Gamma$ to ensure the existence of the generalised reflection map $\Phi_{
		\beta}$ enjoying the required hypotheses. Moreover,
	in the topological setting of Figure~\ref{fig:annular}, we can use the standard reflection technique at the Lipschitz boundary of $\Om_+$ at any point $z\in\Gamma$. Hence, in this case our main results, Theorems~\ref{thm:unique},~\ref{thm:ex}, remain valid without assumption~\ref{hp:C1ext}.
\end{remark}	

A key motivation for the notion of a dissipative renormalised solution introduced above is the fact that it not only allows for global solvability but also enjoys a weak--strong uniqueness principle.

\begin{definition}[Strong solution]\label{def:strong}
We call $\bfU:=(U^+,U^-)$, $U^\sigma\in C^{0,1}_{\loc}([0,T^*)\times\ol\Om_\sigma)$, $\sigma\in\{\pm\}$, $U_{|\Om_\pm}:=U^\pm$, a strong solution to~\eqref{eq:sys.sigma} on $(0,T^*)$ if $\min_{[0,T]\times\ol\Om_\pm} U^\pm>0$ for all $T<T^*$
and if
 $\boldsymbol{U}$ satisfies~\eqref{eq:sys.sigma} in the weak sense on $(0,T^*)$, i.e.\ for $i=1,\dots,n$, for $\sm\in\{\pm\}$, and all $\psi\in L^1(0,T;W^{1,1}(\Om_\sm))$
\begin{align}\label{eq:strongsol.weak}
	\int_{\Om_\sigma}\psi\,\pa_tU_i\,\dd x\dd t=
	-\int_0^T\!\int_\Gamma \psi\,r_i^\sigma(\bfU)\,\dd x\dd t
-\int_0^T\!\int_{\Om_\sigma}\nabla\psi\cdot A_i\nabla U_i
\,\dd x\dd t
+\int_0^T\!\int_{\Om_\sigma}\psi\, f_i(U)\,\dd x\dd t
\end{align}
for every $T\in(0,T^*)$.
\end{definition}

\begin{theorem}[Weak--strong uniqueness]\label{thm:unique}
	Assume hypotheses~\ref{hp:geo.basic},~\ref{hp:C1ext} on the geometry and~\ref{it:HP.diff}--\ref{it:HP.int} on the diffusion coefficients and nonlinearities, and let $\bfu_0$ be measurable, non-negative componentwise, and satisfy $H(u_0)=\sum_{\sm\in\{\pm\}}\int_{\Om_\sm}h(u^\sm_0)\,\dd x<\infty$.  Further adopt the notations in Section~\ref{ssec:ext}.
	Let $\bfu$ be a dissipative renormalised solution with initial data $\bfu_0$ in the sense of Definition~\ref{def:renormalised.gen}, let $T^*\in(0,\infty]$ and let  $\boldsymbol{U}:=(U^+,U^-)$, $U^\sigma\in C^{0,1}_{\loc}([0,T^*)\times\ol\Om_\sigma)$, $\sigma\in\{\pm\}$, $U_{|\Om_\pm}:=U^\pm$, be a strong solution with $U(0)=u_0$. Then $u=U$ $\mathcal{L}^{1+d}$-a.e.\ in $(0,T^*)\times\Om$.
\end{theorem}
Theorem~\ref{thm:unique} will be obtained as a consequence of a stability estimate that controls the deviation from a strong solution in terms of the distance of the initial data up to a multiplicative constant that only depends on the Lipschitz norm and distance from zero of the strong solution, see Proposition~\ref{prop:stab.gen}. As a generalised metric, we will use a suitably truncated version of the classical relative entropy $H_{\rel}^s(u|U):=H(u)-H(U)-\la DH(U),u{-}U\ra$.
Recalling that a distributional formulation is generally not available for renormalised solutions,
just like in the renormalisation of bulk processes a truncation needs to be introduced in the linear correction term to $H(u)$ arising in $H_{\rel}^s(u|U)$ to give a meaning to the evolution of this functional along renormalised solutions $u$ with respect to the strong solution $U$ (cf.~\cite{Fischer_2017,CJ_wkstr-unique_2019,Hopf_2022}), at least in the sense of a (sharp) inequality. 
The main new point in the case of nonlinear interface conditions is that the truncation must also take into account the extensions $\tilde u^\beta$ from the neighbouring compartment, which are only defined locally. We address this by means of a suitable partition of unity.
Furthermore, to establish the necessary coercivity properties of 
the adjusted relative entropy functional, we need to adopt a partially non-local perspective, since the possible occurrence of large values of $\tilde u^\beta$ prevents us from relying on the local convexity of the standard relative entropy density. These core preliminary results will be developed in Section~\ref{ssec:adjrelen.coer}.

\begin{remark}[Existence of a strong solution]
	In order to ensure the short-time existence of a strong solution in the sense of Definition~\ref{def:strong}, stronger regularity hypotheses generally need to be imposed on the data including the domains $\Om_\pm$. 
	It is an open question whether a version of the weak--strong uniqueness principle persists under less restrictive hypotheses on the strong solution such as those arising in Disser's framework~\cite{Disser_2020_bulk-interface}, which covers Lipschitz domains.
\end{remark}

\section{Weak--strong uniqueness}\label{sec:weakstruniq} 
Our weak--strong uniqueness principle extends the strategy in~\cite{Fischer_2017,Hopf_2022}.
Throughout this section, we assume hypotheses~\ref{hp:geo.basic},~\ref{hp:C1ext} and~\ref{it:HP.diff}--\ref{it:HP.int}, and suppose that $\bfu_0=(u^+_0,u^-_0)$ is non-negative componentwise, i.e.\ $u^\sm_{0,i}\ge0$ for $i=1,\dots,n$ and $\sm\in\{\pm\},$ and satisfies $\sum_{\sm\in\{\pm\}}\int_{\Om_\sm}h(u^\sm_0)\,\dd x<\infty$. 
Moreover, {\em unless specified otherwise}, we assume that $\bfU=(U^+,U^-)$ 
 is a  strong solution in the sense of Definition~\ref{def:strong} for some $T^*\in(0,\infty]$.
  We let $\ol T\in(0,T^*)$ be fixed but arbitrary. 
We then set \[\iota:=\min_i\inf_{[0,\ol T]\times\Om} U_i>0\qquad\text{ and }\qquad\mathfrak I:=\max_{\sm\in\{\pm\}}\max_{i=1,\dots,n}\sup_{t\in[0,\ol T]}\|U_i^\sm(t)\|_{C^{0,1}(\Om_\sm)}<\infty.\]
The dependence on $\iota,
\mathfrak I$ of (hidden) constants in the estimates below will not be explicitly indicated.
Note that 
$\max_{i,\sigma}\|\nabla \sqrt{U_i^\sigma}\|_{L^\infty([0,\ol T]\times\Om_\sigma)}\le C,$ where $C=C(\iota,\mathfrak{I})<\infty$, which follows from the identity $\nabla \sqrt{U_i^\sigma}=\frac{1}{2\sqrt{U^\sigma_i}}\nabla U_i$. 

\subsection{Relative entropies and coercivity estimates}\label{ssec:adjrelen.coer}

\paragraph{Partition of unity}

As we have shown in Section~\ref{ssec:ext}, for all $\beta\in\Gamma$, there exists an open neighbourhood $V_\beta\subset\mathbb{R}^d$ of $\beta$ that admits generalised local reflection maps $\Phi_{\beta,\sm}:\Om_{\sm}\cap V_\beta\to \Om_{-\sm}\cap V_\beta$, $\sm\in\{\pm\}$, with the property that  $\Phi_{\beta,-\sm}=\Phi_{\beta,\sm}^{-1}$.
Since $\Gamma$ is compact, the covering $\cup_{\beta}V_\beta\supset\Gamma$ has a finite subcovering,  i.e.\ there exists a finite set of points $\frakP=\{\beta_1,\dots,\beta_P\}\subset \Gamma$ such that 
$\Gamma\subset\cup_{\beta\in\frakP}V_\beta$.
We then choose smooth non-negative functions $0\le\vp_{\beta}\le1,\beta\in\frakP$, and $\vp_{\text{\rm out}}$, supported in $V_{\beta}$ and in $\ol\Om\setminus\Gamma$ respectively, in such a way that we obtain a partition of unity subordinate to the above covering of $\ol\Om=(\ol\Om\setminus\Gamma)\bigcup(\cup_{\beta\in\frakP}V_\beta):$
\begin{align}\label{eq:partition1}
	\vp_{\text{\rm out}}+\sum_{\beta\in\frakP}\vp_{\beta}
	\equiv 1\quad\text{on }\ol\Om=\ol\Om_+\cup\ol\Om_-.
\end{align}

\paragraph{Adjusted relative entropy}Given two sufficiently large parameters $E,N\ge2$ to be determined later,
define the function $\hat\xi:\mathbb{R}\to[0,1]$ via
\begin{align*}
	\hat\xi(r)=\eta\left(\frac{\log r-\log E}{\log E^N-\log E}\right),
\end{align*}
where $\eta\in C^\infty(\mathbb{R})$ is a fixed function with $\eta(s)=1$ for $s\le0$, $\eta(s)=0$ for $s\ge1$, and $\eta'\le0$ with $\eta'(s)<0$ for all $s\in(0,1)$.

We can then define $\xi^*\in C^\infty_c([0,\infty)^{2n})$ by
\begin{align}\label{eq:regions}
	\xi^*(\wt\bfu)=\begin{cases}
		1\qquad&\text{if }\wt \bfu\in \mathcal{A}:=\{\wt \bfv=(v,\tilde v)\in[0,\infty)^{2n}\mid|\wt \bfv|_1:=\sum_{i=1}^nv_i+\sum_{i=1}^n\tilde v_i\le E\},
		\\ \hat\xi(|\wt \bfu|_1)&\text{if }\wt\bfu\in \mathcal{B}:=\{E<|\wt\bfv|_1\le E^N\}, 
		\\0&\text{if }\wt \bfu\in \mathcal{C}:=\{|\wt \bfv|_1>E^N\},
	\end{cases}
\end{align}
and 
$\zeta^*\in C^\infty_c([0,\infty)^{n})$ via
\begin{align*}
	\zeta^*(u)=\begin{cases}
		1\qquad&\text{if }|u|_1:=\sum_{i=1}^nu_i\le E,
		\\ \hat\xi(|u|_1)&\text{if }E<|u|_1\le E^N,
		\\0&\text{if }|u|_1>E^N.
	\end{cases}
\end{align*}

Observe that for all $\wt\bfu\in [0,\infty)^{2n}$ and $u\in [0,\infty)^{n}$
\begin{subequations}\label{eq:decay.trunc}
	\begin{align}\label{eq:decay.xi}
		|D\xi^*(\wt\bfu)|\lesssim\frac{1}{N|\wt\bfu|_1}, \qquad |D^2\xi^*(\wt\bfu)|\lesssim\frac{1}{N|\wt\bfu|_1^2},
		\\\label{eq:decay.zeta}
		|D\zeta^*(u)|\lesssim\frac{1}{N|u|_1}, \qquad |D^2\zeta^*(u)|\lesssim\frac{1}{N|u|_1^2}.
	\end{align}
\end{subequations}
where the constants hidden in $\lesssim$ are independent of $E,N\ge2$.
 Furthermore, the strict decrease of $\eta$ on $(0,1)$ implies the technical properties
	\begin{subequations}
\begin{alignat}{3}\label{eq:xi.less.1}
		&\xi^*(\wt\bfu)<1\qquad&&\text{for }\wt\bfu\in\mathcal{B}\cup\mathcal{C},
		\\&\zeta^*(u)<1&&\text{for }|u|_1>E.\label{eq:zeta.less.1}
\end{alignat}
	\end{subequations}

We are now in a position to define an appropriately truncated relative entropy density
\begin{align*}
h_\rel(x;\wh\bfu|U):=
h(u)-\sum_{j=1}^n\log U_j\XIH(x,\wh\bfu)u_j
+\sum_{j=1}^nU_j,\quad x\in\Om,
\end{align*}
where here $U\in[\iota,\ol U]^n$, and 
\begin{subequations}\label{eq:def.whbfu.XIH}
\begin{align}
\wh\bfu&:=(u,(\tilde u^\beta)_{\beta\in\frakP})\in [0,\infty)^{n(1+\PP)},
\\\XIH(x,\wh\bfu)=\XIH(x,
(u,(\tilde u^\beta)_{\beta\in\frakP}))\label{eq:def.whbfu}
&:=\vp_{\text{\rm out}}(x)\zeta^*(u)+
\sum_{\beta\in\frakP}\vp_\beta(x)\xi^*(u,\tilde u^\beta)
\end{align}
\end{subequations}
with $\PP=\#\frakP$.

\paragraph{Coercivity properties}
We partition $\Om\times[0,\infty)^{n(1+\PP)}$ into the three sets
\begin{align*}
	S_g&:=\{(x,\wh\bfu)\in\Om\times[0,\infty)^{n(1+\PP)}\mid \XIH(x,\wh\bfu)=1\},
	\\S_p&:=\{(x,\wh\bfu)\in\big(\Om\times[0,\infty)^{n(1+\PP)}\big)\setminus S_g\mid |u|_1\ge E/2\}=\big(\Om\times
	\{ |u|_1\ge E/2\}
	\times[0,\infty)^{n\cdot\PP}\big)\setminus S_g,
	\\S_b&:=\{(x,\wh\bfu)\in\big(\Om\times[0,\infty)^{n(1+\PP)}\big)\setminus S_g\mid |u|_1< E/2\}
	=\big(\Om\times\{ |u|_1<E/2\}\times[0,\infty)^{n\cdot\PP}\big)\setminus S_g,
\end{align*}
where we continue to use the notation~\eqref{eq:def.whbfu.XIH}.
 As we will detail below, for points $(x,\wh\bfu)$ in the `good set' $S_g$, the adjusted relative entropy density $h_\rel(x;\wh\bfu|U)$ agrees with the standard relative entropy function
$h_\rel^s(u|U)$, from which the relevant coercivity properties are inherited. For points $(x,\wh\bfu)$ in $S_p$, where $\sum_{i=1}^nu_i$ takes large positive values, coercivity is a consequence of the superlinear growth of $h(u)$ as $|u|_1\to\infty$. When $(x,\wh\bfu)$ lies in the `bad set' $S_b$, we only know that the reflected density cannot be small, which requires a non-local argument to conclude an appropriate coercivity-type bound (cf.\ Proposition~\ref{prop:coerc.relen-functional}). This is the main new source of difficulty in the development of the relative entropy method and the derivation of a weak--strong stability estimate.

 We first gather some preliminary observations.
\begin{lemma}[Partial coercivity of truncated relative entropy density]\label{l:coerc.relen}$ $Let  $0<\iota<\ol U<\infty$ and $U\in [\iota,\ol U]^{n}$.
There exists a threshold $\ul E\gg 1$ (depending on $\ol U$, $\iota$, and fixed data) 
such that for $E\ge \ul E$
	\begin{enumerate}[label={\rm(\roman*)}]
		\item
		 For all $(x,\wh\bfu)\in S_g:$ $h_\rel(x;\wh\bfu|U)=h_\rel^s(u|U)\gtrsim\sum_{i=1}^n|\sqrt{u_i}-\sqrt{U_i}|^2$.
		\item For all 
		 $(x,\wh\bfu)\in S_p:$ $h_\rel(x;\wh\bfu|U)\ge\frac12\sum_iu_i\log u_i 	\gg 1$.
	\end{enumerate}
\end{lemma}
\begin{proof}[Proof of Lemma~\ref{l:coerc.relen}]
	Re (i): the identity
	\[h_\rel(x,\wh\bfu|U)=h_\rel^s(u|U):=h(u)-\sum_{j=1}^nD_jh(U)(u_j-U_j)-h(U)
	=h(u)-\sum_{j=1}^nD_jh(U)u_j+\sum_{j=1}^nU_j\]
	for $(x,\wh\bfu)\in S_g$ follows by construction. 
	Since $h$ is strongly convex on bounded subsets of $[0,\infty)^n$, the classical relative entropy $h_\rel^s$ satisfies $h_\rel^s(u|U)\gtrsim_E|u-U|^2$ if $|u|_1\le E$. 
	For the lower bound $|u-U|^2\gtrsim \sum_{i=1}^n|\sqrt{u_i}-\sqrt{U_i}|^2$ we distinguish two cases
	\begin{itemize}
		\item if $u\in\{v\in[\iota/2,\infty)^n\mid |v|_1\le E\}$, it follows from the uniform Lipschitz continuity of $s\mapsto\sqrt{s}$ on $[\iota/2,\infty)$
		\item if $u_{i_0}\in[0,\iota/2)$ for some $i_0$, it follows from the lower bound $U_i\ge\iota$ for all $i$
	\end{itemize} 
	
	\smallskip
	
	\noindent Re (ii): since $0\le\XIH\le1$, it follows for $(x,\wh\bfu)\in S_p$, if $E\gg1$ is chosen large enough, that
	$$h_\rel(x;\wh\bfu|U)\ge h(u)- C_1|u|_1-C_2\ge c_n|u|_1\log(|u|_1)\gg1,$$
	 where the last two inequalities follow from the superlinearity at infinity of $s\mapsto s\log s$.
\end{proof}

Given $\bfu=(u^+,u^-), u^\sm=u^\sm(x)\in [0,\infty)^{n}$ $\mathcal{L}^{d}$-measurable with $h(u^\sm)\in L^1(\Om_\sm),\sm\in\{\pm\}$, and $U=U(x)\in[\iota,\ol U]^n$, define $u(t,\cdot)$ via~\eqref{eq:def.ui} and $\tilde u^\beta(t,\cdot)$ via~\eqref{eq:def.ext.c}.
We then introduce the relative entropy functional via
\begin{align*}
	H_\rel(\wh\bfu|U):=\int_\Om h_\rel(x;\wh\bfu(x)|U(x))\,\dd x, \quad\text{where }\wh\bfu=(u,(\tilde u^\beta)_{\beta\in\frakP}), \;\;\tilde u^\beta=u\circ\Phi_\beta,
\end{align*}
where $\Phi_\beta$ is given by~\eqref{eq:def.Phi.beta}.
Later on, $\bfu$ will be chosen a renormalised solution $\bfu=\bfu(t)$ and $U$ corresponds to a strong solution $\bfU$.

We further define, up to $\mathcal{L}^d$-negligible sets, three measurable sets $S_a^{\wh\bfu}\subseteq\Om$, $a\in\{g,p,b\}$, satisfying $S_g^{\wh\bfu}\dot\cup S_p^{\wh\bfu}\dot\cup S_b^{\wh\bfu}=\Om$ via
\begin{align*}
&S_g^{\wh\bfu}:=\{x\in\Om\mid \XIH(x,\wh\bfu(x))=1\},
\\&S_p^{\wh\bfu}:=\{x\in\Om\setminus S_g^{\wh\bfu}\mid |u(x)|_1\ge E/2\},
\\&S_b^{\wh\bfu}:=\{x\in\Om\setminus S_g^{\wh\bfu}\mid |u(x)|_1<E/2\}.
\end{align*}
Loosely speaking, this means that $x\in S_a^{\wh\bfu}$ whenever ``$(x,\wh\bfu(x))\in S_a$" (here, for $x\not\in V_\beta$, $\tilde u^\beta(x)$ can be set to an arbitrary value, which does not change the definition since then $\vp_\beta(x)=0$).

The following proposition is the main result of this subsection. In particular, estimate~\eqref{eq:Hrel.coerc.Sg} below shows the distance-like properties of $H_\rel(\wh\bfu|U)$ in the sense that $H_\rel(\wh\bfu|U)\ge0$ with equality if and only if $u_i=U_i$ $\mathcal{L}^d$-a.e.\ in $\Om$ for all $i=1,\dots,n$. We emphasise that the pointwise analogue of this estimate fails, since $h_\rel(;\wh\bfu|U)$ may take negative values in $S_b^{\wh\bfu}$.
\begin{proposition}[Coercivity of relative entropy functional]\label{prop:coerc.relen-functional}
	If $E\gg_{\ol U,\iota}1$ is large enough, then
\begin{align}\label{eq:coerc.Sb}
			H_\rel(\wh\bfu|U)\gtrsim  E\log\!E\cdot \mathcal{L}^d(S_b^{\wh\bfu}).
			\end{align} 
			As a consequence,
	\begin{align}\label{eq:Hrel.coerc.Sg}
		&H_\rel(\wh\bfu|U)\gtrsim 
		\int_\Om	\sum_{i=1}^n|\sqrt{u_i}-\sqrt{U_i}|^2\,\dd x,
		\\&\label{eq:Hrel.coerc.complSg}H_\rel(\wh\bfu|U)\gtrsim E\log\!E\cdot \mathcal{L}^{d}(\Om\setminus S_g^{\wh\bfu}),
	\end{align}
	where all constants hidden in $\gtrsim$ are independent of $E,N$.
\end{proposition}
\begin{proof}For the proof, we fix a representative of $u\in L^1(\Om;[0,\infty)^n)$, determined pointwise everywhere in $\Om$, which we again denote by $u$, and define the measurable sets
	\begin{align*}
	S_{b,\beta}^{\wh\bfu}=\{x\in V_\beta\cap\Om\mid |\tilde u^\beta(x)|_1>E/2\},\qquad \beta\in\frakP. 
	\end{align*}
	We first assert that 
	\begin{align}\label{eq:union.Sb.beta}
		S_b^{\wh\bfu}\subseteq	\bigcup_{\beta\in\frakP} S_{b,\beta}^{\wh\bfu}.
	\end{align}
	To show~\eqref{eq:union.Sb.beta}, let $x\in S_b^{\wh\bfu}$. Then $\XIH(x,\wh\bfu(x))<1$ and $|u(x)|_1<E/2$. The second inequality implies that $\zeta^*(u(x))=1$, so that from the	
	first inequality we can infer the existence of $\beta'\in\frakP$ with $|\wt\bfu^{\beta'}(x)|_1>E$. We therefore have $|\tilde u^{\beta'}(x)|_1=|\wt\bfu^{\beta'}(x)|_1-|u(x)|_1>E-E/2=E/2$, 
	and hence $x\in S_{b,\beta'}^{\wh\bfu}$.
	This proves~\eqref{eq:union.Sb.beta}.
	
	From~\eqref{eq:union.Sb.beta} and the fact that $\Phi_\beta$ is measure preserving, it follows that 
		\begin{align}\label{eq:sum.b.beta}
		\mathcal{L}^d(S_b^{\wh\bfu})\le \sum_{\beta\in\frakP}\mathcal{L}^d(S_{b,\beta}^{\wh\bfu})=\sum_{\beta\in\frakP}\mathcal{L}^d(\Phi_\beta(S_{b,\beta}^{\wh\bfu})).
	\end{align}
	Next, from the definition of $\tilde u^\beta$ and $S_{b,\beta}^{\wh\bfu}$ we deduce  the lower bound
	\[|u(x)|_1=|\tilde u^\beta(\Phi_\beta^{-1}(x))|_1>E/2\qquad\text{ for all }x\in \Phi_\beta(S_{b,\beta}^{\wh\bfu})).\]
	Therefore,  if $E\gg1$ is large enough, arguing as in the proof of  Lemma~\ref{l:coerc.relen}~(ii), there exists $c_n>0$ such that $h_\rel(x;\wh\bfu(x)|U(x))\ge c_n E\log E$ for all $x\in  \Phi_\beta(S_{b,\beta}^{\wh\bfu})$.
	Consequently,
	\begin{align}\begin{aligned}
		\mathcal{L}^d(\Phi_\beta(S_{b,\beta}^{\wh\bfu}))
		&\lesssim \frac{1}{E\log E}\int_{\Phi_\beta(S_{b,\beta}^{\wh\bfu})} 
		h_\rel(x;\wh\bfu(x)|U(x))\,\dd x
		\\&= \frac{1}{E\log E}\Big(H_\rel(\wh\bfu|U)-\int_{\Om\setminus\Phi_\beta(S_{b,\beta}^{\wh\bfu})}	h_\rel(x;\wh\bfu(x)|U(x))\,\dd x\Big).
		\end{aligned}
	\end{align}
	
The pointwise bounds from Lemma~\ref{l:coerc.relen} and the fact that $h\ge0$ and $0\le\chi\le1$ imply
\begin{alignat*}{3}
	 &-h_\rel(\cdot;\wh\bfu|U)\le 0\qquad&&\text{ in }S_g^{\wh\bfu}\cup S_p^{\wh\bfu}=(S_b^{\wh\bfu})^c
	\\&-h_\rel(\cdot;\wh\bfu|U)\lesssim_{\ol U} |u|_1+1\quad &&\text{ in }\Om, 
\end{alignat*}
allowing us to estimate 
\begin{align*}
-\int_{\Om\setminus\Phi_\beta(S_{b,\beta}^{\wh\bfu})}h_\rel(x;\wh\bfu(x)|U(x))\,\dd x
&\le -\int_{\big(\Om\setminus\Phi_\beta(S_{b,\beta}^{\wh\bfu})\big)\setminus\big(S_g^{\wh\bfu}\cup S_p^{\wh\bfu}\big)}h_\rel(x;\wh\bfu(x)|U(x))\,\dd x
\\&\lesssim \int_{\big(\Om\setminus\Phi_\beta(S_{b,\beta}^{\wh\bfu})\big)\setminus\big(S_g^{\wh\bfu}\cup S_p^{\wh\bfu}\big)}(|u|_1+1)\,\dd x
\lesssim E\mathcal{L}^d(S_{b}^{\wh\bfu}),
\end{align*}
where the last inequality follows from the inclusion $\big(\Om\setminus\Phi_\beta(S_{b,\beta}^{\wh\bfu})\big)\setminus\big(S_g^{\wh\bfu}\cup S_p^{\wh\bfu}\big)\subset\Om\setminus\big(S_g^{\wh\bfu}\cup S_p^{\wh\bfu}\big)=S_{b}^{\wh\bfu}$ and the bound
$|u|_1\lesssim E$ in $S_b^{\wh\bfu}$.

In combination, we infer
	\begin{align*}
E(\log E-C_0)\mathcal{L}^d(S_{b}^{\wh\bfu})\le C_1H_\rel(\wh\bfu|U),
\end{align*}
for suitable constants $C_i=C_i(\PP)<\infty $, $i=0,1$.
Thus, we deduce~\eqref{eq:coerc.Sb} if additionally $E\ge \ee^{C_0+1}$.

The remaining estimates~\eqref{eq:Hrel.coerc.Sg},~\eqref{eq:Hrel.coerc.complSg} follow upon combining the bound~\eqref{eq:coerc.Sb} with Lemma~\ref{l:coerc.relen}.
\end{proof}

We introduce further auxiliary `relative'-type quantities
\begin{align}
	d_\rel(u|U)&:=\sum_{i=1}^n|\nabla \sqrt{u_i}-\frac{\sqrt{u_i}}{\sqrt{U_i}}\nabla\sqrt{U_i}|^2, \nonumber
	\\e_\rel(v|V)&:=\sum_{i=1}^n|\sqrt{v_i}-\sqrt{V_i}|^2,\label{eq:def.erel}
\end{align}
where here $u,U:\Om\to[0,\infty)^n$ denote measurable functions, $U_{|\Om_\sm}\in C^{0,1}(\Om_\sm)$ and $\inf U_i>0$, and $u_i$ are such that 
$\sqrt{u_i}$ possesses a weak derivative $\nabla\sqrt{u_i}_{|\Om_\sm}\in L^2(\Om_\sm)$ for every $i=1,\dots,n$. The expression $e_\rel(v|V)$ will be reserved for interfacial quantities $v,V:\Gamma\to[0,\infty)^n$. It is important, throughout this manuscript, to keep in mind that the weak spatial differentiability of $\sqrt{u_i}$ is only true after restriction to the subdomains $\Om_\pm$ and that the respective traces $\tr_{\Gamma_\pm}\sqrt{u_i}$ of $\sqrt{u_i}$ from $\Om_\pm$ to $\Gamma$ typically do not coincide.

The following interpolation estimate is an adaptation of~\cite[Lemma~10]{Fischer_2017}.
\begin{lemma}\label{l:interpol-trace.g}
Let $u,U:\Om\to[0,\infty)^n$ be as above. Then, 
	for every $\ve>0$ there exists $C_\ve<\infty$ such that 
	\begin{align*}
		\int_{\Gamma}e_\rel(u^\sigma|U^\sigma)\,\dd\mathcal{H}^{d-1}&\le 
		\ve \int_{\Om_\sm}d_\rel(u|U)\,\dd x+C_\ve H_\rel(\wh\bfu|U),\qquad\sm\in\{\pm\},
	\end{align*}
	where the integrand on the left-hand side is to be understood in the trace sense.
\end{lemma}
\begin{proof}
	Let $\ve>0$.	By the interpolation-trace theorem, there exists $C_1(\ve)<\infty$ such that 
	\begin{align*}
		\int_{\Gamma}e_\rel(u^\sigma|U^\sigma)\,\dd\mathcal{H}^{d-1}&\le
		\frac{\ve}{2}\int_{\Om_\sm}\sum_{i=1}^n|\nabla\sqrt{u_i}-\nabla\sqrt{U_i}|^2\,\dd x+C_{1}(\ve)\int_{\Om_\sm}\sum_{i=1}^n|\sqrt{u_i}-\sqrt{U_i}|^2\,\dd x
		\\&\le \ve\int_{\Om_\sm}d_\rel(u|U)\,\dd x
		+C_\ve H_\rel(\wh\bfu|U),
	\end{align*}
	where, in the second step, we used the triangle inequality to estimate
	\begin{align*}
		|\nabla\sqrt{u_i}-\nabla\sqrt{U_i}|^2\le 2|\nabla\sqrt{u_i}-\frac{\sqrt{u_i}}{\sqrt{U_i}}\nabla\sqrt{U_i}|^2
		+2\frac{|\nabla\sqrt{U_i}|^2}{U_i}|\sqrt{u_i}-\sqrt{U_i}|^2,
	\end{align*}
	as well as inequality~\eqref{eq:Hrel.coerc.Sg} from Lemma~\ref{l:coerc.relen}.
\end{proof}

\subsection{Stability estimate}

We abbreviate $\xi_j^*(u,\tilde u):=\xi^*(u,\tilde u)u_j$, $\zeta_j^*(u):=\zeta^*(u)u_j$, $j=1,\dots,n$ with $\xi^*,\zeta^*$ as specified in Section~\ref{ssec:adjrelen.coer}. 

\begin{lemma}[Evolution of relative entropy]\label{l:evol.rel-en.gen}
	Under the assumptions of Theorem~\ref{thm:unique}, for a.e.\ $0<T<T^*$ it holds that
	\begin{align}\label{eq:evol.relen}
		\begin{aligned}
			\eval{H_\rel(\wh\bfu|U)}_{t=0}^{t=T}	
			&\le\int_0^T\bigg(\int_\Om\rho_\text{\rm bulk}\,\dd x
			+\sum_{\sigma\in\{\pm\}}\int_{\Gamma_{\sigma}}\rho_{\text{\rm int},\sigma}\,\dd \mathcal{H}^{d-1}\bigg)\dd t
			\\&\qquad+
			\int_0^T\sum_{\sigma\in\{\pm\}}\bigg(\int_{\Om_{-\sigma}}\rho_{\text{\rm bulk,rem},\sigma}\,\dd x+\int_{\Gamma_{-\sigma}}\rho_{\text{\rm int,rem},\sigma}\,\dd\mathcal{H}^{d-1}\bigg)\dd t,
		\end{aligned}
	\end{align}
	where $	\eval{H_\rel(\wh\bfu|U)}_{t=0}^{t=T}:=H_\rel(\wh\bfu(T)|U(T))-H_\rel(\wh\bfu_0|U(0))$. Furthermore,
	\begin{align*}
		\rho_\text{\rm bulk}&=
		\sum_{i=1}^n\big(-4\DC_i\nabla \sqrt{u_i}\cdot\nabla\sqrt{u_i}+D_ih(u)f_i(u)+f_i(U)\big)
		\\&+\sum_{i,j=1}^n\bigg(
		\nabla( \vp_{\text{\rm out}} D_jh(U) D_i\zeta_j^*(u))\cdot A_i\nabla u_i
		+\nabla(\vp_{\rm out}D_{ij}h(U)\zeta_j^*(u))\cdot A_i\nabla U_i
		\\&\qquad\qquad-\vp_{\text{\rm out}}\big(D_jh(U) D_i\zeta_j^*(u)f_i(u)
		+D_{ij}h(U) \zeta_j^*(u)f_i(U)\big)
		\\&\qquad+\sum_{\beta\in \frakP}\Big(\nabla(\vp_\beta D_jh(U) D_i\xi_j^*(\wt\bfu^\beta))\cdot A_i\nabla u_i
		+\nabla(\vp_\beta D_{ij}h(U)\xi_j^*(\wt\bfu^\beta))\cdot A_i\nabla U_i
		\\&\qquad\qquad\qquad\qquad -\vp_\beta\big(D_jh(U) D_i\xi_j^*(\wt\bfu^\beta)f_i(u)+D_{ij}h(U) \xi_j^*(\wt\bfu^\beta)f_i(U)\big)\Big)\bigg),
		\\[2mm]\rho_{\text{\rm int},\sigma}&= -\sum_{i=1}^nD_ih(u) r_i^\sigma(\bfu)
		\\&\qquad+\sum_{i,j=1}^n\sum_{\beta\in\frakP}\vp_\beta\Big( D_jh(U)D_i\xi_j^*(\wt\bfu^\beta)r_i^\sigma(\bfu)
		+ D_{ij}h(U) \xi_j^*(\wt\bfu^\beta)r_i^\sm(\bfU)\Big)\qquad \text{(\,$\Gamma$-trace   from $\Om_\sm$)},
	\end{align*}
	and 
	\begin{align*}
		\rho_{\text{\rm bulk,rem},\sigma}&=	
		\sum_{i,j=1}^n\sum_{\beta\in\frakP}\bigg(
		\nabla \big((\vp_\beta D_jh(U)D_{i+n}\xi_j^*(\wt\bfu^\beta))\circ\Phi_{\beta,-\sigma}\big)
		\cdot A_i\nabla u_i
		\\&\hspace{7em}\qquad
		-(\vp_\beta D_jh(U)D_{i+n}\xi_j^*(\wt\bfu^\beta))\circ \Phi_{\beta,-\sigma}f_i(u)\bigg),
		\\	\rho_{\text{\rm int,rem},\sigma}&=\sum_{i,j=1}^n\sum_{\beta\in\frakP}
		(\vp_\beta D_jh(U)
		D_{i+n}\xi_j^*(\wt\bfu^\beta))\circ\Phi_{\beta,-\sigma}
		r_i^{-\sigma}(\bfu)\qquad \text{(\,$\Gamma$-trace from $\Om_{-\sm}$)},
	\end{align*}
\end{lemma}
\begin{proof}
	Recalling our short-hand $\zeta_j^*(u):=\zeta^*(u)u_j$ and $\xi_j^*(\wt\bfu^\beta):=\xi^*(\wt\bfu^\beta)u_j$,
	we write 
	\begin{align}\label{eq:relen.eval}
		\eval{H_\rel(\wh\bfu|U)}_{t=0}^{t=T}	 &=
		\eval{H(u)}_{t=0}^{t=T}
		-\eval{\int_{\Om}\,\sum_{j=1}^nD_jh(U)\big(\vp_{\text{\rm out}}\zeta^*_j(u)+
		\sum_{\beta\in\frakP}\vp_\beta\xi^*_j(\wt\bfu^\beta)\big)\dd x}_{t=0}^{t=T}
		+\eval{\int_\Om \sum_{j=1}^nU_j\,\dd x}_{t=0}^{t=T}.
	\end{align}
	The first term on the right-hand side is handled using the entropy dissipation inequality~\eqref{eq:edi}. 
	
	For the second term, we use the renormalised solution property in Definition~\ref{def:renormalised.gen}. To this end, we first observe that by an approximation argument, equations~\eqref{eq:def.renorm.out.gen} and~\eqref{eq:def.renorm.c.gen} remain valid for $\psi\in  C^{0,1}_c([0,\infty)\times G)$ with $G=\ol\Om_\sm\setminus\Gamma$, $\ol\Om_\sm\cap V_\beta,\beta\in\frakP,$ respectively (see e.g.~\cite[Remark~2.1]{Hopf_2022}).
	We  then treat each summand separately:
	for the term involving $\vp_{\text{\rm out}}$, we invoke~\eqref{eq:def.renorm.out.gen} with $\psi=-D_jh(U)\vp_{\text{\rm out}}$ and $\zeta=\zeta_j^*$, 
	while for the term involving $\vp_\beta$, $\beta\in\frakP$,
	  we use~\eqref{eq:def.renorm.c.gen} with $\psi=-D_jh(U)\vp_{\beta}$ and the choice $\xi=\xi_j^*$. This yields the following equations, where compactly supported integrands defined on subdomains are understood to be extended by zero to the whole of $\Om_\sigma:$

 For the outer region away from $\Gamma$, we have
	\begin{align}\label{eq:hrel.evol.renorm.out.v1}
		\begin{aligned}
			\eval{-\int_{\Om_\sigma} \sum_{j=1}^n\vp_{\text{\rm out}}D_jh(U)\zeta^*(u)u_j\,\dd x}_{t=0}^{t=T}
			\,=&-\int_0^T\!\!\int_{\Om_\sigma}\sum_{i,j=1}^n\vp_{\text{\rm out}}D_{ij}h(U)\pa_tU_i \zeta_j^*(u)\,\dd x\dd t
			\\&
			+\int_0^T\!\!\int_{\Om_\sigma}\sum_{i,j=1}^n\nabla( \vp_{\text{\rm out}} D_jh(U) D_i\zeta_j^*(u))\cdot A_i\nabla u_i
			\,\dd x\dd t
			\\&-\int_0^T\!\!\int_{\Om_\sigma}\sum_{i,j=1}^n\vp_{\text{\rm out}}D_jh(U) D_i\zeta_j^*(u)f_i(u)\,\dd x\dd t.
		\end{aligned}
	\end{align}	
	For the family $\{V_\beta\mid \beta\in\frakP\}$ covering $\Gamma$ we obtain
	\begin{align*}
		\eval{-\int_{\Om_\sigma}\sum_{j=1}^n \vp_\beta D_jh(U)\xi_j^*(\wt\bfu^\beta )\,\dd x}_{t=0}^{t=T}
		=&-\int_0^T\!\!\int_{\Om_\sigma}\sum_{i,j=1}^n\vp_\beta D_{ij}h(U)\pa_tU_i \xi_j^*(\wt\bfu^\beta)\,\dd x\dd t
		\\&
		+\int_0^T\!\!\int_{\Om_\sigma}\sum_{i,j=1}^n\nabla(\vp_\beta D_jh(U) D_i\xi_j^*(\wt\bfu^\beta))\cdot A_i\nabla u_i
		\,\dd x\dd t
		\\&+\int_0^T\!\!\int_{\Gamma_\sigma}\sum_{i,j=1}^n\vp_\beta D_jh(U)
		D_i\xi_j^*(\wt\bfu^\beta)r_i^\sigma(\bfu)\,\dd\mathcal{H}^{d-1}\dd t
		\\& -\int_0^T\!\!\int_{\Om_\sigma}\sum_{i,j=1}^n\vp_\beta D_jh(U) D_i\xi_j^*(\wt\bfu^\beta)f_i(u)\,\dd x\dd t
		\\&+\int_0^T\!\!\int_{\Om_{-\sigma}}\sum_{i,j=1}^n
		\nabla \big((\vp_\beta D_jh(U)D_{i+n}\xi_j^*(\wt\bfu^\beta))\circ\Phi_{\beta,-\sigma}\big)
		\cdot A_i\nabla u_i\,\dd x\dd t
		\\&+\int_0^T\!\!\int_{\Gamma_{-\sigma}}\sum_{i,j=1}^n(\vp_\beta D_jh(U)
		D_{i+n}\xi_j^*(\wt\bfu^\beta))\circ\Phi_{\beta,-\sigma}
		r_i^{-\sigma}(\bfu)\,\dd\mathcal{H}^{d-1}\dd t
		\\&-\int_0^T\!\!\int_{\Om_{-\sigma}}\sum_{i,j=1}^n (\vp_\beta D_jh(U)D_{i+n}\xi_j^*(\wt\bfu^\beta))\circ \Phi_{\beta,-\sigma}f_i(u) \,\dd x\dd t.
	\end{align*}
	We rewrite the first term on the right-hand side in each of the last two equations, using the weak formulation~\eqref{eq:strongsol.weak} for $U$ with test function $\psi=-\vp_{\text{\rm out}}D_{ij}h(U)\zeta_j^*(u)$ resp.\ $\psi=-\vp_\beta D_{ij}h(U) \xi_j^*(\wt\bfu^\beta)$, which both satisfy the required regularity $\psi\in L^1(0,T;W^{1,1}(\Om_\sm))$, and then take the sum over $i,j=1,\dots,n$.
	For the outer region, we observe
	\begin{align*}
		-\int_0^T\!\!\int_{\Om_\sigma}\sum_{i,j=1}^n\vp_{\text{\rm out}}D_{ij}h(U)\pa_tU_i \zeta_j^*(u)\,\dd x\dd t
		&=\int_0^T\!\!\int_{\Om_\sigma}\sum_{i,j=1}^n\nabla(\vp_{\rm out}D_{ij}h(U)\zeta_j^*(u))\cdot A_i\nabla U_i\,\dd x\dd t
		\\&\qquad-\int_0^T\!\!\int_{\Om_\sigma}\sum_{i,j=1}^n\vp_{\rm out}D_{ij}h(U) \zeta_j^*(u)f_i(U)\,\dd x\dd t.
	\end{align*}
	For $V_\beta$, $\beta\in\frakP$, we note that
	\begin{align*}
		-\int_0^T\!\!\int_{\Om_\sigma}\sum_{i,j=1}^n\vp_\beta D_{ij}h(U)\pa_tU_i \xi_j^*(\wt\bfu^\beta)\,\dd x\dd t
		&=\int_0^T\!\!\int_{\Om_\sigma}\sum_{i,j=1}^n\nabla(\vp_\beta D_{ij}h(U)\xi_j^*(\wt\bfu^\beta))\cdot A_i\nabla U_i\,\dd x\dd t
		\\&\qquad+\int_0^T\!\!\int_{\Gamma_\sigma}\sum_{i,j=1}^n\vp_\beta D_{ij}h(U) \xi_j^*(\wt\bfu^\beta)r_i^\sigma(\bfU)\,\dd \mathcal{H}^{d-1}\dd t
		\\&\qquad-\int_0^T\!\!\int_{\Om_\sigma}\sum_{i,j=1}^n\vp_\beta D_{ij}h(U) \xi_j^*(\wt\bfu^\beta)f_i(U)\,\dd x\dd t.
	\end{align*}
	For the last term on the right-hand side of~\eqref{eq:relen.eval}, we simply have
	\begin{align}
		\eval{\int_\Om \sum_{j=1}^nU_j\,\dd x}_{t=0}^{t=T}=\int_\Om\sum_{j=1}^nf_j(U)\,\dd x,
	\end{align}
	where we used the fact that $\sum_{\sigma\in\{\pm\}}r^\sigma_j\equiv0$ for all $j=1,\dots,n$ as a consequence of property~\eqref{eq:mass.pres}.
	
	Combining the entropy dissipation property~\eqref{eq:edi} with the identities above and rearranging terms yields the asserted inequality.
\end{proof}

\begin{proposition}[Stability estimate]\label{prop:stab.gen}
Let $\bfU$ be a strong solution in the sense of Definition~\ref{def:strong} with maximal lifespan $T^*\in(0,\infty]$ and let $\ol T\in(0,T^*)$.
 If the parameters $E,N$ in the definition of $\xi^*,\zeta^*$ are chosen large enough (depending on $\bfU$ and $\ol T$ through $\iota,\mathfrak J$), then there exists a function $C\in C([0,\ol T];\mathbb{R}_+)$ with $C(0)=1$ such that 
  for almost all $T\in[0,\ol T]$ and any dissipative renormalised solution $\bfu$ in the sense of Definition~\ref{def:renormalised.gen} the following stability estimate holds true
	\begin{align}\label{eq:stab.est.gen}
		H_\rel(\wh\bfu(T)|U(T))\le C(T)H_\rel(\wh\bfu_0|U_0).
	\end{align}
	In particular, if $u_0^\sigma\equiv U_0^\sigma$ for all $\sigma\in\{\pm\}$, it holds that 
	$u^\sigma\equiv U^\sigma$ $\mathcal{L}^{d+1}$-a.e.\ in $(0,T^*)\times\Om_\sigma$ for  $\sigma\in \{\pm\}$.
\end{proposition}
\begin{proof} Lemma~\ref{l:evol.rel-en.gen} reduces our task to estimating appropriately the terms on the right-hand side of~\eqref{eq:evol.relen} in way that Gronwall's lemma becomes applicable.
	We first treat separately the estimates of bulk and interface integrals (Steps 1,2) before deriving in Step~3 the bound~\eqref{eq:stab.est.gen}.	\smallskip
	
	\noindent{\bfseries\em Step~1: Bulk terms.}\, We assert the following estimate for the bulk integrals appearing on the right-hand side of~\eqref{eq:evol.relen}: 
		\begin{align}\label{eq:bulk.terms}
			\begin{aligned}
			\int_{\Om_T}\rho_{\rm bulk}\,\dd \mathcal{L}^{1+d}(t,x)+\sum_{\sigma\in\{\pm\}}
			\int_{(0,T)\times\Om_{-\sigma}} &\rho_{\text{\rm bulk,rem},\sigma}\,\mathcal{L}^{1+d}(t,x)
			\\&\le
			-\frac{\mathfrak{a}}{4}\int_{\Om_T}d_\rel(u|U)\,\dd \mathcal{L}^{d+1}(t,x)
			+C(E,N)\int_0^TH_\rel(\wh\bfu|U)\dd t,	 		
		\end{aligned}
	\end{align}	
	where $\Om_T:=(0,T)\times\Om$.
	
	We start with some preliminary reformulations of the integrands:
	gathering terms involving $\nabla\vp_{\rm out},\nabla\vp_\beta$ in an extra contribution~$\rho_\text{\rm bulk,b}$, we write
	\begin{align*}
		\rho_\text{\rm bulk}&=\rho_\text{\rm bulk,g}+\rho_\text{\rm bulk,b},
	\end{align*}
	where
	\begin{align}\label{eq:bulk.g}
	\begin{aligned}	\rho_\text{\rm bulk,g}=
		\sum_{i=1}^n&\big(-4\DC_i\nabla \sqrt{u_i}\cdot\nabla\sqrt{u_i}+D_ih(u)f_i(u)+f_i(U)\big)
		\\&+\sum_{i,j=1}^n\bigg(
		\vp_{\text{\rm out}} 	\nabla(D_jh(U) D_i\zeta_j^*(u))\cdot A_i\nabla u_i
		+\vp_{\rm out}\nabla(D_{ij}h(U)\zeta_j^*(u))\cdot A_i\nabla U_i
		\\&\qquad\qquad-\vp_{\text{\rm out}}\big(D_jh(U) D_i\zeta_j^*(u)f_i(u)
		+D_{ij}h(U) \zeta_j^*(u)f_i(U)\big)
		\\&\qquad\qquad+\sum_{\beta\in\frakP}\Big(\vp_\beta \nabla(D_jh(U) D_i\xi_j^*(\wt\bfu^\beta))\cdot A_i\nabla u_i
		+\vp_\beta \nabla(D_{ij}h(U)\xi_j^*(\wt\bfu^\beta))\cdot A_i\nabla U_i
		\\&\quad\qquad\qquad\qquad\qquad -\vp_\beta\big(D_jh(U) D_i\xi_j^*(\wt\bfu^\beta)f_i(u)+D_{ij}h(U) \xi_j^*(\wt\bfu^\beta)f_i(U)\big)\Big)\bigg)	
\end{aligned}
\end{align}
	and
	\begin{align}\label{eq:rho.bb.C}
		\begin{aligned}
			\rho_\text{\rm bulk,b}&	= \nabla\vp_{\text{\rm out}}\cdot \sum_{i,j=1}^n\big(
			D_jh(U) D_i\zeta_j^*(u) A_i\nabla u_i
			+D_{ij}h(U)\zeta_j^*(u)A_i\nabla U_i\big)
			\\&\qquad+\sum_{\beta\in\frakP}\nabla\vp_\beta\cdot
			\sum_{i,j=1}^n\big(		D_jh(U) D_i\xi_j^*(\wt\bfu^\beta) A_i\nabla u_i
			+D_{ij}h(U)\xi_j^*(\wt\bfu^\beta)A_i\nabla U_i\big).
		\end{aligned}
	\end{align}
	Using~\eqref{eq:partition1} and recalling $\zeta_j^*(u)=\zeta^*(u)u_j$, $\xi_j^*(u,\tilde u)=\xi^*(u,\tilde u)u_j$, equation~\eqref{eq:rho.bb.C} can be rewritten as
	\begin{align}\label{eq:rho.bb.A}
		\begin{aligned}
			\rho_\text{\rm bulk,b}	&= \nabla\vp_{\text{\rm out}}\cdot \sum_{i,j=1}^n\big(
			D_jh(U) (D_i\zeta_j^*(u)-\delta_{ij}) A_i\nabla u_i
			+D_{ij}h(U)(\zeta^*(u)-1)u_jA_i\nabla U_i\big)
			\\&\qquad+\sum_{\beta\in\frakP}\nabla\vp_\beta\cdot
			\sum_{i,j=1}^n\big(D_jh(U)(D_i\xi_j^*(\wt\bfu^\beta)-\delta_{ij}) A_i\nabla u_i
			+D_{ij}h(U)(\xi^*(\wt\bfu^\beta)-1)u_jA_i\nabla U_i\big).
		\end{aligned}
	\end{align}
	\smallskip
	
	\noindent{\em Step~1.1: For the constant $\mathfrak a>0$ in~\ref{it:HP.diff}, it holds that }
	\begin{align}\label{eq:rho.bulk,g}
		\int_{\Om_T}\rho_{\rm bulk,g}\dd\mathcal{L}^{1+d}(t,x)\le -\frac{\mathfrak{a}}{2}\int_{\Om_T}d_\rel(u|U)\,\dd\mathcal{L}^{1+d}(t,x)+ C(E,N)\int_0^TH_\rel(\wh\bfu|U)\dd t.
	\end{align}	 
{\em Proof of~\eqref{eq:rho.bulk,g}.}
We rewrite~\eqref{eq:bulk.g}, using~\eqref{eq:partition1}, in the form
	\begin{align*}
		\rho_{\rm bulk,g}=\sum_{\beta\in\frakP\cup\{{\rm out}\}}\vp_\beta	\rho_{\rm bulk,g}^\beta,
	\end{align*}
	where for $\beta\in\frakP$
	\begin{align*}
		\rho_{\rm bulk,g}^\beta=
		&\sum_{i=1}^n\big(-4\DC_i\nabla \sqrt{u_i}\cdot\nabla\sqrt{u_i}+D_ih(u)f_i(u)+f_i(U)\big)
		\\&+\sum_{i,j=1}^n\nabla(D_jh(U) D_i\xi_j^*(\wt\bfu^\beta))\cdot A_i\nabla u_i
		+\nabla(D_{ij}h(U)\xi_j^*(\wt\bfu^\beta))\cdot A_i\nabla U_i
		\\& -\sum_{i,j=1}^n\big(D_jh(U) D_i\xi_j^*(\wt\bfu^\beta)f_i(u)+D_{ij}h(U) \xi_j^*(\wt\bfu^\beta)f_i(U)\big),
	\end{align*}
	and 
	\begin{align*}
		\rho_{\rm bulk,g}^{\rm out}=
		&\sum_{i=1}^n\big(-4\DC_i\nabla \sqrt{u_i}\cdot\nabla\sqrt{u_i}+D_ih(u)f_i(u)+f_i(U)\big)
		\\&+\sum_{i,j=1}^n\nabla(D_jh(U) D_i\zeta_j^*(u))\cdot A_i\nabla u_i
		+\nabla(D_{ij}h(U)\zeta_j^*(u))\cdot A_i\nabla U_i
		\\& -\sum_{i,j=1}^n\big(D_jh(U) D_i\zeta_j^*(u)f_i(u)+D_{ij}h(U) \zeta_j^*(u)f_i(U)\big).
	\end{align*}
	In the following, we outline the bound for $\rho_{\mathrm{bulk},g}^{\beta}$, $\beta\in\mathfrak{P}$. 
	In our reasoning, whenever possible, we will argue in a pointwise fashion,
	to be understood for $\mathcal{L}^{1+d}$-a.e.\ $(t,x)\in(0,T)\times (V_\beta\cap\Om_\sm)$, $\sm\in\{\pm\}$  (or for a pointwise determined representative).
	Specifically, we will distinguish three cases 
	depending on the value $\wt\bfu^\beta(t,x)\in [0,\infty)^{2n}=\mathcal{A}\dot\cup\mathcal{B}\dot\cup\mathcal{C}$, where $\mathcal{A},\mathcal{B},\mathcal{C}$ are as specified in~\eqref{eq:regions}. 
	
	\smallskip
	\noindent{Case $\wt\bfu^\beta(t,x)\in \mathcal{A}:$}\nopagebreak
	
	We first observe that $D_i\xi_j^*=\delta_{ij}$ in $\mathcal{A}$. Using also the fact that $D_{ij}h(U)=\frac{1}{U_i}\delta_{ij}$, we infer for $\mathcal{L}^{1+d}$-a.a.\ $(t,x)\in(0,T)\times (V_\beta\cap\Om_\pm)$ with $\wt\bfu^\beta(t,x)\in\mathcal{A}$,  
	\begin{align*}
		\rho_{\rm bulk,g}^\beta&=
		\sum_{i=1}^n\big(-4\DC_i\nabla \sqrt{u_i}\cdot\nabla\sqrt{u_i}+\nabla D_ih(U)\cdot A_i\nabla u_i
		+\nabla(\frac{u_i}{U_i})\cdot A_i\nabla U_i\big)
		\\&\qquad+\sum_{i=1}^n\big(D_ih(u)f_i(u)+f_i(U)-D_ih(U)f_i(u)
		-\frac{u_i}{U_i}f_i(U)\big)
		\\&=	4\sum_{i=1}^n\big(-\DC_i\nabla \sqrt{u_i}\cdot\nabla\sqrt{u_i}+\frac{\sqrt{u_i}}{\sqrt{U_i}}\nabla \sqrt{U_i}\cdot A_i\nabla \sqrt{u_i}
		+\big(\frac{\sqrt{u_i}}{\sqrt{U_i}}\nabla\sqrt{u_i}-\frac{u_i}{U_i}\nabla\sqrt{U_i}
		\big)\cdot A_i\nabla \sqrt{U_i}
		\\&\qquad+\sum_{i=1}^n\big(D_ih(u)f_i(u)+f_i(U)-D_ih(U)f_i(u)
		-\frac{u_i}{U_i}f_i(U)\big)
		\\&\le-4\mathfrak{a}\sum_{i=1}^n|\nabla \sqrt{u_i}-\frac{\sqrt{u_i}}{\sqrt{U_i}}\nabla \sqrt{U_i}|^2+\sum_{i=1}^n\Big(\big(\log\frac{u_i}{U_i}-\frac{u_i}{U_i}+1\big)f_i(u)+\big(\frac{u_i}{U_i}-1\big)(f_i(u)-f_i(U))\Big)
		\\&\le-4\mathfrak{a}\,d_\rel(u|U)
		+C(E)\sum_{i=1}^n|\sqrt{u_i}-\sqrt{U_i}|^2.
	\end{align*}
	In the last step, to estimate the second term in the penultimate line, we distinguished two cases: \\s1: $u_i\ge\iota/2$ for all $i=1,\dots,n$.
	\\s2: $u_{i_0}<\iota/2$ for some $i_0\in\{1,\dots,n\}$.
	
	In case, s1, we used the estimate $|\log s-s+1|\lesssim_\iota |\sqrt{s}-1|^2$ for $s\ge\iota/2$, while in the second case, we rely on the lower bound $|\sqrt{u_{i_0}}-\sqrt{U_{i_0}}|\gtrsim_\iota 1$ and the fact that $|u|_1\le E$. 
	
	To estimate the third term in the penultimate line, we used the local Lipschitz continuity of the reaction rates and the bound 
	\begin{align}	\label{eq:diffuU}
	|u_i-U_i|=|(\sqrt{u_i}-\sqrt{U_i})(\sqrt{u_i}+\sqrt{U_i})|\lesssim_E|\sqrt{u_i}-\sqrt{U_i}|,
	\end{align}
	which is valid in the present case due to $|u|_1\le E$.
	
	We thus conclude, in $(0,T)\times (V_\beta\cap\Om_\sm)$, $\sm\in\{\pm\}$, the estimate 
	\begin{align}\label{eq:rho.g.A}
		\begin{aligned}
		\rho_{\rm bulk,g}^\beta\chi_{\{\wt\bfu^\beta\in\mathcal{A}\}}\le-4\mathfrak{a}\,d_\rel(u|U)\chi_{\{\wt\bfu^\beta\in\mathcal{A}\}}
		+C(E)\sum_{i=1}^n|\sqrt{u_i}-\sqrt{U_i}|^2\chi_{\{\wt\bfu^\beta\in\mathcal{A}\}\cap S_g^{\wh\bfu}}
		+C(E)\sum_{i=1}^n\chi_{\{\wt\bfu^\beta\in\mathcal{A}\}\setminus S_g^{\wh\bfu}},		
		\end{aligned}
	\end{align}	
 where for $\mathcal{Y}\subset\Om_T$ measurable, $\chi_{\mathcal{Y}}=\chi_{\mathcal{Y}}(t,x)$ denotes the characteristic function of $\mathcal{Y}$, i.e.\
\[\chi_{\mathcal{Y}}(t,x)=\begin{cases}
	1&\text{if }(t,x)\in \mathcal{Y},
	\\0&\text{if }(t,x)\in\Om_T\setminus\mathcal{Y},
\end{cases}\]
and where the last term in~\eqref{eq:rho.g.A} arises from the rough bound $|\sqrt{u_i}-\sqrt{U_i}|^2\le C(E)$ for $|u|_1\le E$.
	\smallskip
	
	\noindent{Case $\wt\bfu^\beta(t,x)\in \mathcal{B}:$}\nopagebreak
	
We estimate, using~\ref{it:HP.diff}, the fact that $\sum_{i=1}^n(D_ih(u)f_i(u)+f_i(U))\le C$, and  $\xi^*_j(\wt\bfu^\beta)=u_j\xi^*(\wt\bfu^\beta)$,
	\begin{align}\label{eq:rho.bulk.g.beta}
	\rho_{\rm bulk,g}^\beta\le -4\mathfrak{a}\sum_{i=1}^n|\nabla\sqrt{u_i}|^2 
	+C\sum_{i=1}^n\big(|\nabla\xi^*(\wt\bfu^\beta)|+|u||\nabla D_i\xi^*(\wt\bfu^\beta)|\big)|\nabla u_i|
+	C(E,N)(|\nabla u|+|\nabla\tilde u^\beta|+1).
\end{align}
As in~\cite{Fischer_2017,Hopf_2022}, the key point to cope with the non-negative terms on the right-hand side that are quadratic in the gradient of the renormalised solution is to use the specific decay properties~\eqref{eq:decay.xi} of $\xi^*$, which imply the bound
\begin{align*}
\sum_{i=1}^n\big(|\nabla\xi^*(\wt\bfu^\beta)|+|u||\nabla D_i\xi^*(\wt\bfu^\beta)|\big)|\nabla u_i|
\le \frac{C}{N}\sum_{i=1}^n\big(|\nabla\sqrt{u_i}|^2+|\nabla\sqrt{\tilde u_i^\beta}|^2\big).
\end{align*}
For $N$ large enough, the first term on the right-hand side can immediately be absorbed by the diffusive dissipation, i.e., by the first term on the right-hand side of~\eqref{eq:rho.bulk.g.beta}.
 For the term involving the reflected densities $\tilde u_i^\beta$, absorption will only be possible later at the level of spatial integrals over $\Om$, see~\eqref{eq:bulk.g.integral} below. 

We now combine the estimates above with the following elementary relations, valid due to $|u|_1\le E^N$,
 \begin{align*}
	\sum_{i=1}^n|\nabla\sqrt{u_i}|^2\le 2d_\rel(u|U)+C(E,N)\quad\text{and}\quad 
	d_\rel(u|U)\le 2\sum_{i=1}^n|\nabla\sqrt{u_i}|^2+C(E,N),
\end{align*}
and a change of variables in the term involving $\nabla\sqrt{\tilde u_i^\beta}$, recalling also the Lipschitz bound $|D\Phi_{\beta}|\le C$ for $\Phi_{\beta}=\Phi_\beta^{-1}$ as in~\eqref{eq:def.Phi.beta},
to deduce, for $N$ large enough,
\begin{align}\label{eq:rho.g.B}
	\rho_{\rm bulk,g}^\beta&\le -\mathfrak{a}\,d_\rel(u|U)
	+\frac{C}{N}d_\rel(u|U)\circ\Phi_{\beta}+C(E,N)\qquad \text{ in }\{\wt\bfu^\beta\in \mathcal{B}\},
\end{align}
where, as before,  the constant $C$ is independent of $E,N$, but may change from line to line.
\smallskip
	
	\noindent{Case $\wt\bfu^\beta(t,x)\in \mathcal{C}:$}\nopagebreak
	
	We observe that $\xi^*\equiv0$ and hence $\xi_j^*\equiv0$ in $\mathcal{C}$. Therefore,
	\begin{align}\label{eq:rho.g.C}\begin{aligned}
			\rho_{\rm bulk,g}^\beta&=
			\sum_{i=1}^n\big(-4\DC_i\nabla \sqrt{u_i}\cdot\nabla\sqrt{u_i}+D_ih(u)f_i(u)+f_i(U)\big)
			\\&\le -2\mathfrak{a}\,d_\rel(u|U)+C\qquad \text{ in }\{\wt\bfu^\beta\in \mathcal{C}\}.
		\end{aligned}
	\end{align}
	\smallskip
	
To unify the three cases above, we first recall the partition $\mathcal{A}\dot\cup\mathcal{B}\dot\cup\mathcal{C}=[0,\infty)^{2n}$.  We further observe that, thanks to~\eqref{eq:xi.less.1}, 
	$\chi(x,\wh\bfu(t,x))<1$ for $\mathcal{L}^{1+d}$-a.a.\ 
	$(t,x)\in W_T^\beta:=\big((0,T)\times  \{\vp_\beta>0\}\big)\bigcap \{\wt\bfu^\beta \in \mathcal{B}\cup\mathcal{C}\}$, 
and hence, $\mathcal{L}^{1+d}(W_T^\beta \cap S_g^{\wh\bfu})=0$, where here and below \[S_g^{\wh\bfu}:=\{(t,x)\in\Om_T\mid \chi(x,\wh\bfu(t,x))=1\}=\dot\bigcup_{t\in(0,T)}\{t\}\times S_g^{\wh\bfu(t)}.\]
Thus, combining estimates~\eqref{eq:rho.g.A},~\eqref{eq:rho.g.B}, and~\eqref{eq:rho.g.C}, we deduce $\mathcal{L}^{1+d}$-a.e.\ in $\Om_T$,
	 \begin{align}\label{eq:bulk.g.beta}
	 	\begin{aligned}
	 	\vp_\beta\rho_{\rm bulk,g}^\beta\le &\vp_\beta\Big(-\mathfrak{a}d_\rel(u|U)
	 	+\frac{C}{N}d_\rel(u|U)\circ \Phi_{\beta}
	 	\\&\qquad
	 	+C(E,N)\big(\sum_{i=1}^n|\sqrt{u_i}-\sqrt{U_i}|^2\chi_{S_g^{\wh\bfu}}+\chi_{\Om_T\setminus S_g^{\wh\bfu}}\big)\Big).
	 \end{aligned}
 \end{align} 
	 
	For	$\rho_{\mathrm{bulk},g}^{\mathrm{out}}$ we obtain similar estimates by distinguishing the cases $|u|_1\le E$, $E<|u|_1\le E^N$, $|u|_1>E^N$, and exploiting~\eqref{eq:zeta.less.1}, resulting in the bound
	\begin{align}\label{eq:bulk.g.out}
		\vp_{{\rm out}}\rho_{\mathrm{bulk},g}^{\mathrm{out}}\le \vp_{{\rm out}}\Big(-\mathfrak{a}d_\rel(u|U)+C(E,N)\big(\sum_{i=1}^n|\sqrt{u_i}-\sqrt{U_i}|^2\chi_{S_g^{\wh\bfu}}+\chi_{\Om_T\setminus S_g^{\wh\bfu}}\big)\Big).
	\end{align}
	Combining~\eqref{eq:bulk.g.beta} and~\eqref{eq:bulk.g.out}, we infer, upon possibly further increasing $N$,
	\begin{align}\label{eq:bulk.g.integral}
	&	\int_{\Om_T}\sum_{\beta\in\frakP\cup\{{\rm out}\}}\vp_\beta\rho_{\rm bulk,g}^\beta\dd\mathcal{L}^{1+d}(t,x)
			\\&\le -\frac{\mathfrak a}{2}\int_{\Om_T}d_\rel(u|U)\dd\mathcal{L}^{1+1}(t,x)
			+C(E,N)\int_{S_g^{\wh\bfu}}\sum_{i=1}^n|\sqrt{u_i}-\sqrt{U_i}|^2\,\dd\mathcal{L}^{d+1}(t,x)
			+C(E,N)\mathcal{L}^{1+d}(\Om_T\setminus S_g^{\wh\bfu})\nonumber
		\\&	\le -\frac{\mathfrak{a}}{2}\int_{\Om_T}d_\rel(u|U)\,\dd\mathcal{L}^{1+d}(t,x)
			+C(E,N)\int_0^T H_\rel(\wh\bfu|U)\dd t,\nonumber
	\end{align}
	where in the last step we used the coercivity estimates~\eqref{eq:Hrel.coerc.Sg},~\eqref{eq:Hrel.coerc.complSg} pointwise in time to  deduce, after an appeal to Fubini's theorem, the bound
	$\mathcal{L}^{1+d}(\Om_T\setminus S_g^{\wh\bfu})=\int_0^T\mathcal{L}^d(\Om\setminus S_g^{\wh\bfu(t)})\,\dd t\lesssim \int_0^TH_\rel(\wh\bfu(t)|U(t))\,\dd t$. 
	
		This concludes Step~1.1.
	\smallskip
	
	\noindent{\em Step~1.2: Estimate of $\rho_{\rm bulk,b}$}	
	
	Observe that $\{\nabla\vp_\beta\neq0\}\subseteq\{\vp_\beta>0\}$ for all $\beta\in\frakP$. Hence, 
	$\mathcal{L}^{1+d}$-a.e.\ in $S_g^{\wh\bfu}\bigcap\big((0,T)\times\{\nabla\vp_\beta\neq0\}\big)$ it holds that $\xi^*(\wt\bfu^\beta)=1$ and $D_i\xi_j^*(\wt\bfu^\beta)=\delta_{ij}$. Similarly, $\mathcal{L}^{1+d}$-a.e.\ in $S_g^{\wh\bfu}\bigcap\big((0,T)\times\{\nabla\vp_{{\textrm{out}}}\neq0\}\big)$ it holds that $\zeta^*(u)=1$ and $D_i\zeta_j^*(u)=\delta_{ij}$. 
	Hence, from~\eqref{eq:rho.bb.A} we infer $\rho_{\rm bulk,b}=0$ in $S_g^{\wh\bfu}$.
	We further notice that formula~\eqref{eq:rho.bb.C} implies that $\rho_{\rm bulk,b}=0$  in $\{|u|_1\ge E^N\}$, since then $\zeta^*(u)=0$ and $\xi^*(\wt\bfu^\beta)=0$.
	Thus, an inequality of the form 
	\begin{align*}
		\rho_{\rm bulk,b}\le \ve d_\rel(u|U)+C_\ve(E,N)\chi_{\Om_T\setminus S_g^{\wh\bfu}}
	\end{align*}
	for any $\ve\in(0,1]$ can be obtained by straightforward estimates.
	(In the present case, the specific decay~\eqref{eq:decay.trunc} is not needed.)
	Integrating over $\Om_T$ and using~\eqref{eq:Hrel.coerc.complSg}, we infer 
	\begin{align}\label{eq:rhobb.basic}
		\int_{\Om_T}\rho_{\rm bulk,b}\,\dd\mathcal{L}^{1+d}(t,x)
		\le \ve\int_{\Om_T} d_\rel(u|U)\,\dd\mathcal{L}^{1+d}(t,x)+C_{1,\ve}(E,N)\int_0^TH_\rel(\wh\bfu|U)\dd t.
	\end{align}
	\smallskip
	
	\noindent{\em Step~1.3: Remainder bulk terms}
	
We next turn to the remaining bulk integrand $\rho_{\text{\rm bulk,rem},\sigma}$, defined on $(0,T)\times\Om_{-\sm}$ and given by
	\begin{align*}
		\rho_{\text{\rm bulk,rem},\sigma}&=	
		\sum_{i,j=1}^n\sum_{\beta\in\frakP}\bigg(
		\nabla \big((\vp_\beta D_jh(U)D_{i+n}\xi_j^*(\wt\bfu^\beta))\circ\Phi_{\beta,-\sigma}\big)
		\cdot A_i\nabla u_i
		\\&\hspace{7em}\qquad
		-(\vp_\beta D_jh(U)D_{i+n}\xi_j^*(\wt\bfu^\beta))\circ \Phi_{\beta,-\sigma}f_i(u)\bigg)
		=:\sum_{\beta\in\frakP}\rho_{b,r,\sm,\beta},
	\end{align*}
	where 
	\[\rho_{b,r,\sm,\beta}:=
	\sum_{i,j=1}^n 	\nabla \big((\vp_\beta D_jh(U)D_{i+n}\xi_j^*(\wt\bfu^\beta))\circ\Phi_{\beta,-\sigma}\big)
	\cdot A_i\nabla u_i-(\vp_\beta D_jh(U)D_{i+n}\xi_j^*(\wt\bfu^\beta))\circ \Phi_{\beta,-\sigma}f_i(u).\]
	Thus, it suffices to estimate each of the integrals
	\begin{align*}
	\int_{(0,T)\times(\Om_{-\sm}\cap V_\beta)}\rho_{b,r,\sm,\beta}\,\dd\mathcal{L}^{1+d}(t,x)
	=	\int_{(0,T)\times(\Om_{\sm}\cap V_\beta)}\rho_{b,r,\sm,\beta}\circ\Phi_{\beta}\,\dd\mathcal{L}^{1+d}(t,x),
	\end{align*}
	where we changed variables via $\Phi_{\beta}$ using $|\det D\Phi_{\beta}|=1$.
	The integrand on the right-hand side can be written as
	\begin{align*}
	\rho_{b,r,\sm,\beta}\circ\Phi_{\beta}&=
	\sum_{i,j=1}^n 	\Big(\nabla (\vp_\beta D_jh(U)D_{i+n}\xi_j^*(\wt\bfu^\beta))
	\cdot \tilde A_i\nabla \tilde u_i^\beta
	-(\vp_\beta D_jh(U)D_{i+n}\xi_j^*(\wt\bfu^\beta))f_i(\tilde u^\beta)\Big),
	\end{align*}
	where $\tilde A_i :=\big(D\Phi_\beta \,A_i(D\Phi_{\beta})^T\big)\circ \Phi_{\beta}^{-1}$.
	We assert that \[\rho_{b,r,\sm,\beta}\circ\Phi_{\beta}=0\qquad\mathcal{L}^{1+d}\text{-a.e.\ in }S_g^{\wh\bfu}\cap\big((0,T)\times(\Om_{\sm}\cap V_\beta)\big).\]
	The assertion follows by arguing similarly as in Step~1.2: we observe that in $S_g^{\wh\bfu}\bigcup\big((0,T)\times\{\vp_\beta>0\}\big)$ it must hold that $\xi^*(\wt\bfu^\beta)=1$ and hence, by~\eqref{eq:xi.less.1}, $\wt\bfu^\beta\in\mathcal{A}$. Owing to $D_{i+n}\xi_j^*\equiv0$ in $\mathcal{A}$, the assertion follows.
	In the complementary case $\big((0,T)\times(\Om_{\sm}\cap V_\beta)\big)\setminus S_g^{\wh\bfu}$,
	we use estimates similar to those leading to~\eqref{eq:rho.g.B}, \eqref{eq:rho.g.C} and take advantage of the decay~\eqref{eq:decay.trunc} to deduce the rough bound
	\begin{align*}
		\int_{(0,T)\times(\Om_{\sm}\cap V_\beta)}\rho_{b,r,\sm,\beta}\circ\Phi_{\beta}\,\dd x\dd t
&\le \frac{C
		}{N}\int_{\Om_T}d_\rel(u|U)\,\dd\mathcal{L}^{1+d}(t,x)+ C(E,N)\,\int_0^T\mathcal{L}^{d}(\Om\setminus S_g^{\wh\bfu(t)})\,\dd t
		\\&\le \frac{C
		}{N}\int_{\Om_T}d_\rel(u|U)\,\dd\mathcal{L}^{1+d}(t,x)+ C(E,N)\,\int_0^TH_\rel(\wh\bfu|U)\dd t,
	\end{align*}
		where the second step follows from~\eqref{eq:Hrel.coerc.complSg}.
	Since this is true for all $\beta\in\frakP$ and $\sm\in\{\pm\}$, we arrive at
		\begin{align*}
	\sum_{\sigma\in\{\pm\}}\int_{(0,T)\times\Om_{-\sigma}} \rho_{\text{\rm bulk,rem},\sigma}\,\mathcal{L}^{1+d}(t,x)\le \frac{C_0
	}{N}\int_{\Om_T}d_\rel(u|U)\,\dd\mathcal{L}^{1+d}(t,x)+C_1(E,N)\,\int_0^TH_\rel(\wh\bfu|U)\dd t.
	\end{align*}

\smallskip
	
	Combining Steps~1.1--1.3 and choosing $\ve$ in~\eqref{eq:rhobb.basic} small enough such that $\ve\le\frac{\mathfrak{a}}{8}$ and 
	$N$ sufficiently large so that $\frac{C_0}{N}\le \frac{\mathfrak{a}}{8}$ 
	yields the bound~\eqref{eq:bulk.terms}.
\smallskip
	
\noindent{\bfseries\em Step~2: Interface terms.}\,  The purpose of this step is to derive the bound
\begin{align}\label{eq:int-est.all}
	\int_0^T\sum_{\sigma\in\{\pm\}} \Big(\int_{\Gamma_{\sigma}}\rho_{\text{\rm int},\sigma}\,\dd \mathcal{H}^{d-1}+\int_{\Gamma_{-\sigma}}\rho_{\text{\rm int,rem},\sigma}\,\dd\mathcal{H}^{d-1}\bigg)\dd t
	\le C_{*}(E,N)\int_0^T\sum_{\sigma\in\{\pm\}} \int_{\Gamma} e_\rel(u^\sigma|U^\sigma)\,\dd \mathcal{H}^{d-1}\dd t,
\end{align}
where $e_\rel$ is given by~\eqref{eq:def.erel}, and $\rho_{\text{\rm int},\sigma}$, $\rho_{\text{\rm int,rem},\sigma}$ are defined in Lemma~\ref{l:evol.rel-en.gen}.

	We first rewrite the integrands $\rho_{\text{\rm int},\sigma}$ and $\rho_{\text{\rm int,rem},\sigma}$, $\sm\in\{\pm\}$, using identity~\eqref{eq:partition1} and the fact that $\supp\vp_{\rm out}\cap\Gamma=\emptyset$,
	\begin{align}\label{eq:exp.beta.int}
		\rho_{\text{\rm int},\sigma}=\sum_{\beta\in\frakP}\vp_\beta \rho_{\text{\rm int},\sigma}^\beta \quad\text{($\Gamma_\sm$-trace)},\qquad\quad
			\rho_{\text{\rm int,rem},\sigma}=\sum_\beta\vp_\beta\circ \Phi_{\beta,-\sigma} \rho_{\text{\rm int,rem},\sigma}^\beta\quad \text{($\Gamma_{-\sm}$-trace)},
	\end{align}
	where
	\begin{align*}
		\rho_{\text{\rm int},\sigma}^\beta
		&=-\sum_{i=1}^nD_ih(u^\sm) r_i^\sigma(\bfu)
		+\sum_{i,j=1}^n\Big( D_jh(U^\sm)D_i\xi_j^*(\wt\bfu^\beta)r_i^\sigma(\bfu)
		+ D_{ij}h(U^\sm) \xi_j^*(\wt\bfu^\beta)r_i^\sigma(\bfU)\Big),
		\\\rho_{\text{\rm int,rem},\sigma}^\beta&=\sum_{i,j=1}^n
		(D_jh(U^\sm)
		D_{i+n}\xi_j^*(\wt\bfu^\beta))\circ\Phi_{\beta,-\sigma}
		r_i^{-\sigma}(\bfu).
	\end{align*}
	Note that here and in the rest of this step, by adding the superindex $\sm\in\{\pm\}$ in $u^\sm,U^\sm$ we explicitly indicate the compartment where the $\Gamma$-trace is taken.
	We now assert that, for each $\beta\in\frakP$, it holds,  $\mathcal{L}^1\otimes\mathcal{H}^{d-1}$-a.e.\ in $(0,T)\times(\Gamma\cap V_\beta)$,
\begin{subequations}\label{eq:int.bound.both}
	\begin{align}
		\label{eq:int.bound.main}
		&\sum_{\sigma\in\{\pm\}}\rho_{\text{\rm int},\sigma}^\beta
		\lesssim_{E,N} \sum_{\sigma\in\{\pm\}}\sum_{i=1}^n|\sqrt{u_i^\sigma}-\sqrt{U_i^\sigma}|^2,
		\\&\label{eq:int.bound.rem}	\sum_{\sigma\in\{\pm\}}|\rho_{\text{\rm int,rem},\sigma}^\beta|\lesssim_{E,N} 
			\sum_{\sigma'\in\{\pm\}}\sum_{i=1}^n|\sqrt{u_i^{\sigma'}}-\sqrt{U_i^{\sigma'}}|^2.
	\end{align}	
	\end{subequations}
	{\em Proof of~\eqref{eq:int.bound.main}.}
	We can adopt the same case distinction as for the bulk terms. Indeed, by construction of the extension $\tilde u^\beta$, it holds that
	\begin{align}
		|\tr_{|\Gamma_\sigma\cap V_\beta}\wt\bfu^\beta|_1=|(\tr_{|\Gamma_\sigma\cap V_\beta}u, \tr_{|\Gamma_{-\sigma}\cap V_\beta}u)|_1=|\tr_{|\Gamma_{-\sigma}\cap V_\beta}\wt\bfu^\beta|_1\quad\text{ $\mathcal{L}^1\otimes\mathcal{H}^{d-1}$-a.e.\ in }(0,T)\times(\Gamma\cap V_\beta).
	\end{align}
	Hence, the condition $\tr_{\Gamma_\sigma}\wt\bfu^\beta(t,z)\in \mathcal{R}$ for $\mathcal{R}\in \{\mathcal{A},\mathcal{B},\mathcal{C}\}$ is independent of $\sigma$ and thus
	well-defined for $\mathcal{L}^1\otimes\mathcal{H}^{d-1}$-a.a.\ $(t,z)\in(0,T)\times(\Gamma\cap V_\beta)$.
	\smallskip
	
	\noindent{Case $\tr_{\Gamma_\pm}\wt \bfu^\beta(t,z)\in \mathcal{A}$:}\nopagebreak
	
	Using $D_i\xi_j^*(\wt\bfu^\beta)=\delta_{ij}$ and $\xi^*_j(\wt\bfu^\beta)=u_j$ for $\wt\bfu^\beta\in\mathcal{A}$, we have 
	\begin{align*}
		\sum_{\sigma\in\{\pm\}}\rho_{\text{\rm int},\sigma}^\beta&=- \sum_{\sigma\in\{\pm\}}\sum_{i=1}^n\big(D_ih(u^\sm) r_i^\sigma(\bfu)
		-D_ih(U^\sm)r_i^\sigma(\bfu)-D_{ij}h(U^\sm)u_j^\sm r_i^\sigma(\bfU)\big)
		\\&=-\sum_{i=1}^n\bigg(\Big(\log \frac{u_i^\sigma}{U_i^\sigma}- \log \frac{u_i^{-\sigma}}{U_i^{-\sigma}}\Big)r_i^\sigma(\bfu)
		-\Big(\frac{u_i^\sigma}{U_i^\sigma}-\frac{u_i^{-\sigma}}{U_i^{-\sigma}}\Big)r_i^\sigma(\bfU)\bigg)
		\\&=- \sum_{i=1}^n\Big(\log \frac{u_i^\sigma}{U_i^\sigma}-\frac{u_i^\sigma}{U_i^\sigma}+1- \big(\log \frac{u_i^{-\sigma}}{U_i^{-\sigma}}-\frac{u_i^{-\sigma}}{U_i^{-\sigma}}+1\big)\Big)r_i^\sigma(\bfu)
		-\sum_{i=1}^n\Big(\frac{u_i^\sigma}{U_i^\sigma}-\frac{u_i^{-\sigma}}{U_i^{-\sigma}}\Big)(r_i^\sigma(\bfu)-r_i^\sigma(\bfU))
		\\&=:I+II.
	\end{align*}
	To estimate the first term in the last line, we distinguish once more the two subcases: 
	\\s1: $u_i^\sigma\ge\iota/2$ for all $i=1,\dots,n$, $\sigma\in\{\pm\}$.
	\\s2: $u_{i_0}^{\sigma_0}<\iota/2$ for some $i_0\in\{1,\dots,n\}$ and some $\sigma_0\in\{\pm\}$.
	In case s1, we use again the bound $|\log s-s+1|\lesssim_\iota |\sqrt{s}-1|^2$) to deduce 
	\begin{align}
		I\lesssim_E \sum_{\sigma\in\{\pm\}}\sum_{i=1}^n|\sqrt{u_i^\sigma}-\sqrt{U_i^\sigma}|^2.
	\end{align}
	In the subcase s2, we estimate, using~\eqref{eq:R.diss},
	\begin{align*}
		I\le \underbrace{
			-\sum_{\sigma\in\{\pm\}}\sum_{i=1}^nD_ih(u^\sm) r_i^\sigma(\bfu)
			}_{\le0}+C(E)
		\lesssim_E\sum_{\sigma\in\{\pm\}}\sum_{i=1}^n|\sqrt{u_i^\sigma}-\sqrt{U_i^\sigma}|^2,
	\end{align*}
	where the second step follows from the coercivity property
	\begin{align*}
		\sum_{\sigma\in\{\pm\}}\sum_{i=1}^n|\sqrt{u_i^\sigma}-\sqrt{U_i^\sigma}|^2\ge c>0,
	\end{align*}
	valid for $\mathcal{L}^1\otimes\mathcal{H}^{d-1}$-a.e.\ $(t,z)\in(0,T)\times\Gamma$ for which s2 is satisfied.
	
	For the term $II$, we use the local Lipschitz regularity of $r_i^\sm$ and argue as in~\eqref{eq:diffuU} to deduce
	\begin{align*}
		II\lesssim_E\sum_{\sigma\in\{\pm\}}\sum_{i=1}^n|\sqrt{u_i^\sigma}-\sqrt{U_i^\sigma}|^2.
	\end{align*}
	
	\smallskip
	
	\noindent{Case $\tr_{\Gamma_\pm}\wt\bfu^\beta(t,z)\in \mathcal{B}:$} \nopagebreak
	
	In this case, for $E$ large enough (depending on $\|\bfU\|_{L^\infty_{t,x}}$), we have the coercivity property
	\begin{align}\label{eq:coerc.int}
		\sum_{\sigma\in\{\pm\}} \sum_{i=1}^n|\sqrt{u_i^\sigma}-\sqrt{U_i^\sigma}|^2\gtrsim 1.
	\end{align}
	Thus,
	\begin{align*}	
		\sum_{\sigma\in\{\pm\}}\rho_{{\rm int},\sigma}^\beta&=\underbrace{-\sum_{\sigma\in\{\pm\}}\sum_{i=1}^nD_ih(u^\sm) r_i^\sigma(\bfu)}_{\le0}
		+\sum_{\sigma\in\{\pm\}} \sum_{i,j=1}^n\big(D_jh(U^\sm)D_i\xi_j^*(\wt \bfu^\beta)r_i^\sigma(\bfu)
		+D_{ij}h(U^\sm)\xi_j^*(\wt \bfu^\beta)r_i^\sigma(\bfU)\big)
		\\&\lesssim_{E,N}\sum_{\sigma\in\{\pm\}} \sum_{i=1}^n|\sqrt{u_i^\sigma}-\sqrt{U_i^\sigma}|^2.
	\end{align*}
	
	\smallskip
	
	\noindent{Case $\tr_{\Gamma_\pm}\wt\bfu^\beta(t,z)\in \mathcal{C}:$}\nopagebreak
	
	We observe that $\xi^*\equiv0$ and hence $\xi_j^*\equiv0$ in $\mathcal{C}$. Therefore, 
	\begin{align*}
		\sum_{\sigma\in\{\pm\}} \rho_{\text{\rm int},\sigma}^\beta=
	-\sum_{\sigma\in\{\pm\}}\sum_{i=1}^nD_ih(u^\sigma) r_i^\sigma(\bfu)\le0.
	\end{align*}
	Combining the cases above implies~\eqref{eq:int.bound.main}.
	
		{\em Proof of~\eqref{eq:int.bound.rem}.} 
	
	By construction, $|\wt\bfu^\beta\circ\Phi_{\beta,-\sigma}|_1=|\wt\bfu^\beta|_1$. Thus, if
	$|\wt\bfu^\beta|_1\le E$ or $|\wt\bfu^\beta|_1\ge E^N$, then 
	$D_{i+n}\xi_j^*(\wt\bfu^\beta\circ\Phi_{\beta,-\sigma})=0$, and hence
	$\rho_{\text{\rm int,rem},\sigma}^\beta=0$.
	In the remaining case
	 $\tr_{\Gamma_\pm}\wt\bfu^\beta(t,z)\in \mathcal{B}$, 
	 the fact that $\big|(u^{\sigma}_{|_{\Gamma}},u^{-\sigma}_{|_{\Gamma}})\big|_1=|\bfu_{|\Gamma}|_1$ and the coercivity~\eqref{eq:coerc.int} yield 
	\begin{align*}
		|\rho_{\text{\rm int,rem},\sigma}^\beta|\le C\sum_{i,j=1}^n|D_{i+n}\xi_j(u^{\sigma}_{|_{\Gamma\cap V_\beta}},u^{-\sigma}_{|_{\Gamma\cap V_\beta}})r_i^{-\sigma}(\bfu)|
		\lesssim_{E,N} \sum_{\sigma'\in\{\pm\}}\sum_{i=1}^n|\sqrt{u_i^{\sigma'}}-\sqrt{U_i^{\sigma'}}|^2,
	\end{align*}
	which establishes~\eqref{eq:int.bound.rem}.

Taking the sum over $\sm\in\{\pm\}$ in~\eqref{eq:exp.beta.int} and 
inserting the pointwise bounds~\eqref{eq:int.bound.both}, we deduce~\eqref{eq:int-est.all} upon integration over $z\in\Gamma$ and $t\in(0,T)$.
	\smallskip
	
	\noindent {\bfseries\em Step~3: Conclusion.}\,
	We use Lemma~\ref{l:interpol-trace.g} with $\ve'=\ve/C_*(E,N)$ to bound the right-hand side of~\eqref{eq:int-est.all} above by 
	\begin{align*}
		\ve\int_0^T\sum_{\sm\in\{\pm\}}\int_{\Om_\sm}d_\rel(u|U)\,\dd x\dd t+C_{\ve}(E,N) \int_0^T H_\rel(\wh\bfu|U)\,\dd t.
	\end{align*}
Invoking Lemma~\ref{l:evol.rel-en.gen} and inserting Step 1 and the upper bound on the left-hand side of~\eqref{eq:int-est.all} derived in the previous line, we deduce for $\ve=\mathfrak{a}/8$ 
	\begin{align*}
		\eval{H_\rel(\wh\bfu|U)}_{t=0}^{t=T}	
		\le -\frac{\mathfrak{a}}{8}\int_0^T\sum_{\sm\in\{\pm\}}\int_{\Om_\sm}d_\rel(u|U)\,\dd x\dd t
		+C_{2}(E,N)\int_0^TH_\rel(\wh\bfu|U)\dd t.
	\end{align*}	
	Hence, for a.e.\ $0<T<T^*$,
	\[H_\rel(\wh\bfu(T)|U(T))\le H_\rel(\wh\bfu_0|U_0)+C_{2}(E,N)\int_0^TH_\rel(\wh\bfu|U)\dd t.\] 
	Appealing to Gronwall's lemma yields the asserted estimate~\eqref{eq:stab.est.gen} with $C(T)=\ee^{C_{2}(E,N)T}$. 
	
	This completes the proof of Proposition~\ref{prop:stab.gen}, which entails Theorem~\ref{thm:unique}.
\end{proof}	

\section{Existence of dissipative renormalised solutions}\label{sec:existence}
Throughout this section, we assume the hypotheses of Theorem~\ref{thm:ex} and adopt the notations and conventions introduced in Section~\ref{sec:main.results}.
\subsection{Approximating systems and a priori estimates}
We first consider a problem with regularised reaction and interface transmission rates preserving entropic structure and quasi-positivity.
The approximate system with parameter $\ve\in(0, 1]$ formally reads as
\begin{equation}\label{approx_sys}
	\begin{cases}
		\partial_t u_{\eps,i}^{\sm} - \nabla\cdot(A_i \nabla u_{\eps,i}^{\sm}) = f_{\eps,i}(u_{\ve}^{\sm}), &x\in\Omega_{\sm},\\
	-A_i \nabla u_{\eps,i}^{\sm}\cdot \nu^\sigma = r_{\eps,i}^{\sm}(\bfu_{\ve}), &x\in\Gamma,\\
	-A_i \nabla u_{\eps,i}^{\sm}\cdot \nu^\sigma = 0, &x\in \partial\Omega_{\sm}\setminus\Gamma,\\
		u_{\eps,i}^{\sigma}(0,x) = u_{0,\eps,i}^{\sm}(x), &x\in\Osm,
	\end{cases}
\end{equation}
where we define, for $s\in[0,\infty)^n,\bfs\in[0,\infty)^{2n}$,
\begin{equation*}
	f_{\eps,i}(s):= \frac{f_{i}(s)}{1 + \ve |f(s)|}, \quad r_{\eps,i}^{\sm}(\bfs):= \frac{r_{i}^{\sm}(\bfs)}{1 + \ve|r^{\sm}(\bfs)|
		 }, \quad \text{ and } \quad u_{0,\eps,i}^\sm(x) = \min\{u_{0,i}^{\sm}(x),\eps^{-1}\}.
\end{equation*}

Global existence of non-negative weak solutions to this approximate system can be obtained in a standard way.
\begin{proposition}\label{prop:ex.approx}
	Let $\eps\in(0,1]$. 
	There exists a unique non-negative weak solution $\bfu_\ve=(u_{\ve,i}^\sm)$, $u_{\ve,i}^\sm\ge0$,
	to the approximating system \eqref{approx_sys} in the sense that
    \begin{equation*}
        u_{\eps,i}^{\sm}\in L^2(0,T;H^1(\Osm)), \quad\partial_t u_{\eps,i}^{\sm}\in L^2(0,T;(H^1(\Osm))^*)
    \end{equation*}
with $u_{\eps,i}^{\sm}(0)=u_{0,\eps,i}^{\sm}$ in $L^2(\Om_\sm)$,
    and for every test function $\psi \in L^2(0,T;H^1(\Omega_\sm))$ it holds
	\begin{equation}\label{eq:u-eps.weak}
		\int_0^{T}\la \partial_t u_{\eps,i}^{\sm}, \psi\ra_{(H^1(\Omega_\sm))^*,H^1(\Omega_\sm)} + \int_0^T\la A_i\nabla u_{\eps,i}^{\sm}, \nabla \psi\ra_{\sm} = -\int_0^T\!\!\int_{\Gamma}r_{\eps,i}^{\sm}(\bfu_\ve)\psi \,\dd \mathcal{H}^{d-1} \dd t + \int_0^T\!\!\int_{\Omega_\sm}f_{\eps,i}(u_\ve^{\sm})\psi \,\dd x\dd t,
	\end{equation}
	where $\la \cdot, \cdot \ra_{\sm}$ is the inner product in $L^2(\Omega_\sm)$. Moreover, for all $i=1,\dots,n,$ and $\sm\in\{\pm\}$
	\begin{equation*}
		\|u_{\eps,i}^{\sm}\|_{L^\infty(0,T;L^\infty(\Omega_{\sm}))} + \|u_{\eps,i}^{\sm}\|_{L^\infty(0,T;L^\infty(\Gamma))} \le C(T, \eps).
	\end{equation*}
\end{proposition}
We omit the proof of this well-known result.

\begin{lemma}\label{lem4.2}
	Let $\ve\in(0,1]$ and let $(u_{\eps,i}^{\sm})$ be the solution of \eqref{approx_sys}. Then, for all $t>0\ge 0$,
	\begin{align}\label{eq:edi.approx}
	H(u_\ve(t))\nonumber
		&+ 4\sum_{\sm\in\{\pm\}}\sum_{i=1}^n\int_{0}^{t}\intOsm \nabla \sqrt{u_{\eps,i}^{\sm}}\cdot A_i\nabla \sqrt{u_{\eps,i}^{\sm}}\dd x\dd s
        -\sum_{\sm\in\{\pm\}}\int_{0}^{t}\intOsm \sum_{i=1}^nf_{\eps,i}(u_\eps^\sm)\log u_{\eps,i}^\sm \dxds
        \\&+\int_{0}^{t}\int_{\Gamma}\sum_{\sm\in\{\pm\}}\sum_{i=1}^n r_{\eps,i}^\sm(\bfu_\eps)\log{u_{\eps,i}^\sm}\dSdt\le H(u_{\ve,0})\le H(u_{0}),
	\end{align}
	where $u_\ve$ is obtained from $\bfu_\ve=(u_\ve^+,u_\ve^-)$ as in~\eqref{eq:def.ui}.
\end{lemma}
\begin{proof}
	For $\delta>0$, 
    we have $\psi = \log(u_{\eps,i}^\sm + \delta) \in L^2(0,T;H^1(\Osm))$ thanks to $u_{\eps,i}^\sm\ge 0$ and $|\nabla\psi| \le \delta^{-1}|\nabla u_{\eps,i}^\sm|$. Therefore, by taking $\psi$ as a test function in the weak formulation~\eqref{eq:u-eps.weak} of the approximating problem, summing up the resulting identities, and observing that \[\sum_{\sm\in\{\pm\}}	\int_{0}^{t}\la \partial_t u_{\eps,i}^{\sm}, \log(u_{\eps,i}^\sm + \delta)\ra_{(H^1(\Omega_\sm))^*,H^1(\Omega_\sm)}=H(u_\ve(t)+\delta)-H(u_\ve(0)+\delta),\] where $(u_\ve(t)+\delta)_i:=u_{\ve,i}(t)+\delta$,
    we obtain
	\begin{align*}
		H(u_\ve(t)+\delta)
		&+ \sum_{\sm\in \{\pm\}}\sum_{i=1}^n\int_{0}^{t}\int_{\Omega_\sm}\nabla u_{\eps,i}^{\sm}\cdot A_i\frac{\nabla u_{\eps,i}^{\sm}}{u_{\eps,i}^{\sm}{+\delta}}\dxds
		\\&= 
		H(u_\ve(0)+\delta)
		+ \sum_{\sm \in \{\pm\}}\sum_{i=1}^n \int_{0}^{t}\int_{\Omega_{\sm}}f_{\eps,i}(u_\eps^\sm) \log (u_{\eps,i}^\sm + \delta) \dxds \\&\qquad\qquad
		- \sum_{i=1}^n\int_{0}^{t}\int_{\Gamma}(r_{\eps,i}^+(\bfu_\eps) - r_{\eps,i}^{-}(\bfu_\eps))\log\frac{u_{\eps,i}^+ + \delta}{u_{\eps,i}^- + \delta}\dSds.
	\end{align*}
	We now take the limit $\delta\downarrow 0$,  using the uniform bounds of $u_{\eps,i}^\sm$ in $L^\infty((0,T)\times\Osm)$ and in $L^\infty((0,T)\times\Gamma)$ for any $T>0$, 
	and observing that the reaction and interface terms can be handled with Fatou's lemma 
	by virtue of the local Lipschitz continuity of $f_i$ and $r_i$,
 to arrive at 
	\begin{align*}
			H(u_\ve(t))
		+ \sum_{\sm\in \{\pm\}}\sum_{i=1}^n\int_{0}^{t}\int_{\Omega_\sm}\nabla u_{\eps,i}^{\sm}\cdot A_i\frac{\nabla u_{\eps,i}^{\sm}}{u_{\eps,i}^{\sm}}\dd x \dd s
         -\sum_{\sm \in \{\pm\}}\int_{0}^{t}\int_{\Omega_{\sm}}\sum_{i=1}^n f_{\eps,i}(u_\eps^\sm) \log u_{\eps,i}^\sm \dxds&
         \\ + \int_{0}^{t}\int_{\Gamma}\sum_{i=1}^n(r_{\eps,i}^+(\bfu_\eps) - r_{\eps,i}^{-}(\bfu_\eps))\log\frac{u_{\eps,i}^+}{u_{\eps,i}^-}\dSds
	\le 	H(u_\ve(0))&=H(u_{\ve,0}).
    \end{align*}
	The inequality $H(u_{\ve,0})\le H(u_{0})$ for $\ve\le 1$ follows from the fact that $s\mapsto \mathfrak{B}(s)$ is increasing for $s\ge1$.
\end{proof}

\subsection{Renormalised formulation}
To obtain renormalised solutions, we adapt the construction in \cite{Fischer_2015}. Let $\omega: \mathbb R \to [0,1]$ be a $C^{\infty}$ function such that $\omega|_{(-\infty,0)} = 1$ and $\omega|_{(1,\infty)} = 0$. 
For $E\ge2$ and $j\in \{1,\ldots, 2n\}$, we define 
$\xi_j^E: [0,\infty)^{2n}\to [0,\infty)$ via
\begin{equation*}
	\xi_j^E(\bfu) = (u_j - 3E)\omega\bra{\frac 1E \sum_{i=1}^{2n}u_i - 1} + 3E,\qquad\bfu=(u_1,\dots,u_{2n}).
\end{equation*} 
The function $\xi_j^{E}$ is essentially a smooth approximation of the projection $[0,\infty)^{2n} \ni \bfu \mapsto u_j$ onto the $j$-th coordinate. We have the following properties of $\xi_j^E$, where constants hidden in $\lesssim$ are independent of $E$: 
\begin{itemize}
	\item[(V1)] $\xi_j^E\in C^2([0,\infty)^{2n})$.
	\item[(V2)] For every $j=1,\ldots, 2n$, it holds that 
	\begin{equation*}
		\max_{1\le i,k\le 2n}\sqrt{u_i}\sqrt{u_k}|D_iD_k \xi_j^E(\bfu)| \lesssim 1. 
	\end{equation*}
	\item[(V3)] $\text{supp}(D\xi_j^E)$ is bounded for every fixed $E>0$ and every $j=1,\ldots, 2n$.
	\item[(V4)] $\lim_{E\to \infty}D_i\xi_j^E(\bfu) = \delta_{ji}$ for all $i,j = 1,\ldots, 2n$ and all $\bfu\in [0,\infty)^{2n}$.
	\item[(V5)] It holds that 
	\begin{equation*}
		\max_{i,j}\sup_{\bfu\in [0,\infty)^{2n}}|D_i\xi_j^E(\bfu)| \lesssim1.
	\end{equation*}
	\item[(V6)] If $\bfu\in[0,\infty)^{2n}$ and $\sum_{j=1}^{2n}u_j\le E$, then $\xi_j^E(\bfu)= u_j$.
	\item[(V7)] For every $K>0$, and every $i,j,k\in \{1,\ldots, 2n\}$ it holds
	\begin{equation*}
		\lim_{E\to \infty}\sup_{|\bfu|_1 \le K}|D_iD_k\xi_j^E(\bfu)| = 0.
	\end{equation*}
	\item[(V8)] If $\bfu \in [0,\infty)^{2n}$ and $\sum_{j=1}^{2n}\xi_j^{E}(\bfu) \le E$, then $\xi_i^E(\bfu) = u_i$.
	\item[(V9)]	If $\bfu \in [0,\infty)^{2n}$ and $\sum_{j=1}^{2n}u_j \le E$, then $\xi_i^E(\bfu) = u_i$ for all $i=1,\ldots, 2n.$
\end{itemize}

\begin{lemma}[Compactness and limit candidate]\label{l:compactness}
	There exists $\bfu=(u_i^{\sm})$ with
	\begin{align*}
		u_i^{\sm} \in L^{\infty}_{\rm{loc}}([0,\infty);(L\log L)(\Omega_\sm)) \quad \text{ and } \quad \sqrt{u_i^{\sm}} \in L^2_{\rm{loc}}([0,\infty); H^1(\Omega_\sm))
	\end{align*}
	such that, after passage to a subsequence,  
	\begin{align*}
		&u_{\eps,i}^{\sm} \to u_i^{\sm} \quad \mathcal{L}^{1+d}\text{-a.e. in } (0,\infty)\times \Omega_\sm,
		\\&\tr_{\Gamma_\sm} u_{\eps,i}^{\sm} \to \tr_{\Gamma_\sm} u_i^{\sm}\quad \mathcal{L}^1\otimes\mathcal{H}^{d-1}\text{-a.e. in } (0,\infty)\times \Gamma,
		\\&\sqrt{u_{i\ve}^\sm}\rightharpoonup\sqrt{u^\sm_i}\quad\text{in }L^2(0,T;H^1(\Om_\sm)),\quad\text{ for all }T>0.
	\end{align*}
\end{lemma}
\begin{proof}
		The proof  extends~\cite[Lemma~2]{Fischer_2015}. Below, we provide the arguments that need to be adapted due to the interface conditions. The main difference is that the truncation device is only defined 
	 locally in space on $V_\beta\cap\ol\Om_\sigma$. Hence, we will first establish local compactness. 
	The subsequential limits will be shown to be independent of $\beta$, and combined with the uniform bounds from the entropy estimate~\eqref{eq:edi.approx} at the approximate level, we will subsequently infer the asserted convergence results  (Steps 2, 3).
	\smallskip 
	
	\noindent{\em Step~1: Compactness of the family $\{\xi_i^E(\wt\bfu_\eps^\beta)\}_\ve$ for $E<\infty$.}
	
	We first establish compactness locally in $V_\beta\cap\ol\Om_\sigma$ for each  $\beta\in\frakP$ with the finite set of points $\frakP\subset\Gamma$ as specified in Section~\ref{ssec:adjrelen.coer}. To this end, we define $w_{\eps,i}^{E,\beta}:=\xi_i^E(\wt\bfu_\eps^\beta)$ with $\wt\bfu_\eps^\beta  = (u_\ve,\tilde u_\ve^\beta)$, $i=1,\ldots, n$ and let $\hat\vp_\beta\in C^\infty(\ol\Om_\sm)$  with $\supp\hat\vp_\beta\subset V_\beta\cap\ol\Om_\sm$.
	 From the properties of the truncation $\xi_i^E$ and the $\ve$-uniform bounds of $u_\eps$ in $L^{\infty}(0,T;L^1(\Osm))$ and of $\sqrt{u_\eps}$ in $L^2(0,T;H^1(\Osm))$ in Lemma \ref{lem4.2}, we see that $\{\hat\vp_\beta w_{\eps,i}^{E,\beta}\}$ is $\ve$-uniformly bounded in $L^2(0,T;H^1(\Om_\sm))$. Indeed, this follows from
	\begin{align*}
		\nabla (\hat\vp_\beta w_{\eps,i}^{E,\beta}) &= w_{\eps,i}^{E,\beta}\nabla\hat\vp_\beta + \sum_{j=1}^{n}\bra{D_j \xi_i^E(u_\eps,\tilde u_\eps^\beta)\nabla u_{\eps,j} + D_{j+n}\xi_i^E(u_\eps,\tilde u_\eps^\beta)\nabla\tilde{u}_{\eps,j}^{\beta}}\hat\vp_\beta\\
		&= w_{\eps,i}^{E,\beta}\nabla\hat\vp_\beta + 2\sum_{j=1}^{n}\bra{D_j \xi_i^E(u_\eps,\tilde u_\eps^\beta)\sqrt{u_{\eps,j}}\nabla \sqrt{u_{\eps,j}} + D_{j+n}\xi_i^E(u_\eps,\tilde u_\eps^\beta)\sqrt{\tilde{u}_{\eps,j}^{\beta}}\nabla\sqrt{\tilde{u}_{\eps,j}^{\beta}}}\hat\vp_\beta.
	\end{align*}

	To obtain temporal control, we consider  the distributional time derivative of $\hat\vp_\beta w_{\eps,i}^{E,\beta}$, which for $\psi\in C^{\infty}_c([0,\infty)\times \ol\Om_\sm)$, $T\in(0,\infty)$, and $t_1,t_2\in [0,T]$ is given by
	\begin{equation}\label{b6}
		\begin{aligned}
			&\intOsm \hat\vp_\beta w_{\eps,i}^{E,\beta}(t_2)\psi(t_2) \dd x - \intOsm \hat\vp_\beta w_{\eps,i}^{E,\beta}(t_1)\psi(t_1) \dd x - \int_{t_1}^{t_2}\intOsm \hat\vp_\beta w_{\eps,i}^{E,\beta}\partial_t\psi \dxdt\\
			&= \sum_{j=1}^{n}\int_{t_1}^{t_2}\int_{\Osm}\hat\vp_\beta \psi D_j\xi_i^E(u_\eps,\tilde u_\eps^\beta) \partial_t u_{\eps,j}^{\sm} \dxdt + \sum_{j=1}^{n}\int_{t_1}^{t_2}\int_{\Osm}\hat\vp_\beta \psi D_{j+n}\xi_i^E(u_\eps, \tilde u_\eps^\beta) \partial_t \tilde{u}_{\eps,j}^{\beta} \dxdt\\
			&=: (I) + (II).
		\end{aligned}
	\end{equation}
Using the equations~\eqref{eq:u-eps.weak}, the 
properties of $\xi_i^E$, and the approximate entropy estimate~\eqref{eq:edi.approx} it is not difficult to show that, for all $\eps\in(0,1]$, \[|(I) + (II)|\lesssim_{E}\|\psi\|_{L^\infty(0,T;W^{1,\infty}(\Om_\sm))},\]
where here and below the dependence of the estimate on $H(u_0),T$ will not be indicated.
For completeness, we provide the details. We estimate the terms $(I)$ and $(II)$ separately. For $(I)$ we use the equation~\eqref{eq:u-eps.weak} for $u_{\eps,j}^{\sm}$ to rewrite,  abbreviating $\xi_i^E(\cdot):= \xi_i^E(\wt\bfu_\eps^\beta)$,
	\begin{equation}\label{b7}
		\begin{aligned}
			(I) &= - \sum_{j=1}^{n}\int_{t_1}^{t_2}\int_{\Osm}A_j  \nabla u_{\eps,j}^\sm\cdot\Big[\hat \vp_\beta \psi\sum_{q=1}^n (D_{q}D_{j}\xi_i^E(\cdot) \nabla u_{\eps,q}^\sm + D_{q+n}D_j\xi_i^E(\cdot) \nabla \tilde{u}_{\eps,q}^{\beta})\\
			&\hspace{2.5in} + \psi D_{j}\xi_i^E(\cdot) \nabla \hat\vp_\beta + \hat\vp_\beta D_{j}\xi_i^E(\cdot) \nabla \psi \Big]\dxdt\\
			&\quad - \sum_{j=1}^n\int_{t_1}^{t_2}\int_{\Gamma_{\sm}}\hat\vp_\beta \psi D_j  \xi_i^E(\cdot) r_i^{\sm}(\bfu_\eps) \dSdt\\
			&\quad + \sum_{j=1}^{n}\int_{t_1}^{t_2}\int_{\Osm}\hat{\vp}_\beta\psi D_j \xi_i^E(\cdot)f_j(u_\eps^\sm)\dxdt =: (I1) + (I2) + (I3) + (I4) + (I5) + (I6).
		\end{aligned}
	\end{equation}
	Using property (V2) of $\xi_i^E$ and Lemma \ref{lem4.2}, we estimate
	\begin{align*}
		|(I1)| &\lesssim \sum_{j,q=1}^{n} \int_{t_1}^{t_2}\intOsm |A_j|| \hat \vp_\beta| |\psi| \Big|D_qD_j\xi_i^E(\cdot)\sqrt{u_{\eps,j}^\sm}\sqrt{u_{\eps,q}^\sm}\Big|\big|\nabla \sqrt{u_{\eps,j}^\sm}\big|\big|\nabla \sqrt{u_{\eps,q}^\sm}\big| \dxdt\\
		&\lesssim \sum_{j=1,q}^n\|A_j\|_{L^\infty} \|\hat \vp_\beta\|_{L^\infty}\left\|\nabla \sqrt{u_{\eps,j}^\sm}\right\|_{L^2((0,T)\times\Osm)}\left\|\nabla \sqrt{u_{\eps,q}^\sm}\right\|_{L^2((0,T)\times\Osm)}\|\psi\|_{L^\infty((0,T)\times\Osm)}\\
		&\lesssim\|\psi\|_{L^\infty((0,T)\times\Osm)}.
	\end{align*}
	Completely analogous estimates yield the bound
	$|(I2)| \lesssim\|\psi\|_{L^\infty((0,T)\times\Osm)}.$
Furthermore, 
	\begin{align}\label{eq:I1I2}
		\begin{aligned}
		|(I3)| &\lesssim
		  \int_{t_1}^{t_2}\intOsm|\psi D_j\xi_i^E(\cdot)|\sqrt{u_{\eps,j}^\sm}\Big|\nabla \sqrt{u_{\eps,j}^\sm}\Big|\dxdt \lesssim \|\psi\|_{L^\infty((0,T)\times\Osm)},
		  \\	|(I4)| &\lesssim 
		  \int_{t_1}^{t_2}\intOsm \sqrt{u_{\eps,j}^\sm}|D_j\xi_i^E(\cdot)| \Big|\nabla \sqrt{u_{\eps,j}^\sm}\Big|\big|\nabla \psi\big| \dxdt \lesssim\|\nabla \psi\|_{L^\infty((0,T)\times\Osm)},
		  \\	|(I5)| &\lesssim 
		  \sum_{j=1}^{n}\int_{t_1}^{t_2}\int_{\Gamma_\sm}
		  |\psi| |D_j\xi_i^E(\cdot)||r_i^\sm(\bfu_\eps)|\dSdt \lesssim_E \|\psi\|_{L^2((0,T);H^1(\Osm))}, 
		  \\	|(I6)| &\lesssim 
		  \sum_{j=1}^n\int_{t_1}^{t_2}\intOsm |\psi||D_j\xi_i^E(\cdot)||f_j(u_\eps^\sm)|\dxdt
		   \lesssim_E \|\psi\|_{L^1((0,T)\times\Osm)}.		   
		   \end{aligned}
	\end{align}
	For $(II)$, we change variables via $\Phi_{\beta,\pm}$, recall that $|\det D\Phi_{\beta,\pm}|=1$, and use the equation for $(u_{\ve,i}^{-\sm})_i$ to infer
	\begin{align*}
		(II) 
		&= -\sum_{j=1}^{n}\int_{t_1}^{t_2}\int_{\Osmm\cap V_\beta}{A}_j\nabla{u}_{\eps,j}\cdot\nabla\big((\hat\vp_\beta\psi D_{j+n}\xi_i^E(\cdot))\circ\Phi_{\beta,-\sm} 
		\big)\dxdt\\
		&\quad - \sum_{j=1}^{n}\int_{t_1}^{t_2}\int_{\Gamma_{-\sm}\cap V_\beta}(\hat\vp_\beta \psi D_{j+n}\xi_i^E(\cdot)){\circ\Phi_{\beta,-\sm} } r_{j}^{-\sm}(\bfu_\eps)\dSdt\\
		&\quad + \sum_{j=1}^{n}\int_{t_1}^{t_2}\int_{\Osmm\cap V_\beta}\hat\vp_\beta\circ\Phi_{\beta,-\sm} \psi\circ\Phi_{\beta,-\sm} D_{j+n}\xi_i^E(\cdot\circ\Phi_{\beta,-\sm})f_j(u_\eps^{-\sm})\dxdt
		\\&=: (II1) + (II2) + (II3).
	\end{align*}
	The estimates of $(II2)$, $(II3)$ are similar to those of $(I5)$, $(I6)$, while $(II1)$ can be handled as $(I1)$--$(I4)$, giving 
	\begin{equation*}
		|(II)| \lesssim_E \|\psi\|_{L^\infty(0,T;W^{1,\infty}(\Osm))}.
	\end{equation*}
	
	The estimates above show that $\{\partial_t (\hat\vp_\beta w_{\eps,i}^{E,\beta})\}_{\eps}$ is bounded in {$L^1(0,T;(W^{1,\infty}(\Omega))^*)$}.  
    Hence, by the Aubin-Lions lemma, the family $\{\hat\vp_\beta w_{\eps,i}^{E,\beta}\}_{\ve}\subset L^2((0,T)\times\Osm)$ is relatively compact.
    Since the function $\hat\vp_\beta\in C^\infty(\overline{\Om}_\sm)$ with $\supp(\hat\vp_\beta) \subset V_\beta \cap \overline{\Om}_\sm$ was arbitrary, there exists $w_{i}^{E,\beta}\in L^2_\loc([0,T]\times(V_\beta \cap \overline{\Om}_\sm))$ and a subsequence $\ve\downarrow0$ such that the local convergence $w_{\eps,i}^{E,\beta}\to w_{i}^{E,\beta}$ in $L^2_\loc([0,T]\times(V_\beta \cap \overline{\Om}_\sm))$ holds true, and thus, after extracting another subsequence, $w_{\eps,i}^{E,\beta}\to w_{i}^{E,\beta}$ pointwise $\mathcal{L}^{1+d}$-a.e.\ in $[0,T]\times(V_\beta \cap \overline{\Om}_\sm)$.
    Furthermore, thanks to the property (V3) of $\xi_i^E$ and the definition $w_{\eps,i}^{E,\beta} = \xi_i^E(\wt\bfu_\eps^\beta)$, we see that $\{w_{\eps,i}^{E,\beta}\}_{\eps}$ is bounded uniformly in $\eps$ in $L^p((0,T)\times (V_\beta\cap\Osm))$ for any $p>2$. 
    Noting also that, by a diagonal sequence argument, the convergent subsequence can be chosen independent of $\sm,\beta,$ and of  $E,T\in\mathbb{N}$,
    we conclude the following convergence results for all $E,T\in\mathbb{N},\beta\in\frakP,\sm\in\{\pm\}$
    	\begin{subequations}\label{eq:wE.L2.strong}
    	\begin{align}\label{eq:wE.L2.strong.or}
    					&w_{\eps,i}^{E,\beta}\to w_{i}^{E,\beta}\quad\text{in }L^2((0,T)\times (V_\beta\cap\Om)),
    		\\&\tilde w_{\eps,i}^{E,\beta}\to \tilde w_{i}^{E,\beta}\quad\text{in }L^2((0,T)\times (V_\beta\cap\Om)), \nonumber
    	\end{align}
    \end{subequations}    
    where $\tilde w_i^{E,\beta}(t,x):= w_i^{E,\beta}(t,\Phi_{\beta}(x))$. 
Upon passage to another subsequence, we also find that
\begin{align*}
&w_{\eps,i}^{E,\beta} \to w_{i}^{E,\beta} \quad \mathcal{L}^{1+d}\text{-a.e. in } \quad (0,\infty)\times (V_\beta\cap\Omega),
\\&\tilde w_{\eps,i}^{E,\beta} \to \tilde w_{i}^{E,\beta} \quad \mathcal{L}^{1+d}\text{-a.e. in } \quad (0,\infty)\times (V_\beta\cap\Omega).
\end{align*}
From the approximate entropy estimate~\eqref{eq:edi.approx} we further infer, using Fatou's lemma,
\begin{align}\label{eq:w.E-unif.int}
\sup_{E\in \mathbb{N}}\|w_{i}^{E,\beta} \log w_{i}^{E,\beta} \|_{L^\infty(0,T; L^1(V_\beta\cap\Om))}\lesssim1.
\end{align}

	\smallskip 

	\noindent{\em Step~2: Limit $E\to\infty$ and candidate $u$.}
	
	We assert that, for all $\beta\in\frakP$,  the pointwise limit $\lim_{E\to\infty}w_i^{E,\beta}$ exists $\mathcal{L}^{1+d}$-a.e.\ in $(0,\infty)\times (V_\beta\cap\Om)$ and analogously for $\beta={\rm out}$ (adopting once more the notations from Section~\ref{ssec:adjrelen.coer}), and that the limiting value is independent of $\beta$, thus defining a measurable function $u_i$ on $(0,\infty)\times\Om$.
	To this end, consider any $(t,x)\in(0,T)\times (V_\beta\cap\Om)$ for which the pointwise limits $\lim_{\ve\downarrow0}w_{\eps,i}^{E,\beta}(t,x)=w_i^{E,\beta}(t,x)$ and $\lim_{\ve\downarrow0}\tilde w_{\eps,i}^{E,\beta}(t,x)=\tilde w_i^{E,\beta}(t,x)$ exist and additionally 
	\begin{align}\label{eq:sum.less.E}
	\sum_{i=1}^n\Big(w_{i}^{E,\beta}(t,x)+\tilde w_{i}^{E,\beta}(t,x)\Big)<E.
	\end{align}
	Then, for $\ve>0$ small enough, since $\xi_{i+n}^{E}(u_\eps,\tilde u_\eps^\beta)=\xi_{i}^{E}(\tilde u_\eps^\beta,u_\eps)=\xi_{i}^{E}(u_\eps,\tilde u_\eps^\beta)\circ\Phi_\beta$ for $i=1,\dots,n$,
	\[
	\sum_{i=1}^{2n}\xi_{i}^{E}(u_\eps,\tilde u_\eps^\beta)(t,x) =w_{\eps,i}^{E,\beta}(t,x)+\tilde w_{\eps,i}^{E,\beta}(t,x)
	\le E.
	\]
	From (V8) and (V9) we then deduce that 
	\[w_{\eps,i}^{E,\beta}(t,x)=\xi_i^E(u_\ve,\tilde u^\beta_\ve)(t,x)=u_{\ve,i}(t,x) = \xi_i^{E'}(u_\eps, \tilde u_\eps^\beta)(t,x),\qquad i=1,\dots,n,\]
	for all $E' > E$, showing that the limit
 $\lim_{E\to\infty}\xi_i^E(u_\eps, \tilde u_\eps^{\beta})(t,x)$ exists and does not depend on $\beta$.
 
 To cope with points, for which~\eqref{eq:sum.less.E} is not satisfied, we rely on the $E$-uniform integrability of $w_{i}^{E,\beta}$ and $\tilde w_{i}^{E,\beta}$, which follows from~\eqref{eq:w.E-unif.int}, to infer
	\[\mathcal{L}^{1+d}\Big(\big\{(t,x)\in(0,T)\times (V_\beta\cap\Om)\mid\sum_{i=1}^n\big(w_{i}^{E,\beta}(t,x)+\tilde w_{i}^{E,\beta}(t,x)\big)\ge E\big\}\Big)\longrightarrow 0\quad\text{ as }E\to\infty.\]
	Thus, $\lim_{E\to\infty}w_i^{E,\beta} = u_i$ holds true $\mathcal{L}^{1+d}$-a.e.\ in $(0,T)\times(V_\beta \cap \Om)$ for some $\mathcal{L}^{1+d}$-measurable $u_i:(0,\infty)\times\Om\to[0,\infty)$. Furthermore, $\lim_{E\to\infty}w_i^{E,\beta}(t,\Phi_{\beta}(x)) = u_i(t,\Phi_{\beta}(x))$ a.e. in $(0,T)\times(V_\beta \cap \Om)$. From~\eqref{eq:w.E-unif.int}
	 and Fatou's lemma, we further infer that $u_i\log u_i \in L^{\infty}(0,T;L^1(V_\beta\cap\Om))$.
	 
	 We henceforth denote by $u_i^\sm$ the restriction of $u_i$ to $(0,\infty)\times\Om_\sm$, $\sm\in\{\pm\}$.
	
		\smallskip 
	 
	 	\noindent {\em Step~3: Convergence of $u_{\ve,i}^\sm$ to $u_i^\sm$.}
	 
	 Let $\delta>0$ be arbitrary and denote by $Q_{T,\beta}^{\sm}:= (0,T)\times (V_\beta\cap\Osm)$. We now estimate 
	\begin{align}
		\mathcal{L}^{1+d}&\bra{\left\{(t,x)\in Q_{T,\beta}^{\sm}: |u_{\eps,i}(t,x) - u_i(t,x)| > \delta\right\}}\nonumber\\
		&\le \mathcal{L}^{1+d}\bra{\left\{(t,x)\in Q_{T,\beta}^{\sm}: u_{\eps,i}(t,x) \ne \xi_i^E(u_\eps, \tilde u_\eps^\beta)(t,x)\right\}}\label{eq:ptw.conv.ineq}\\
		&\qquad+ \mathcal{L}^{1+d}\bra{\left\{(t,x)\in Q_{T,\beta}^{\sm}: |\xi_i^E(u_\eps, \tilde u_\eps^\beta) - w_i^{E,\beta}(t,x)| > \frac{\delta}{2} \right\}}\nonumber\\
		&\qquad+ \mathcal{L}^{1+d}\bra{\left\{(t,x)\in Q_{T,\beta}^{\sm}: |w_i^{E,\beta}(t,x) - u_i(t,x)| > \frac{\delta}{2} \right\}}.\nonumber
	\end{align}
	Using the $\ve$-uniform limit, which follows from the entropy inequality in Lemma~\ref{lem4.2},
	\begin{align*}
		&\lim_{E\to\infty}\mathcal{L}^{1+d}\bra{\left\{(t,x)\in Q_{T,\beta}^{\sm}: u_{\eps,i}(t,x) \ne \xi_i^E(u_\eps, \tilde u_\eps^\beta)(t,x)\right\}}\\
		&\qquad\le \lim_{E\to\infty}\mathcal{L}^{1+d}\bra{\left\{(t,x)\in Q_{T,\beta}^{\sm}: \sum_{j=1}^{n}(u_{\eps,j}(t,x) + \tilde u_{\eps,j}^{\beta}(t,x)) \ge E\right\}} = 0
	\end{align*}
	and the previous pointwise convergence results, we deduce from~\eqref{eq:ptw.conv.ineq} that $\lim_{\eps\downarrow 0}u_{\eps,i} = u_i$ a.e. in $Q_{T,\beta}^{\sm}$.
This is true for all $\beta \in \frakP$. Furthermore, by directly following~\cite{Fischer_2015},
we can obtain an analogous result for
	$\beta = \text{out}$, so that, along an appropriate subsequence $\ve\downarrow0$,
	\begin{align}\label{eq:ptw.ui.eps.bulk}
	\lim_{\eps\downarrow 0}u_{\eps,i} = u_i \quad \mathcal{L}^{1+d}\text{-a.e.\ in }(0,\infty)\times\Om.
	\end{align}	
Combining the uniform bounds in Lemma~\ref{lem4.2} with~\eqref{eq:ptw.ui.eps.bulk}, we thus infer for all $T>0$, along the same subsequence $\ve\downarrow0$,
\begin{align*}
\lim_{\epsilon\downarrow0}\Big\|\sqrt{u_{\eps,i}} - \sqrt{u_i}\Big\|_{L^2((0,T)\times\Osm)}=0.
\end{align*}
Moreover, from the boundedness of $\{\sqrt{u_{\eps,i}}\}_{\eps}$ in $L^2((0,T);H^1(\Osm))$, we deduce the weak convergence $\sqrt{u_{\eps,i}}\rightharpoonup \sqrt{u_i}$ in $L^2(0,T;H^1(\Osm))$. Therefore, thanks to the weak lower semicontinuity of the $L^2$-norm, 
	\begin{equation*}
		\int_0^T\int_{\Osm}\big|\nabla \sqrt{u_i}\big|^2\dxdt \le \liminf_{\ve\downarrow0}\int_0^T\int_{\Osm}\Big|\nabla\sqrt{u_{\eps,i}}\Big|^2\dxdt \lesssim 1.
	\end{equation*}
	Finally, by the interpolation trace inequality
	\begin{align*}
		\int_0^T\int_{\Gamma_\sm}\Big|\sqrt{u_{\eps,i}} - \sqrt{u_i}\Big|^2\dSdt
		\le \Big\|\sqrt{u_{\eps,i}} - \sqrt{u_i}\Big\|_{L^2((0,T)\times\Osm)}\Big\|\sqrt{u_{\eps,i}} - \sqrt{u_i}\Big\|_{L^2(0,T;H^1(\Osm))},
	\end{align*}
and thus
	\begin{equation*}
		\sqrt{\text{tr}_{\Gamma_\sm}u_{\eps,i}} \to \sqrt{\text{tr}_{\Gamma_\sm} u_i} \qquad\text{ in}\quad L^2((0,T)\times\Gamma)
	\end{equation*}
	for all $T<\infty$.
Upon passage to another subsequence, we conclude
 the pointwise convergence $\mathcal{L}^1\otimes \mathcal{H}^{d-1}$-a.e.\ of $\tr_{\Gamma_\sm} u_{\eps,i}^{\sm}$ to $\tr_{\Gamma_\sm} u_i^{\sm}$ in $(0,\infty)\times\Gamma$.
	 
	This completes the proof of Lemma \ref{l:compactness}.
\end{proof}

\begin{lemma}[Preliminary equation with defect measure]\label{l:equation.defect-meas} Let $u$ be as in Lemma~\ref{l:compactness} and recall the notation $\wt\bfu^\beta:=(u,\tilde u^\beta)$ for $\beta\in\Gamma$. There exist signed Radon measures $\mu_i^{E,\beta,\sm}\in \mathcal{M}_\loc([0,\infty)\times\overline{V_\beta\cap \Osm})$, $i=1,\dots,n$, $\sm\in\{\pm\}$,
	such that, abbreviating  $Q_{T,\infty}^{\sm} := (0,\infty)\times(V_\beta\cap \Osm)$,
	for any function $\psi\in C^\infty_c({[0,\infty)}\times {V_\beta})$,  
	it holds that
	\begin{equation}\label{eq:defect}
	\begin{aligned}
		-\int_{V_\beta\cap\Osm}& \xi_i^E(\wt\bfu^{\beta}_0)\psi(0)\dd x - \int_{Q_{\infty,\beta}^{\sm}} \xi_i^E(\wt\bfu^{\beta})\partial_t \psi \dxdt\\
		&= -\int_{\overline Q_{\infty,\beta}^{\sm}}\psi \dd\mu_i^{E,\beta,\sm}(t,x)   -\sum_{j=1}^{n}\int_{Q_{\infty,\beta}^{\sm}} D_j\xi_i^E(\wt\bfu^\beta)\nabla\psi\cdot A_j \nabla u_j \dxdt \\
		&\quad -\sum_{j=1}^n\int_{0}^{\infty}\!\!\int_{\Gamma_\sm\cap V_\beta} \psi D_j\xi_i^E(\wt\bfu^{\beta})r^{\sm}_j(\bfu)\dSdt  + \sum_{j=1}^{n}\int_{Q_{\infty,\beta}^{\sm}} \psi D_j\xi_i^E(\wt\bfu^{\beta})f_j(u)\dxdt
			\\&\quad  - \sum_{j=1}^n\int_{Q_{\infty,\beta}^{-\sm}}D_{j+n}\xi_i^E(\wt\bfu^\beta)\circ\Phi_{\beta,-\sm}  \nabla(\psi\circ \Phi_{\beta,-\sm})\cdot A_j\nabla u_{j} \dxdt 
		\\&\quad-\sum_{j=1}^n\int_{0}^{\infty}\!\!\int_{\Gamma_{-\sm}\cap V_\beta} (\psi D_{j+n}\xi_i^E(\wt\bfu^{\beta}))\circ\Phi_{\beta,-\sm} r_j^{-\sm}(\bfu)\dSdt\\
		&\quad + \sum_{j=1}^{n} \int_{Q_{\infty,\beta}^{-\sm}}\big(\psi D_{j+n}\xi_i^E(\wt\bfu^{\beta})\big)\circ \Phi_{\beta,-\sm} f_j(u)\dxdt,
	\end{aligned}
\end{equation}
	where for all $T<\infty$
	\begin{align}\label{eq:meas.E.tail}
		\lim_{E\to\infty}|\mu_i^{E,\beta,\sm}|\big([0,T)\times\overline{V_\beta\cap \Osm}\big)=0.
	\end{align}
\end{lemma}
\begin{proof}
 Let $\psi\in C^\infty_c({[0,\infty)}\times \ol\Om_\sm)$ with $\supp\psi\Subset[0,\infty)\times (V_\beta\cap\ol\Om_\sm)$ be given.
	We aim to derive~\eqref{eq:defect} by taking the limit $\ve\downarrow0$ in \eqref{b6}, which is valid for arbitrary $\beta\in\Gamma$. We choose $\hat\vp_\beta\in C^\infty_c(V_\beta\cap\ol\Om_\sm)$ in~\eqref{b6} such that $\hat\vp_\beta=1$ on $\supp\psi(t,\cdot)$ for all $t\in[0,\infty)$, which is possible since, by hypothesis, $\psi(t,\cdot)\equiv0$ for all $t\ge T_\psi$ for some $T_\psi<\infty$. 
	Using the notations in the proof of Lemma~\ref{l:compactness} and observing the convergence  $\xi_i^E(\wt\bfu^{\beta}_\ve)\to \xi_i^E(\wt\bfu^{\beta})$ in $L^1((0,T)\times(V_\beta\cap\Om))$, we infer from~\eqref{b6} (with $t_2>T_\psi$),
	\begin{equation}\label{b8}
		-\int_{Q_{\infty,\beta}^{\sm}}\xi_i^E(\tilde\bfu_0^\beta)
		\psi(0)\,\dd x - \int_{Q_{\infty,\beta}^{\sm}}\xi_i^E(\wt\bfu^{\beta}) 
		\partial_t \psi \dxdt = \lim_{\eps\to0}\Big[(I^\eps) + (II^\eps)\Big],
	\end{equation}
	where $(I^\eps)$, $(II^\eps)$ correspond to the terms on the right-hand side of~\eqref{b6}, where here the superscripts are inserted to indicate the dependence on $\eps>0$. 
	The passage to $\ve\downarrow0$ in $(I5^\eps)$, $(I6^\eps)$ (cf.~\eqref{b7}) is enabled by Lemma \ref{l:compactness} and property (V3) of the truncation $\xi_i^E$,
	\begin{align*}
		\lim_{\ve\downarrow0}((I5^\eps) + (I6^\eps)) = -\sum_{j=1}^n\int_{0}^{\infty}\!\!\int_{\Gamma\cap V_\beta} \psi D_j\xi_i^E(\wt\bfu^\beta)r_i^{\sm}(\bfu)\dSdt + \sum_{j=1}^{n}\int_{Q_{\infty,\beta}^{\sm}} \psi D_j\xi_i^E(\wt\bfu^\beta)f_j(u)\dxdt.
	\end{align*}
	Recalling that $\nabla\hat\vp_\beta\equiv0$ on $\supp\psi(t,\cdot)$ for all $t\ge0$, we find that $(I3^\eps)\equiv0$. For $(I4^\eps)$ we obtain
	\begin{align*}
		\lim_{\ve\downarrow0}(I4^\eps) = -2\lim_{\ve\downarrow0}\int_{Q_{\infty,\beta}^{\sm}}A_j\sqrt{u_{\eps,j}}  D_j\xi_i^{E}(\wt\bfu_\eps^\beta) \nabla \sqrt{u_{\eps,j}}\cdot\nabla \psi\, \dxdt\\
		= -2\int_{Q_{\infty,\beta}^{\sm}}A_j\sqrt{u_j} D_j\xi_i^E(\wt\bfu^\beta)\nabla \sqrt{u_j}\cdot  \nabla \psi\,\dxdt,
	\end{align*}
	where we used the weak convergence of the gradient term, the pointwise convergence and boundedness of of $\sqrt{u_{\eps,j}}D_j\xi_i^E(\wt\bfu_\eps^\beta)$, as well as \cite[Lemma A.2]{FHKM_2022}. Repeating these arguments for $(II^\ve)$ yields
	\begin{align*}
		\lim_{\ve\downarrow0}(II2^\eps) &= -\sum_{j=1}^n\int_{0}^{\infty}\!\!\int_{\Gamma_{-\sm}\cap V_\beta} (\psi D_{j+n}\xi_i^E(\wt\bfu^\beta)){\circ\Phi_{\beta,-\sm} } r_j^{-\sm}(\bfu)\dSdt,
		\\	\lim_{\ve\downarrow0}(II3^\eps) &= \sum_{j=1}^{n}\int_{Q_{\infty,\beta}^{-\sm}}\psi \circ\Phi_{\beta,-\sm} D_{j+n}\xi_i^E(\wt\bfu^\beta\circ\Phi_{\beta,-\sm}) f_j(u)\dxdt,
	\end{align*}
	and, with the notation $\tilde \psi:= \psi\circ \Phi_{\beta,-\sm}$,
	\begin{align*}
		\lim_{\ve\downarrow0}\sum_{j=1}^n\int_{Q_{\infty,\beta}^{-\sm}} D_{j+n}\xi_i^E(\wt\bfu^\beta_\eps\circ\Phi_{\beta,-\sm})A_j  \nabla   u_{\eps,j} \cdot \nabla\tilde\psi\,\dxdt
		= \sum_{j=1}^n\int_{Q_{\infty,\beta}^{-\sm}}D_{j+n}\xi_i^E(\wt\bfu^\beta\circ\Phi_{\beta,-\sm})A_j  \nabla u_{j} \cdot \nabla\tilde{\psi}\, \dxdt.
	\end{align*}
	Inserting these limits into~\eqref{b8} and reordering terms,  we deduce
		\begin{align*}
			-\int_{V_\beta\cap \Osm}& \xi_i^E(\tilde\bfu_0^\beta)\psi(0)\dd x - \int_{Q_{\infty,\beta}^{\sm}}\xi_i^E(\wt\bfu^{\beta}) \partial_t \psi \dxdt\\
			&= -\lim_{\ve\downarrow0}\bigg(\sum_{j,q=1}^{n}\int_{Q_{\infty,\beta}^{\sm}}\psi\Big[ A_j \nabla u_{\eps,j} \cdot\big(D_qD_j\xi_i^E(\wt\bfu^\beta_\eps)\nabla u_{\eps,q} + D_{q+n}D_j\xi_i^E(\wt\bfu^\beta_\eps)\nabla \tilde u_{\eps,q}^{\beta}  \big)\\
			&\hspace{3.75cm} +\tilde{A}_j \nabla \tilde{u}_{\eps,j}^{\beta}\big(D_{q}D_{j+n}\xi_i^E(\wt\bfu^\beta_\eps)\nabla u_{\eps,q} + D_{q+n}D_{j+n}\xi_i^E(\wt\bfu^\beta_\eps)\nabla \tilde{u}_{\eps,q}^{\beta}\big)\Big]\dxdt \bigg)\\
			&\quad -\sum_{j=1}^n\int_{0}^{\infty}\!\!\int_{\Gamma_\sm\cap V_\beta} \psi D_j\xi_i^E(\wt\bfu^\beta)r_j^{\sm}(\bfu)\dSdt + \sum_{j=1}^{n}\int_{Q_{\infty,\beta}^{\sm}} \psi D_j\xi_i^E(\wt\bfu^\beta)f_j(u)\dxdt\\
			&\quad -\int_{Q_{\infty,\beta}^{\sm}}D_j\xi_i^E(\wt\bfu^\beta)A_j  \nabla {u_j }\cdot \nabla \psi \dxdt - \sum_{j=1}^n\int_{Q_{\infty,\beta}^{-\sm}} D_{j+n}\xi_i^E(\wt\bfu^\beta\circ \Phi_{\beta,-\sm})A_j  \nabla   u_{j} \cdot\nabla\tilde{\psi}\dxdt \\&\quad-\sum_{j=1}^n\int_{0}^{\infty}\!\!\int_{\Gamma_{-\sm}\cap V_\beta}(\psi D_{j+n}\xi_i^E(\wt\bfu^\beta)){\circ\Phi_{\beta,-\sm} } r_j^{-\sm}(\bfu)\dSdt
			\\&\quad  + \sum_{j=1}^{n}\int_{Q_{\infty,\beta}^{-\sm}} \tilde\psi D_{j+n}\xi_i^E(\wt\bfu^\beta\circ \Phi_{\beta,-\sm}) f_j(u)\dxdt,
		\end{align*}
	where we used a change of variables and abbreviate $\tilde A_j :=
	\big(D\Phi_\beta \,A_j(D\Phi_{\beta})^T\big)\circ \Phi_{\beta}^{-1}$
to arrive at the common domain of integration $Q_{\infty,\beta}^{\sm}$ in the second and third line.

	We then define the signed Radon measures
	\begin{align*}
		\mu_{\eps,i}^{E,\beta,\sm}:= &\sum_{j,q=1}^{n}  \Big[ A_j \nabla u_{\eps,j} \cdot\big(D_qD_j\xi_i^E(\wt\bfu^\beta_\eps)\nabla u_{\eps,q} + D_{q+n}D_j\xi_i^E(\wt\bfu^\beta_\eps)\nabla \tilde u_{\eps,q}^{\beta}  \big)\\
		&+\tilde{A}_j \nabla \tilde{u}_{\eps,j}^{\beta}\big(D_{q}D_{j+n}\xi_i^E(\wt\bfu^\beta_\eps)\nabla u_{\eps,q} + D_{q+n}D_{j+n}\xi_i^E(\wt\bfu^\beta_\eps)\nabla \tilde{u}_{\eps,q}^{\beta}\big)\Big]\dxdt.
	\end{align*}
	From the estimates of $(I1)$, $(I2)$, $(II1)$ (cf.~\eqref{eq:I1I2}), we see that for all $T'<\infty$
	\begin{equation*}
		|\mu_{\eps,i}^{E,\beta,\sm}|([0,T']{\times}\ol{V_\beta\cap\Om_{\sm}}) \le C_{T'}<\infty.
	\end{equation*}
 Thus, up to a subsequence, we deduce the local convergence $\mu_{\eps,i}^{E,\beta,\sm} \weakstar \mu_i^{E,\beta,\sm}$ in
 $\mathcal{M}_\loc([0,\infty){\times}\ol{V_\beta\cap\Om_{\sm}})$
	 for a signed Randon measure $\mu_i^{E,\beta,\sm}\in \mathcal{M}_\loc([0,\infty){\times}\ol{V_\beta\cap\Om_{\sm}})$. This implies~\eqref{eq:defect}.
	 
	  It remains to show~\eqref{eq:meas.E.tail} for every $T<\infty$. Defining the measures
	\begin{align*}
		&\nu_{j,K,\eps}:= \chi_{\{K - 1 \le |\tilde\bfu^{\beta}|_1 < K\}}\Big|\nabla\sqrt{u_{\eps,j}} \Big|^2\dxdt,
		\\[2mm]&	\tilde\nu_{j,K,\eps}^{\beta}:= \chi_{\{K - 1 \le |\tilde\bfu^{\beta}|_1 < K\}}\Big|\nabla\sqrt{\tilde u_{\eps,j}^{\beta}} \Big|^2\dxdt,
	\end{align*}
	we can estimate 
	\begin{align*}
		|\mu_{\eps,i}^{E,\beta}|([0,T){\times}\overline{V_\beta\cap \Osm}) &\le C\sum_{j,q=1}^{n}\sum_{K=1}^{\infty}\int_0^T\!\int_{\overline{V_\beta\cap \Osm}}\Bigg(
		\big|\sqrt{u_{\eps,j}}\sqrt{u_{\eps,q}}
		D_qD_j\xi_i^E(\wt\bfu^\beta_\eps)
		\big|\chi_{\{K - 1 \le |\tilde\bfu^{\beta}|_1 < K\}}\big|\nabla \sqrt{u_{\eps,j}}\big|^2\\ 
		&\qquad+ 
		\big|\sqrt{u_{\eps,j}}\sqrt{\tilde u_{\eps,q}^{\beta}}
		D_{q+n}D_j\xi_i^E(\wt\bfu^\beta_\eps)
		\big|\chi_{\{K - 1 \le |\tilde\bfu^{\beta}|_1 < K\}}\big(\big|\nabla \sqrt{u_{\eps,j}}\big|^2 + \nabla \sqrt{\tilde u_{\eps,q}^{\beta}}\big|^2\big)\\
		\\&\qquad+\big|\sqrt{\tilde u_{\eps,j}^{\beta}}\sqrt{u_{\eps,q}}
		D_qD_{j+n}\xi_i^E(\wt\bfu^\beta_\eps)
		\big|\chi_{\{K - 1 \le |\tilde\bfu^{\beta}|_1 < K\}}\big(\big|\nabla \sqrt{\tilde u_{\eps,j}^{\beta}}\big|^2 + |\nabla\sqrt{u_{\eps,q}}|^2\big)\\
		\\&\qquad+\big|\sqrt{\tilde u_{\eps,j}^{\beta}}\sqrt{\tilde u_{\eps,q}^{\beta}}D_{q+n}D_{j+n}\xi_i^E(\wt\bfu^\beta_\eps)
		\big|\chi_{\{K - 1 \le |\tilde\bfu^{\beta}|_1 < K\}}\big|\nabla \sqrt{\tilde u_{\eps,j}^{\beta}}\big|^2
		\Bigg)\dxdt\\
		\le C\sum_{j=1}^{n}\sum_{K=1}^{\infty}\Big(&\nu_{j,K,\eps}([0,T){\times}\overline{V_\beta\cap \Osm}) + \tilde \nu_{j,K,\eps}^{\beta}([0,T){\times}\overline{V_\beta\cap \Osm}) \Big)\underset{\underset{K-1\le |\bfu|_1 \le K}{1\le j,q\le 2n}}{\sup}|\sqrt{u_j}\sqrt{u_q}D_{j}D_{q}\xi_i^E(\bfu)|.
	\end{align*}
	From the boundedness, uniformly in $j, K$ and $\eps$, of $\nu_{j,K,\eps}$ and $\tilde\nu_{j,K,\eps}^{\beta}$ in $\mathcal{M}([0,T']{\times}\ol{V_\beta\cap\Om_{\sm}})$ for every $T'<\infty$,
	 we deduce, along a subsequence $\ve\downarrow0$, $\nu_{j,K,\eps} \weakstar \nu_{j,K}$ and $\tilde\nu_{j,K,\eps}^{\beta} \weakstar \tilde\nu_{j,K}^{\beta}$ in
	 $\mathcal{M}_\loc([0,\infty){\times}\ol{V_\beta\cap\Om_{\sm}})$.
	Therefore, using also the lower semicontinuity of the total variation measure on relatively open  subsets with respect to weak-$*$ convergence of signed Radon measures, we deduce
	\begin{align*}
	&	|\mu_{i}^{E,\beta,\sm}|([0,T){\times}\overline{V_\beta\cap \Osm})
		\le \liminf_{\ve\downarrow0}|\mu_{\eps,i}^{E,\beta,\sm}|([0,T){\times}\overline{V_\beta\cap \Osm})\\
		&\quad\le C\sum_{j=1}^{n}\sum_{K=1}^{\infty}\big(\nu_{j,K}([0,T){\times}\overline{V_\beta\cap \Osm})
		 + \tilde \nu_{j,K}^{\beta}([0,T){\times}\overline{V_\beta\cap \Osm})\big)
		 \underset{\underset{K-1\le |\bfu|_1 \le K}{1\le j,q\le 2n}}{\sup}|\sqrt{u_j}\sqrt{u_q}D_{j}D_{q}\xi_i^E(\bfu)|,     
	\end{align*}
where
\begin{align*}
	\sum_{K=1}^{\infty}\Big(\nu_{j,K}([0,T){\times}\overline{V_\beta\cap \Osm}) &+ \tilde \nu_{j,K}^{\beta}([0,T){\times}\overline{V_\beta\cap \Osm}) \Big)
	\\&\le \liminf_{\ve\downarrow0}\sum_{K=1}^{\infty}\Big(\nu_{j,K,\eps}([0,T){\times}\overline{V_\beta\cap \Osm}) + \tilde \nu_{j,K,\eps}^{\beta}([0,T){\times}\overline{V_\beta\cap \Osm}) \Big)\\
	&= \liminf_{\ve\downarrow0}\int_0^T\!\!\int_{\overline{V_\beta\cap \Osm}}\bra{\Big|\nabla\sqrt{u_{\eps,j}}\Big|^2 + \Big|\nabla\sqrt{\tilde u_{\eps,j}^{\beta}}\Big|^2}\dxdt < +\infty.
\end{align*}
	Therefore, by (V2), (V7), and the dominated convergence theorem, we obtain the asserted limit
	\begin{align*}
		&\lim_{E\to\infty}|\mu_{i}^{E,\beta,\sm}|([0,T){\times}\overline{V_\beta\cap \Osm})
		\\&\;\; \le C\sum_{j=1}^{n}\sum_{K=1}^{\infty}\Big(\nu_{j,K}([0,T){\times}\overline{V_\beta\cap \Osm}) + \tilde \nu_{j,K}^{\beta}([0,T){\times}\overline{V_\beta\cap \Osm}) \Big)\lim_{E\to\infty}\underset{\underset{K-1\le |\bfu|_1 \le K}{1\le j,q\le 2n}}{\sup}|\sqrt{u_ju_q}D_{j}D_{q}\xi_i^E(\bfu)| = 0.
	\end{align*}
\end{proof}

We need the following weak chain rule, which can be deduced from~\cite[Lemma 4]{Fischer_2015}.
\begin{lemma}\label{lem:chain_rule} Let $W:=V\cap\ol{\Om}_0$, where $\Om_0$ is a bounded Lipschitz domain and $V$ is open, and let $\Gamma\subset V\cap\pa\Om_0$  be $\mathcal{H}^{d-1}$-measurable.
	Let ${\textbf{v}}=(v_i)_{i=1}^{m} \in L^1_\loc([0,\infty);L^1(W)^{m})\cap L^2_\loc([0,\infty);H^1(W)^{m})$, $\textbf{v}_0=(v_{0,i})_{i=1}^m\in L^1(W)^{m}$.
	 Let $\nu_i\in \mathcal{M}_\loc([0,\infty)\times\ol W)$, $p_i \in L^1_\loc([0,\infty);L^1(W))$, $\mathsf{z}_i\in L^2_\loc([0,\infty);L^2(W)^d)$, $a_i\in L^1_\loc([0,\infty);L^1(\Gamma))$ for $i=1,\dots,m$. 
		Assume that,  for $i=1,\dots,m$, it holds for all $\psi \in C^\infty_c([0,\infty)\times V)$
	\begin{align}\label{eq:hp.weak.chainrule}
		\begin{aligned}
		 - \int_{W}v_{0,i} \psi(0) \dd x -\int_{0}^\infty\!\!\int_{W}&v_i \partial_t\psi \dxdt\\
		&= \int_{[0,\infty){\times}\ol W}\psi \dd \nu_i + \int_0^\infty\!\!\int_{W}(p_i \psi + \mathsf{z}_i \cdot \nabla \psi)\dxdt + \int_{0}^\infty\!\!\int_{\Gamma}a_i \psi \dSdt.
\end{aligned}
\end{align}
	Let $\xi\in C^\infty(\mathbb{R}^m)$ with $D\xi\in C_c^\infty(\mathbb{R}^m)^m$.
	Then, for all $\psi \in C^{\infty}_c([0,\infty)\times V)$ and $T_\psi<\infty$ such that $\supp\psi\Subset[0,T_\psi)\times V$ it holds that
	\begin{align*}
		\Bigg| &
		- \int_{W} \xi(\bff{v}(0))\psi(0)\dd x-\int_{0}^\infty\!\!\int_{W} \xi(\bff{v}) \partial_t \psi \dxdt\\
		& - \sum_{i=1}^{m}\int_{0}^\infty\!\!\int_{W} p_i \psi D_i \xi(\bff v)\dxdt 
		 - \sum_{i=1}^{m}\int_{0}^\infty\!\!\int_{W}\mathsf{z}_i\cdot \nabla(\psi D_i\xi(\bff v))\dxdt 
		- \sum_{i=1}^{m}\int_{0}^\infty\!\!\int_{\Gamma}a_i \psi D_i\xi(\bff v)\dSdt
		\Bigg|\\
	&\lesssim \|\psi\|_{L^\infty}\|D\xi\|_{L^\infty(\mathbb{R}^m)} \sum_{i=1}^{m}\|\nu_i\|_{\mathcal{M}([0,T_\psi)\times\ol W)}.
	\end{align*}
\end{lemma}

We are now ready to prove the existence of global-in-time dissipative renormalised solutions.
\begin{proof}[Proof of Theorem \ref{thm:ex}]
	Let $\xi\in C^\infty(\mathbb{R}^{2n})$ with $D\xi\in C_c^\infty(\mathbb{R}^{2n})^n$ be an admissible truncation function (cf.\ Definition~\ref{def:renormalised.gen}).
	To derive the equation in $(0,T)\times(V_\beta\cap\ol\Om_\sm)$ for given $\beta\in\Gamma$, we wish to apply the weak chain rule in Lemma \ref{lem:chain_rule} with 
	$V=V_\beta, \Om_0=\Om_\sm$, and 
	$m=2n$ to the map
	\[
	\boldsymbol{v}:=\boldsymbol{\xi}^E(\tilde \bfu^\beta)= \big(\xi_i^E(u,\tilde u^\beta)
	\big)_{i=1,\ldots, 2n}.
	\]
	To this end, we must first determine the equations~\eqref{eq:hp.weak.chainrule}.

For $v_i=\xi_i^E(u,\tilde u^\beta),$ $i=1,\dots,n$, by Lemma~\ref{l:equation.defect-meas} with a change of variables in the fourth and sixth line in~\eqref{eq:defect}, we conclude that equation~\eqref{eq:hp.weak.chainrule} holds with the choice
	\begin{align*}
		&\nu_i^{\sm}:= -\mu_i^{E,\beta,\sm}\\
		&p_i^\sm:= \sum_{j=1}^{n}\big(D_j\xi_i^E(\wt\bfu^\beta)f_j(u)+D_{j+n}\xi_i^E(\wt\bfu^\beta)f_j(\tilde u^\beta)\big) 
			\\
		&\mathsf{z}_i^\sm:= -\sum_{j=1}^n(D_j\xi_i^E(\wt \bfu^\beta)A_j\nabla u_j+D_{j+n}\xi_i^E(\wt \bfu^\beta)\wt A_j\nabla\tilde u_j^\beta)\\
	&a_i^\sm:= -\sum_{j=1}^n\big(D_j\xi_i^E(\wt \bfu^\beta)r_j^\sm(\bfu) +D_{j+n}\xi_i^E(\wt\bfu^\beta)\circ\Phi_{\beta}\,r_j^{-\sm}(\bfu)\big),
	\end{align*}
	with
	\[\wt A_j := \big(D\Phi_\beta \,A_j(D\Phi_{\beta})^T\big)\circ \Phi_{\beta}^{-1}\]
	and where we recall that that $r^\pm_j(\bfu)$ is to be understood in the sense $r^\pm_j(\tr_\Gamma u^+,\tr_\Gamma u^-)$.
	
To determine the equation for $v_{i+n}=\xi_{i+n}^E(u,\tilde u^\beta),$  $i=1,\dots,n$, we observe that by construction and due to $\Phi_\beta^{-1}=\Phi_\beta$ 
\begin{align}\label{eq:xi.i+n.transform}
\xi_{i+n}^E(\wt\bfu^{\beta})=\xi_{i+n}^E(u,\tilde u^\beta)=\xi_{i}^E(\tilde u^\beta,u)=\xi_{i}^E(u,\tilde u^\beta)\circ\Phi_\beta^{-1}\qquad\text{ in }(0,\infty)\times (V_\beta\cap\Om).
\end{align}
Thus, it suffices to transform the equation for $\xi_{i}^E(u,\tilde u^\beta)$ for space variable $y\in V_\beta\cap\Om_{-\sm}$ in Lemma~\ref{l:equation.defect-meas} via $\Phi_\beta$ to $x\in V_\beta\cap\Om_{\sm}$. Choosing in~\eqref{eq:defect} the test function $\psi(t,\Phi_\beta(y))$ instead of $\psi(t,y)$, which is admissible by a density argument, and using~\eqref{eq:xi.i+n.transform}, we thereby deduce the identity
	\begin{equation*}
	\begin{aligned}
		-&\int_{V_\beta\cap\Osm} \xi_{i+n}^E(\wt\bfu^{\beta}_0)\psi(0)\dd x - \int_{Q_{\infty,\beta}^{\sm}} \xi_{i+n}^E(\wt\bfu^{\beta})\partial_t \psi \dxdt\\
		&= -\int_{\overline Q_{\infty,\beta}^{\sm}}\psi \dd((\Phi_{\beta})_{\#}\mu_i^{E,\beta,-\sm})(t,x)  
		 -\sum_{j=1}^{n}\int_{Q_{\infty,\beta}^{\sm}}   D_j\xi^E_{i}(\wt\bfu^{\beta})\circ \Phi^{-1}_\beta\nabla \psi 
	\cdot\wt A_j\nabla \tilde u_j^\beta \dxdt  \\&\qquad+ \sum_{j=1}^{n}\int_{Q_{\infty,\beta}^{\sm}} \psi 
		 \big(D_j\xi^E_{i}(\wt\bfu^{\beta}) f_j(u)\big)\circ \Phi^{-1}_\beta\dxdt\\
		&\qquad -\sum_{j=1}^n\int_{0}^{\infty}\int_{\Gamma_\sm\cap V_\beta} \psi \big(D_j\xi^E_{i}(\wt\bfu^{\beta})\circ \Phi^{-1}_\beta\big) r^{-\sm}_j(\bfu)\dSdt 
			\\&\qquad
			- \sum_{j=1}^n\int_{Q_{\infty,\beta}^{\sm}}D_{j+n}\xi_i^E(\wt\bfu^\beta)\circ\Phi_{\beta} \nabla\psi\cdot A_j\nabla u_{j} \dxdt 
			\\&\qquad-\sum_{j=1}^n\int_{0}^{\infty}\int_{\Gamma_{\sm}\cap V_\beta} \psi D_{j+n}\xi_i^E(\wt\bfu^{\beta}\circ\Phi_{\beta}) r_j^{\sm}(\bfu)\dSdt\\
			&\qquad + \sum_{j=1}^{n} \int_{Q_{\infty,\beta}^{\sm}}\psi D_{j+n}\xi_i^E(\wt\bfu^{\beta}\circ \Phi_{\beta}) f_j(u)\dxdt,
	\end{aligned}
\end{equation*}
Hence, Lemma~\ref{lem:chain_rule} becomes applicable if we further choose (recalling $\Phi_\beta^{-1}=\Phi_\beta$)
	\begin{align*}
	&\nu_{i+n}^{\sm}:=-(\Phi_{\beta})_{\#}\mu_i^{E,\beta,-\sm}\\  
	&p_{i+n}^\sm:= \sum_{j=1}^{n}   \big(D_j\xi^E_{i}(\wt\bfu^{\beta}\circ \Phi_\beta) f_j(\tilde u^\beta)
	+D_{j+n}\xi_i^E(\wt\bfu^{\beta}\circ \Phi_{\beta}) f_j(u)\big)\\
	& \mathsf{z}_{i+n}^\sm:= -\sum_{j=1}^n\big(D_j\xi_i^E(\tilde \bfu^\beta\circ\Phi_\beta)\, \wt A_j\nabla \tilde u_j^\beta+D_{j+n}\xi_i^E(\wt\bfu^\beta\circ\Phi_\beta) A_j\nabla u_j\big)\\
&	a_{i+n}^\sm:= -\sum_{j=1}^n\big(D_j\xi_i^E(\wt\bfu^\beta\circ\Phi_\beta) \,r_j^{-\sm}(\bfu)+D_{j+n}\xi_i^E(\wt\bfu^{\beta}\circ\Phi_{\beta}) r_j^{\sm}(\bfu)\big)
\end{align*}
for $i=1,\dots,n$. Now, invoking Lemma~\ref{lem:chain_rule} and observing that 
\[\|(\Phi_{\beta})_{\#}\mu_i^{E,\beta,-\sm}\|_{\mathcal{M}(\ol Q_{T_\psi,\beta}^\sm)}=\|\mu_i^{E,\beta,-\sm}\|_{\mathcal{M}(\ol Q_{T_\psi,\beta}^{-\sm})},\]
we  deduce
	 for $\psi \in C^{\infty}_c([0,\infty)\times V_\beta)$
	\begin{align}\label{eq:E.renorm}
		&\Bigg|
		- \int_{V_\beta\cap\Osm} \xi(\boldsymbol{\xi}^E(\tilde \bfu^\beta_0)\psi(0)\,\dd x-\int_{Q_{\infty,\beta}^\sm} \xi(\boldsymbol{\xi}^E(\tilde \bfu^\beta) \partial_t \psi \dxdt \\
	&\quad + \sum_{i=1}^{n}\int_{Q_{\infty,\beta}^\sm}\sum_{j=1}^n(D_j\xi_i^E(\wt \bfu^\beta)A_j\nabla u_j+D_{j+n}\xi_i^E(\wt \bfu^\beta)\wt A_j\nabla\tilde u_j^\beta)
	\cdot \nabla \big(\psi D_i\xi(\boldsymbol{\xi}^E(\tilde \bfu^\beta))\big)\dxdt  \nonumber\\
		&\quad + \sum_{i=1}^{n}\int_{Q_{\infty,\beta}^\sm}\sum_{j=1}^n
		\big(D_j\xi_i^E(\tilde \bfu^\beta\circ\Phi_\beta)\, \wt A_j\nabla \tilde u_j^\beta+D_{j+n}\xi_i^E(\wt\bfu^\beta\circ\Phi_\beta) A_j\nabla u_j\big)
	\cdot \nabla \big(\psi D_{i+n}\xi(\boldsymbol{\xi}^E(\tilde \bfu^\beta))\big)\dxdt  \nonumber\\
		&\quad - \sum_{i=1}^{n}\int_{Q_{\infty,\beta}^{\sm}} \sum_{j=1}^{n}\big(D_j\xi_i^E(\wt\bfu^\beta)f_j(u)+D_{j+n}\xi_i^E(\wt\bfu^\beta)f_j(\tilde u^\beta)\big)  \psi D_i \xi(\boldsymbol{\xi}^E(\tilde \bfu^\beta))\dxdt \nonumber\\
		&\quad - \sum_{i=1}^{n}\int_{Q_{\infty,\beta}^{\sm}} \sum_{j=1}^{n}
		 \big(D_j\xi^E_{i}(\wt\bfu^{\beta}\circ \Phi_\beta) f_j(\tilde u^\beta)
		+D_{j+n}\xi_i^E(\wt\bfu^{\beta}\circ \Phi_{\beta}) f_j(u)\big)
		  \psi D_{i+n} \xi(\boldsymbol{\xi}^E(\tilde \bfu^\beta))\dxdt \nonumber\\
		&\quad + \sum_{i=1}^{n}\int_{0}^\infty\int_{V_\beta\cap \Gamma}
		\sum_{j=1}^n\big(D_j\xi_i^E(\wt \bfu^\beta)r_j^\sm(\bfu) +D_{j+n}\xi_i^E(\wt\bfu^\beta)\circ\Phi_{\beta}\,r_j^{-\sm}(\bfu)\big)
		\psi D_i\xi(\boldsymbol{\xi}^E(\tilde \bfu^\beta))\dSdt  \nonumber\\
			&\quad + \sum_{i=1}^{n}\int_{0}^\infty\int_{V_\beta\cap \Gamma}
			\sum_{j=1}^n
			\big(D_j\xi_i^E(\wt\bfu^\beta\circ\Phi_\beta) \,r_j^{-\sm}(\bfu)+D_{j+n}\xi_i^E(\wt\bfu^{\beta}\circ\Phi_{\beta}) r_j^{\sm}(\bfu)\big)
			\psi D_{i+n}\xi(\boldsymbol{\xi}^E(\tilde \bfu^\beta))\dSdt  \nonumber
		\Bigg| \nonumber\\
		&\le C\|\psi\|_{L^\infty}\|D\xi\|_{L^\infty(\mathbb{R}^{2n})}\sum_{i=1}^{n}\big(\|\mu_i^{E,\beta,+}\|_{\mathcal{M}(\ol Q_{T_\psi,\beta}^+)}+\|\mu_i^{E,\beta,-}\|_{\mathcal{M}(\ol Q_{T_\psi,\beta}^-)}\big) \nonumber.
	\end{align}
	We consider the limit $E\to\infty$ in \eqref{eq:E.renorm}. 
	The convergence of the first line is straightforward thanks to the boundedness of $\xi $ and $\xi_i^E(\wt\bfu^\beta) \to (\wt\bfu^\beta)_i$ a.e.\ as $E\to\infty$. 
	To proceed, we note that for $i=1,\dots,n$
	 \[D_j\xi_i^E(\wt \bfu^\beta)\to\delta_{ij}\qquad\text{ and }\qquad D_{j+n}\xi_i^E(\wt \bfu^\beta)\to0\qquad\text{ as }E\to\infty\]
	 in the a.e.\ sense w.r.t.\ $\mathcal{L}^{1+d}$ and $\mathcal{L}\otimes\mathcal{H}^{d-1}$. Furthermore, since $\supp D\xi\subset[0,\infty)^{2n}$ is bounded, there exists a threshold $\ul E<\infty$ such that $D\xi(\wt\bfu)=0$ for all $\wt\bfu\in [0,\infty)^{2n}$ with $|\wt\bfu|_1\ge\ul E$. Moreover, the definition of $\xi_i^E$ implies that, whenever $|\wt\bfu|_1\ge E'$ and $E\ge E'$,  it holds that
	  $|\boldsymbol{\xi}^{E}(\wt\bfu)|_1\ge E'$. Consequently,  if $E\ge\ul E$, then
	 \begin{align*}
	 D\xi(\boldsymbol{\xi}^E(\wt\bfu))=0\quad\text{for all }\wt\bfu\in[0,\infty)^{2n}\text{ with }|\wt\bfu|_1\ge \ul E.
	 \end{align*}
	  
Using these properties, the passage to the limit $E\to\infty$ in the remaining terms is fairly straightforward.
The second line in~\eqref{eq:E.renorm} can be seen to converge to 
	\[\sum_{i=1}^n\int_{Q_{\infty,\beta}^\sm}A_i\nabla u_i	\cdot \nabla \big(\psi D_i\xi(\wt \bfu^\beta)\big)\dxdt.\]
	Similarly, we obtain that the third line converges to 
	\[\sum_{i=1}^{n}\int_{Q_{\infty,\beta}^\sm}
	 \wt A_i\nabla \tilde u_i^\beta
	\cdot \nabla \big(\psi D_{i+n}\xi(\wt \bfu^\beta)\big)\dxdt=\sum_{i=1}^{n}\int_{Q_{\infty,\beta}^{-\sm}}
	 A_i\nabla u_i\cdot \nabla \big(\psi D_{i+n}\xi(\wt \bfu^\beta)\circ\Phi_\beta\big)\dxdt.\]
	The fourth line converges to 
	\[ - \sum_{i=1}^{n}\int_{Q_{\infty,\beta}^{\sm}} f_i(u)\psi D_i \xi(\wt\bfu^\beta))\dxdt,\]
	while the fifth line converges to
	\[- \sum_{i=1}^{n}\int_{Q_{\infty,\beta}^{\sm}} f_i(u)\circ \Phi^{-1}_\beta	\psi D_{i+n} \xi(\wt\bfu^\beta)\dxdt
	=- \sum_{i=1}^{n}\int_{Q_{\infty,\beta}^{-\sm}} f_i(u)\,(\psi D_{i+n} \xi(\wt\bfu^\beta))\circ\Phi_\beta\dxdt.
	\]
The sixth line converges to
\[\sum_{i=1}^{n}\int_{0}^\infty\!\!\int_{V_\beta\cap \Gamma}
\sum_{j=1}^n\big(D_j\xi_i^E(\wt \bfu^\beta)r_j^\sm(\bfu) +D_{j+n}\xi_i^E(\wt\bfu^\beta)\circ\Phi_{\beta}\,r_j^{-\sm}(u^+,u^-)\big)
\psi D_i\xi(\wt \bfu^\beta))\dSdt,\]
while the seventh line converges to
 \[ \sum_{i=1}^{n}\int_{0}^\infty\!\!\int_{V_\beta\cap \Gamma_\sm}
r_i^{-\sm}(\bfu)\psi D_{i+n}\xi(\wt\bfu^\beta))\dSdt 
=\sum_{i=1}^{n}\int_{0}^\infty\!\!\int_{V_\beta\cap \Gamma_{-\sm}}
r_i^{-\sm}(\bfu)(\psi D_{i+n}\xi(\wt\bfu^\beta))\circ\Phi_\beta\dSdt.\]	
	Combining these limits with the fact that the right-hand side of inequality~\eqref{eq:E.renorm} vanishes as $E\to\infty$ due to~\eqref{eq:meas.E.tail}, we deduce a version of the renormalised formulation~\eqref{eq:def.renorm.c.gen} in the interfacial region (with time integral $\int_0^T$ replaced by $\int_0^\infty$ and without the term involving $\psi(T)$). The exact form~\eqref{eq:def.renorm.c.gen} can be deduced from this version by a standard approximation argument.
	 Moreover, the renormalised formulation away from the interface~\eqref{eq:def.renorm.out.gen} can be deduced as in the bulk case, see~\cite{Fischer_2015}.
	Thus, the candidate $\bfu=(u^+,u^-)$, $u^\pm=(u^\pm_i)_i$ obtained in Lemma~\ref{l:compactness} is a renormalised solutions in the sense of Definition~\ref{def:renormalised.gen}. 
	\smallskip
	
	It remains to show the entropy dissipation inequality~\eqref{eq:edi}. In the approximate entropy inequality~\eqref{eq:edi.approx} in Lemma \ref{lem4.2}, we use the weak convergence $\sqrt{u_{\eps,i}^\sm} \rightharpoonup \sqrt{u_{i}^\sm}$, the pointwise convergence $u_{\eps,i}^\sm \to u_{i}^\sm$ both in $\mathcal{L}^{1+d}$-a.e.\ in $(0,T)\times\Osm$ and $\mathcal{L}\otimes \mathcal{H}^{d-1}$-a.e. in $(0,T)\times\Gamma_\sm$, the non-negativity
	\begin{equation*}
		-\sum_{i=1}^n f_{\eps,i}(u_\eps) \log u_{\eps,i} \ge 0\quad \text{ and } \quad \sum_{\sm\in\{\pm\}}\sum_{i=1}^nr_{\eps,i}^\sm(\bfu_\eps)\log u_{\eps,i}^\sm \ge 0,
	\end{equation*}
	and Fatou's lemma to pass to the limit as $\eps \downarrow 0$ to obtain, for a.e.\ $t>0$,
	\begin{align*}
		H(u(t)) + 4\sum_{\sm\in\{\pm\}}\sum_{i=1}^n\int_{0}^{t}\intOsm \nabla \sqrt{u_{i}^{\sm}}\cdot A_i^\sm\nabla \sqrt{u_{i}^{\sm}}\dd x\dd s
		-\sum_{\sm\in\{\pm\}}\int_{0}^{t}\intOsm \sum_{i=1}^nf_{i}(u)\log u_{i} \dxds&
		\\+\int_{0}^{t}\int_{\Gamma}\sum_{\sm\in\{\pm\}}\sum_{i=1}^n r_{i}^\sm(\bfu)\log{u_{i}^\sm}\dSds&
		\le H(u_0).
	\end{align*}
	This completes the proof of Theorem~\ref{thm:ex}.
\end{proof}

\appendix 

\section{Appendix}

\subsection{Reflection at Lipschitz boundary}\label{ssec:reflection}
Below, we recall a standard local reflection technique at Lipschitz boundaries. For further background and references, see e.g.\ the PhD thesis~\cite{nittka2016elliptic}.
 Let $\Om'\subset\mathbb{R}^d$ be a bounded Lipschitz domain and
let $\Gamma'=\pa\Om'$.
For every $z\in\Gamma'$, there exists an orthogonal matrix $O$, a radius $r > 0$, and a Lipschitz continuous function $\eta: \mathbb{R}^{d-1} \to \mathbb{R}$ such that with
\[
V := \left\{ \left( y, \eta(y) + s \right) : y \in B^{d-1}_r(0)\subset \mathbb{R}^{d-1},\ s \in (-r, r) \right\}
\]
we have
\[
O(\Om_{\sigma}-z) \cap V = \left\{ \left( y, \eta(y) + s \right) : y \in B^{d-1}_r(0),\ s \in (-r, 0) \right\}. 
\]
If $\Om'$ has $C^1$ boundary, then $\eta$ can be choosen to be $C^1$.

To simplify notation, we only discuss the case $O = I$ and $z = 0$. The adaptations needed to treat general $O,z$ are elementary.

Define $\BIL(y, s) := (y, \eta(y) + s)$ for $y \in \mathbb{R}^{d-1}$, $s \in \mathbb{R}$. Then $\BIL$ is a bi-Lipschitz homeomorphism from $B^{d-1}_r(0) \times (-r, r)$ to $V$ with derivative
\[
D\BIL(y, s) = \begin{pmatrix} I & 0 \\ D\eta(y) & 1 \end{pmatrix}, \quad \text{and} \quad D\BIL(y, s)^{-1} = \begin{pmatrix} I & 0 \\ -D\eta(y) & 1 \end{pmatrix}
\]
Importantly,
\[\det D\BIL=1.\]
Define the reflection $\RFL : V \to V$ at $\Gamma'\cap V$ by $\RFL(\BIL(y, s)) := \BIL(y, -s)$.
Then, (a.e., in Lipschitz case)
\[
D\RFL(\BIL(y, s)) = D\BIL(y, -s) \begin{pmatrix} I & 0 \\ 0 & -1 \end{pmatrix} D\BIL(y, s)^{-1} = \begin{pmatrix} I & 0 \\ 2D\eta(y) & -1 \end{pmatrix}.
\]
In particular,
\[\det D\RFL=-1.\]
Note that $\RFL(\RFL(x)) = x$, $D\RFL$ is bounded. Moreover, $D\RFL(y, s)$ does not depend on $s$, whence $D\RFL(\RFL(x)) = D\RFL(x)$, and $D\RFL(x)^{-1} = D\RFL(x)$.

\section*{Acknowledgements}
The authors would like to thank Jan--Frederik Pietschmann for several helpful discussions on this topic. The first author further thanks Virginie Ehrlacher and Annamaria Massimini for valuable discussions on the paper~\cite{CCCE_2024_moving-interface}.

\printbibliography[heading=bibintoc]

\end{document}